\documentclass[11pt,reqno]{amsart}

\usepackage[utf8]{inputenc}
\usepackage[margin=1.25in]{geometry}
\usepackage{hyperref}
\usepackage{appendix}
\usepackage{amsfonts}
\usepackage{amsthm}
\usepackage{amssymb}
\usepackage{stmaryrd} 
\usepackage{amsmath}
\usepackage{amsthm}
\usepackage[dvipsnames]{xcolor}
\usepackage{mathrsfs}
\usepackage{lipsum} 

\theoremstyle{plain}
\newtheorem{theorem}{Theorem}[section]
\newtheorem{corollary}[theorem]{Corollary}
\newtheorem{lemma}[theorem]{Lemma}
\newtheorem{Proposition}[theorem]{Proposition}

\newtheorem{Definition}[theorem]{Definition}

\newtheorem{fact}[theorem]{Fact}

\newtheorem{notation}[theorem]{Notation}

\theoremstyle{remark}

\newtheorem{remark}[theorem]{Remark}

\usepackage{biblatex} 
\numberwithin{equation}{section}
\title[large deviation for the rank of symmetric random matrices]{On the rank of a random symmetric matrix in the large deviation regime}

\author{Yi HAN}
\address{Department of Mathematics, Massachusetts Institute of Technology, Cambridge, MA
}
\email{hanyi16@mit.edu}

\begin{document}

\begin{abstract}
Let $A$ be an $n\times n$ random symmetric matrix with independent identically distributed subgaussian entries of unit variance. We prove the following large deviation inequality for the rank of $A$: for all $1\leq k\leq c\sqrt{n}$,
$$
\mathbb{P}(\operatorname{Rank}(A)\geq n-k)\geq 1-\exp(-c'kn),
$$ for some fixed constants $c,c'>0$. 
A similar large deviation inequality is proven for the rank of the adjacency matrix of dense Erdős–Rényi graphs. This corank estimate enhances the recent breakthrough of Campos, Jensen, Michelen and Sahasrabudhe that the singularity probability of a random symmetric matrix is exponentially small, and echoes a large deviation inequality of Mark Rudelson for the rank of a random matrix with independent entries.
\end{abstract}
\maketitle

\section{Introduction}

In this work we study the rank of a random symmetric matrix $A=(a_{ij})_{1\leq i,j\leq n}$ where $a_{ij}=a_{ji}$ are mean zero subgaussian random variables of unit variance and $(a_{ij})_{1\leq i\leq j\leq n}$ are independent and identically distributed. We also study the rank of the adjacency matrix of an Erdös-Rényi graph $G(n,p)$ on $n$ vertices, where $p\in(0,1)$ is a fixed constant.

It has already been a challenging problem to prove that $A$, or $G(n,p)$ has full rank with high probability: this is because the symmetry assumption on $A$ imposes considerable complication to the analysis. (In contrast, analogous problems on the singularity probability of a random matrix with \textit{independent} entries have been investigated with much higher precision, and we refer to the survey \cite{MR4680362} for a comprehensive literature review.) Consider the special case where $A$ has uniform $\{\pm 1\}$ or $\{0,1\}$ entries, then the first proof of non-singularity of $A$ was obtained by Costello, Tao and Vu \cite{costello2006random}, who proved that $\mathbb{P}(\det A=0)=o(1)$. The singularity probability of $A$ was subsequently improved in  \cite{nguyen2012inverse}, \cite{vershynin2014invertibility}, \cite{ferber2019singularity}, \cite{campos2021singularity}, and \cite{campos2022singularity}, arriving at the bound $\mathbb{P}(\det A=0)=e^{-\Omega(\sqrt{n\log n})}$. Finally, in the breakthrough work \cite{campos2025singularity} it was proven that the singularity probability is exponentially small, i.e. $\mathbb{P}(\det A=0)=e^{-\Omega(n)}$. The exponential bound is optimal up to the exact leading coefficient in $n$, since with probability at least $2^{-n}$, $A$ will have a zero column and hence $\det A=0$.
The exact leading-order asymptotic for the singularity probability was conjectured to be $\mathbb{P}(\det A=0)=n^22^{-n}(1+o(1))$(see \cite{costello2006random},\cite{vu2021recent}), but no proof has been given.

For the rank of the random graph $G(n,p)$, significant progress has also been made in recent years, particularly in the sparse case where $p\to 0$. The works \cite{costello2008rank}, 
\cite{basak2021sharp}, \cite{glasgow2023exact}, \cite{ferber2023singularity} reveal a transition of invertibility of the adjacency matrix $G(n,p)$ around the critical threshold $p\sim \log n/n$, but the $k$-core of the matrix remains invertible even when $p=\lambda/n$ for some $\lambda>0$ and $k\geq 3$. The rank distribution of $A$ on a finite field $\mathbb{F}_p$ has also been studied in \cite{maples2010singularity} and \cite{ferber2023random} for $p\leq\exp(o(n^{1/8}))$ and $p\leq\exp(O(n^{1/4}))$ respectively. 

A very natural extension of the singularity problem is to study the corank of the random matrix $A$, namely $k=n-\operatorname{Rank}(A)$, in the regime where $k$ increases with $n$. It is natural to guess that $\mathbb{P}(\operatorname{Rank}(A)\leq n-k)$ should decay faster than $e^{-tn}$ for any $t>0$ when $k$ grows in $n$, and after a moment of thought one would also guess that $\mathbb{P}(\operatorname{Rank}(A)\leq n-k)\leq e^{-c'nk}$ for some $c'>0$, since one may expect that the leading contribution to the probability for the event $\operatorname{Rank}(A)\leq n-k$ is caused by some $k$ linearly dependent rows, which has probability roughly $e^{-c'nk}$. In this paper we prove that all these heuristic predictions are true, at least when the corank is not too large with $1\leq k\leq c\sqrt{n}$ for some fixed $c>0$.

A similar question on the corank of a random matrix with i.i.d. entries has been studied in much greater depth. Let $B=(b_{ij})$ be an $n\times n$ random matrix with i.i.d. entries of uniform $\{\pm 1\}$ distribution. The singularity probability of $B$ was also conjectured to be that $\mathbb{P}(\det B=0)=n^22^{-n}(1+o(1))$ and Tikhomirov proved the best current estimate $\mathbb{P}(\det B=0)=(\frac{1}{2}+o(1))^n$ in \cite{tikhomirov2020singularity}, giving the asymptotically sharp exponent. The singularity probability of $B$ was previously studied in \cite{komlos1967determinant}, \cite{kahn1995probability}, \cite{tao2007singularity}, \cite{bourgain2010singularity}. When the corank of $B$ grows with $n$, Kahn, Komlos, and Szemeredi proved in \cite{kahn1995probability} that $\mathbb{P}(\operatorname{Rank}B\leq n-k)\leq f(k)^n$ for a function $f(k)$ satisfying $f(k)\to 0$ as $k\to\infty$. Building on techniques in \cite{tikhomirov2020singularity}, Jain, Sah and Sawhney proved in \cite{jain2021singularity} and \cite{jain2022rank} that when $B$ has Bernoulli($p$)- entries, then $\mathbb{P}(\operatorname{Rank}B\leq n-k)\leq (1-p+\epsilon)^{nk}$ for any fixed $k$ when $n$ is sufficiently large, which gives the optimal exponent for the large deviations probability with a fixed corank $k$. When $k$ is growing in $n$, the technique in \cite{jain2022rank} does not seem to work. Mark Rudelson proved in \cite{rudelson2024large} the first large deviation inequality for growing $k$, which states that \begin{equation}\label{rudelsonstate}
\mathbb{P}(\operatorname{Rank}B\leq n-k)\leq e^{-c'kn},\quad \forall 1\leq k\leq c\sqrt{n}
\end{equation} for some $c,c'>0$. The proof of \cite{rudelson2024large} works for $B$ with a general subgaussian distribution.

For a random symmetric matrix $A$, even the analogue statement in Kahn, Komlos, and Szemeredi \cite{kahn1995probability}, which can be formulated as $\mathbb{P}(\operatorname{Rank}A\leq n-k)\leq f(k)^n$ for a function $f(k)\to 0$ as $k\to\infty$, was not previously proven. 
Our main result is that the random symmetric matrix $A$ satisfies a similar version of large deviations inequality as in \eqref{rudelsonstate} \cite{rudelson2024large}:

\begin{theorem}\label{maintheorem1.1}
Let $A=(a_{ij})_{1\leq i,j\leq n}$ be an $n\times n$ random symmetric matrix with $a_{ij}=a_{ji}$ and where $(a_{ij})_{1\leq i\leq j\leq n}$ are independent copies of $\zeta$, which is a random variable of mean 0, variance 1 with a subgaussian tail. Then there exist constants $c,c'>0$ depending only on $\zeta$ such that
\begin{equation}\label{largedeviationstateindep}
\mathbb{P}(\operatorname{Rank}A\leq n-k)\leq e^{-c'kn},\quad \forall 1\leq k\leq c\sqrt{n}.
\end{equation} 
\end{theorem}

\begin{remark}
    After the completion of this paper, \cite{hunter2025random} proved that the large deviation estimate \eqref{rudelsonstate} for the i.i.d. matrix $B$ actually holds for all $1\leq k\leq n$, extending the result from \cite{rudelson2024large}. It is very reasonable to expect that for the random symmetric matrix $A$, the estimate \eqref{largedeviationstateindep} also holds for all $1\leq k\leq n$. However, the proof ideas in \cite{hunter2025random} do not appear to be generalizable to the random symmetric model $A$.
\end{remark}

The same large deviation inequality can be proven for the Erdös-Rényi graph $G(n,p)$:

\begin{theorem}\label{corollaryrandomgraph}
    Fix $p\in(0,1)$ and let $G$ denote the adjacency matrix of $G(n,p)$, a random Erdös-Rényi graph on $n$ vertices with parameter $p$. That is, $G=(g_{ij})_{1\leq i,j\leq n}$ with independently $g_{ij}=g_{ji}\sim\operatorname{Ber}(p)$ for $i\neq j$ and $g_{ii}=0$. Then there exists $c,c'>0$ and $C_p>0$ depending only on $p$ such that\begin{equation}\label{largerandomgraphs}
\mathbb{P}(\operatorname{Rank} G\leq n-k)\leq e^{-c'kn},\quad \forall C_p\leq k\leq c\sqrt{n}.
\end{equation} 
In the special case $p=\frac{1}{2}$, we prove \eqref{largerandomgraphs} for all corank $1\leq k\leq c\sqrt{n}$.
\end{theorem}

The inequality \eqref{largerandomgraphs} should hold for any $p\in(0,1)$ and any $1\leq k\leq c\sqrt{n}$ but may require a much longer argument to prove; see remark \ref{remakrp1/2}. 

The large deviation probability for the rank of a random symmetric matrix is also connected to the problem of bounding the number of integral graphs (i.e., graphs with an integral spectrum where all the eigenvalues of its adjacency matrix are integers). These integral graphs are known to play an important role for quantum networks that support a perfect state transfer \cite{facer2008quantum}, \cite{saxena2007parameters}, \cite{christandl2005perfect}, \cite{christandl2004perfect}. The first work obtaining an upper bound for the number of integral graphs was in Ahmadi, Blake, Alon and Shparlinski \cite{ahmadi2009graphs}, who proved that the proportion of integral graphs on $n$ vertices is at most $2^{-\Omega(n)}$.  This result was then improved by Costello and  Williams in \cite{MR3452749} that the proportion of integral graphs is at most $2^{-\Omega(n^{3/2})}$, and they conjectured that the proportion of integral graphs should be at most $2^{-\Omega(n^2)}$. In \cite{MR3452749}, Question 1, the authors proposed that studying  $\mathbb{P}(\operatorname{Rank}G\leq n-k)$ for $k$ growing in $n$ could be a first step towards resolving this conjecture. We settle this question in Theorem \ref{corollaryrandomgraph} with a sharp probability estimate up to the multiplicative constant in front of $kn$, for all $C_p\leq k\leq c\sqrt{n}$. This result gives an alternative proof to the $2^{-\Omega(n^{3/2})}$ fraction of integral graphs, since at least one integral eigenvalue of $G$ with integral spectrum should have multiplicity $\Omega(\sqrt{n})$. We hope that our large deviation inequality can lead to progress towards the conjectural $2^{-\Omega(n^2)}$ proportion of integral graphs in a future work.

\subsection{Main ideas of the proof}

The key technical step to the proof of Theorem \ref{maintheorem1.1} is to show that, when $\operatorname{Rank}(A)\leq n-k$, then with probability $1-e^{-\Omega(nk)}$, we can extract a linear subspace of $\ker A$ of dimension $k/4$ consisting entirely of unit vectors that have no rigid arithmetic structures (in the sense that the essential LCD is large, see Definition \ref{essentiallcdone}). Then we apply this result to some principal minors $A_{n-k}$ of the matrix $A$, say 
$A=\begin{bmatrix}
    A_{n-k}&X\\X^T&D
\end{bmatrix}$, and use the independence of the columns of $X$ with $A_{n-k}$, and the fact that the columns of $X$ must lie in the linear span of the columns of $A_{n-k}$, to conclude the proof.

For a random symmetric matrix, at least two sorts of technical difficulties will arise in this large deviation problem. The first sort of difficulty arises when we consider the large deviation probability of a random matrix using techniques for the singularity probability of random matrices (see survey \cite{MR4680362}). Among those difficulties, the first issue is we have no control of the operator norm $\|A\|$ in large deviation regime, but we can still control $\|A\|_{HS}$. The second issue is that we cannot show $\ker A$ consists entirely of incompressible vectors (see Definition \eqref{wherearecompressiblevectors}), or entirely of vectors with small essential LCD, with a very sharp probability estimate at $1-e^{-\Omega(nk)}$ with $k$ growing in $n$. We will show, at best, that we can extract linear subspaces of $\ker A$ of dimensions $k/4$, say, satisfying these conditions. Although we have lost by a factor $3/4$, this does not qualitatively ruin our large deviations estimate. The third difficulty is that when we use inverse Littlewood-Offord theorems (Theorem \ref{littlewoodincomp}) to rule out incompressible vectors with a small essential LCD, we have to consider a tuple of $l$ almost orthogonal vectors each having essential LCD in a range $[c\sqrt{n},\exp(Cn/l)]$ but their essential LCD may differ significantly from one another. Then it may appear to be difficult how to control the essential LCD of this whole $l$-tuple of vectors. We will use a carefully defined selection lemma, Lemma \ref{lemma2.4}. This lemma leads to the restriction $k\leq c\sqrt{n}$.
The method to solve all these three technical difficulties are inspired by the work of Rudelson \cite{rudelson2024large} on large deviations of the rank of an i.i.d. matrix.

The second sort of difficulty arises when we consider specifically a random symmetric matrix, where it was proven only very recently by Campos, Jensen, Michelen and Sahasrabudhe \cite{campos2025singularity} that its singularity probability is exponentially small. We will use a generalized version of the conditioned inverse Littlewood-Offord theorem in \cite{campos2025singularity} so that we can maintain a control on the least singular value of submatrices of $A$ in addition to a (standard) inverse Littlewood-Offord theorem. As we adopt the ideas of \cite{campos2025singularity} to the large deviation setting, the first technical challenge is to find a linear subspace of $\ker A$ where every unit vector satisfies a version of no-gaps delocalization (see Section \ref{nogapsdev} for details) before we can apply any argument in the spirit of \cite{campos2025singularity}. 
The resolution of this problem in the large deviation regime, in some sense, uses all the technical contributions in the spirit of \cite{campos2025singularity}. A more specific proof outline is given in the beginning of Section \ref{nogapsdev}. The last remaining challenge is that, while we adopt the idea of inversion of randomness of \cite{campos2025singularity} to randomly generate vectors from a lattice to approximate vectors in $\ker A$ and use this to compensate for the loss of randomness due to dependence in $A$, via the conditional inverse Littlewood-Offord theorem (Theorem \ref{compressibletheorem12.234567new}), we cannot use a completely random method to generate all candidate vectors simultaneously. Although a completely random generating method works well for the singularity problem of \cite{campos2025singularity}, in the large deviation range we cannot have a good control of essential LCD via this method. Instead, for a value $d=\Theta(n)$, we use a completely random method to generate the last $n-d$ coordinates of these vectors, but we use Lemma \ref{lemma2.4} to fix the first $d$ coordinates of these tuples of $l$ vectors. This also allows us to implement the program of \cite{campos2025singularity} without having an available bound on $\|A\|$ in the large deviation regime.

This paper is organized as follows. In Section \ref{technicaltools} we introduce several technical preparations, in particular Lemma
\ref{lemma2.4}. In Section \ref{compressiblevectors3} we prove that we can extract from $\ker A$ a subspace of incompressible vectors via standard techniques. In Section \ref{nogapsdev} we strengthen this result to prove that we can find a further subspace whose unit eigenvectors are delocalized in the no-gap sense: several important technical lemmas are introduced here as well. In Section \ref{thresholdcardinality}, we use a similar argument to prove that we can find a further subspace of $\ker A$ where all unit vectors have a large essential LCD. Finally, in Section \ref{ranklastpart} we prove the two main theorems about the rank of random matrix. Particular effort is needed for the random graph model for the case $k=1$. Section \ref{proofinverse} contains the proof of several technical results. 

\subsection{A list of notations}

Throughout the paper, we will use the following list of notations.

For a vector $x\in\mathbb{R}^n$ we use $\|x\|_2$ for the Euclidean norm of $x$: $\|x\|_2=\sqrt{\sum_{i=1}^n x_i^2}$. We also write $\mathbb{S}^{n-1}$ for the unit sphere in $\mathbb{R}^n$. 

For $x\in\mathbb{R}_+$, $\lfloor x\rfloor$ denotes the largest integer smaller or equal to $x$, and $\lceil x\rceil$ denotes the smallest integer strictly larger than $x$.     

We denote by $A\in\operatorname{Sym}_n(\zeta)$ for a random variable $\zeta$ to mean that $A=(a_{ij})_{1\leq i,j\leq n}$ is a random matrix where $a_{ij}=a_{ji}$ and $\{a_{ij}\}_{1\leq i\leq j\leq n}$ are i.i.d. copies of the random variable $\zeta$.

For a mean 0, variance 1 random variable $\zeta$, we define the subgaussian moment of $\zeta$ via 
$$
\|\zeta\|_{\psi_2}:=\sup_{p\geq 1} p^{-1/2}
(\mathbb{E}|\zeta|^p)^{1/p}.$$
We say $\zeta$ is subgaussian if $\|\zeta\|_{\psi_2}<\infty$, 
and we denote by $\Gamma_B$ as the collection of all random variables $\zeta$ with $\|\zeta\|_{\psi_2}\leq B$.

Let $V$ be an $m\times l$ matrix. Then its singular values are denoted either by 
$s_1(V)\geq s_2(V)\geq\cdots\geq s_m(V)\geq 0$, or otherwise denoted by 
$\sigma_1(V)\geq \sigma_2(V)\geq\cdots\geq \sigma_m(V)\geq 0.$

We define the operator norm of $V=(v_{ij})$ via 
$$
\|V\|=\sup_{x\in\mathbb{S}^{l-1}}\|Vx\|_2.
$$ The Hilbert-Schmidt norm of $V$ is then defined as
$$
\|V\|_{HS}=(\sum_{i=1}^m\sum_{j=1}^l v_{ij}^2)^{\frac{1}{2}}.
$$

We shall also frequently use the following notion of restricted vectors.

For two integers $m_1\leq m_2\leq n$, we denote by $[m_1,m_2]$ the set of integers $[m_1,m_2]\cap\mathbb{Z}$. We also abbreviate $[m]=[1,m]\cap\mathbb{Z}$.

For a vector $v\in\mathbb{R}^n$ we denote by $v_{[m_1,m_2]}$ the restriction of $v$ to the coordinates indexed by $[m_1,m_2]$. More generally, for a subset $I\subset[n]$ we denote by $v_I$ the restriction of $v$ onto coordinates indexed by $I$.

For a random vector $X\in\mathbb{R}^m$, we denote by $\mathcal{L}(X,t)$ its Levy concentration function 
$$
\mathcal{L}(X,t)=\sup_{y\in\mathbb{R}^m}\mathbb{P}
(\|X-y\|_2\leq t).$$

\section{Technical preparations}\label{technicaltools}

The first technical issue in this paper is that while one may wish to condition on the event $\|A\|_{op}=O(\sqrt{n})$, this event may only hold with probability $1-\exp(-cn)$ rather than with probability $1-\exp(-ckn)$, falling short of the large deviation regime in this paper. Fortunately, we have much better control of $\|A\|_{HS}$:

\begin{lemma}\label{lemma2.1} Let $A\sim\operatorname{Sym}_n(\zeta)$ with $\zeta\in\Gamma_B$. Then we can find constants $K$ and $c_{\ref{lemma2.1}}>0$ depending only on $B$ such that 
$$
\mathbb{P}(\|A\|_{HS}\geq 2Kn)\leq\exp(-c_{\ref{lemma2.1}}n^2).
$$
\end{lemma}

Conditioning on $\|A\|_{HS}=O(n)$, we will then use a random rounding method of Livshyts \cite{livshyts2021smallest} to do approximations even if $\|A\|=O(\sqrt{n})$ does not hold.

We also need a tensorization lemma for independent random variables that are not identically distributed. The version for i.i.d. random variables can be found in \cite{rudelson2008littlewood}, \cite{rudelson2024large}.

\begin{lemma}\label{tensorizationlemma} Let $l\in\mathbb{N}_+$, $M_1,\cdots,M_l>0$ and $Y_1,\cdots, Y_l\geq 0$ be independent random variables such that $\mathbb{P}(Y_j\leq s)\leq (M_js)^m$ for any $s\geq s_0$. Then we have, for some $C>0$,
$$
\mathbb{P}(\sum_{j=1}^lY_j\leq lt)\leq C^{lm}t^{lm}(\prod_{j=1}^l M_j)^m.
$$
    
\end{lemma}

\begin{proof}
    For any $t\geq s_0$ we apply Markov's inequality to get
    $$
\mathbb{P}(\sum_{j=1}^lY_j\leq lt)\leq\mathbb{E}[\exp(ml-\frac{m}{t}\sum_{j=1}^lY_j)]\leq \exp(ml)\prod_{j=1}^l\mathbb{E}\exp(-\frac{m}{t}Y_j),
    $$
    and where for each $j$,
    $$\begin{aligned}
&\mathbb{E}\exp(-\frac{m}{t}Y_j)=\int_0^\infty e^{-u}\mathbb{P}(Y_j\leq \frac{t}{m}u)du\\&\leq\int_0^m   e^{-u}\mathbb{P}(Y_j\leq t)du
+\int_m^\infty e^{-u} \mathbb{P}(Y_j\leq \frac{t}{m}u)du\\&\leq (M_jt)^m+\int_m^\infty e^{-u}(\frac{M_jt}{m}u)^mdu\\&\leq (M_jt)^m(1+\frac{\Gamma(m+1)}{m^m})\leq (CM_jt)^m.
    \end{aligned}$$
    Combining the above two estimates completes the proof.
\end{proof}

We also use a well-known fact on the number of integer points with bounded radius:
\begin{lemma}\label{numberofintegralppoints}
    For any given $R>0$, let $B(0,R)$ be the unit ball in $\mathbb{R}^n$ of radius $R$. Then 
    $$
|\mathbb{Z}^n\cap B(0,R)|\leq (2+\frac{CR}{\sqrt{n}})^n
    $$ for some universal $C>0$.
\end{lemma}

\subsection{On almost orthogonal vectors}
A technical difficulty in establishing large deviation estimates is that we need to consider a large family of vectors simultaneously in the kernel of $A$, and we need a method to select a minimal almost orthogonal pair of vectors from a given family of deterministic vectors. We will use the following selection method from \cite{rudelson2024large}.
\begin{Definition}\label{definitionalmostorthogonals}
    Fix $\nu\in(0,1)$. We call an $l$-tuple of vectors $(v_1,\cdots,v_l)\in\mathbb{R}^n\setminus\{0\}$ $\nu$-almost orthogonal if the matrix $W$ of size $n\times l$ with columns $(\frac{v_1}{\|v_1\|_2},\cdots,\frac{v_l}{\|v_l\|_2})$ satisfies 
    $$ 1-\nu\leq s_l(W)\leq s_1(W)\leq 1+\nu.
    $$ In the rest of this paper, we will use ``$\nu$-orthogonal'' as an abbreviation for ``$\nu$-almost orthogonal''.
\end{Definition}
A crude yet convenient criterion for an $\nu$-orthogonal pair is the following

\begin{lemma}\label{lemmaorthogonalchoices}(\cite{rudelson2024large}, Lemma 3.2)
    Fix $\nu\in[0,\frac{1}{4}]$ and an $l$-tuple of vectors $(v_1,\cdots,v_l)\in\mathbb{R}^n\setminus\{0\}$ such that 
    $$
\|P_{\operatorname{span}(v_1,\cdots,v_j)}v_{j+1}\|_2\leq\frac{\nu}{\sqrt{l}}\|v_{j+1}\|_2\quad\text{for all }j\in[l-1].
    $$ Then we have that $(v_1,\cdots,v_l)\in\mathbb{R}^n$ is a $(2\nu)$-almost orthogonal system. Denoting by $V$ the $n\times l$ matrix with columns $v_1,\cdots,v_l$, then we have
    $$
\operatorname{det}^{1/2}(V^TV)\geq 2^{-l}\prod_{j=1}^l \|v_j\|_2.
    $$
\end{lemma}

Next, we have a lemma (yielding a two-case alternative) on the selection of an almost orthogonal pair with certain minimality conditions. The original version of this lemma was proven in \cite{rudelson2024large}, Lemma 3.3 but here we need a variant for which the vectors are almost orthogonal when restricted to a fixed subinterval of $[n]$. 

For a vector $v\in\mathbb{R}^n$ and an interval $I\subset[n]$, we denote by $v_I$ the restriction of $v$ on the columns indexed by $I$ and denote by $\mathbb{R}^I$ the subspace of $\mathbb{R}^n$ with coordinates indexed by $I$.

\begin{lemma}\label{lemma2.4} Let $E\subset\mathbb{R}^n$ be a linear subspace of dimension $k$. Fix an interval $I\subset[n]$ with $|I|\geq k$ and let $W_I\subset\mathbb{R}^n\setminus\{0\}$ be a closed subset such that $\{v_I:v\in W_I\}$ is also a closed set in $\mathbb{R}^I\setminus\{0\}$. Then at least one of the following holds:
\begin{enumerate}
    \item  We can find vectors $v_1,\cdots,v_l\in E\cap W_I$ satisfying
    \begin{enumerate}
        \item The $l$-tuple $((v_1)_I,\cdots,(v_l)_I)$ is $\frac{1}{8}$-almost orthogonal in $\mathbb{R}^I$;
        \item  Given any $\theta\in\mathbb{R}^l$ with $\|\theta\|_2\leq\frac{1}{20\sqrt{l}}$, we have
        $$
\sum_{i=1}^l \theta_iv_i\notin W_I.
        $$
    \end{enumerate}
    \item We can find a subspace $F\subset E$ of dimension $k-l$ such that $F\cap W_I=\emptyset$.
\end{enumerate}
    
\end{lemma}

The statement of this lemma is somewhat abstract. We now discuss how we will apply this lemma in forthcoming proofs. We always take $I\subset[n]$ to be an interval of length $\Theta(n)$, and in Section \ref{thresholdcardinality} we actually take $c_0^2n/4\leq |I|\leq c_0^2n$ for a very small $c_0>0$.
We take the linear subspace $E$ to be either the kernel of $A$ (where we assume $A$ is an $(n-k)\times (n-k)$ symmetric matrix), or $E$ is certain linear subspace of the kernel of $A$. 

In Section \ref{nogapsdev}, we let $W_I$ be the set of vectors that are supported only on some $(1-c)$ fraction of the $n$ coordinates having a non-negligible fraction of their mass on the interval $I$. In Section \ref{thresholdcardinality}, the subset $W_I$ is chosen as follows: for a given $I\subset[n]$ and $\tau>0$, 
$$
W_I:=\{v\in\mathbb{R}^n: \frac{\tau}{8}\sqrt{n}\leq \|v_I\|_2\leq\exp(\frac{\rho n}{4L^2}),\quad \text{dist}(v_I,\mathbb{Z}^I)\leq\rho\sqrt{n}
\}.
$$ That is, $W_I$ consists of vectors $v$ whose restriction $v_I$ is incompressible but has essential LCD smaller than $\exp(\Theta(n/k))$, and we normalize the length of $\|v_I\|_2$.

\begin{proof}
    We use a similar argument as in \cite{rudelson2024large}, Lemma 3.3, but here all the orthogonality relations are taken when restricted to the interval $I$.

    Assume that $E\cap W_I\neq\emptyset$ (otherwise case (2) holds for any subspace $F\subset E$) and take $v_1\in E\cap W_I$ to be the vector in $E\cap W_I$ such that $\|(v_1)_I\|_2$ is the smallest. We also take $v_0=0$ for convenience.

We now define a family of subspaces $H_0\supset H_1\supset H_2\supset H_3\cdots$ where we take $H_0=E\cap W_I$, and choose $v_1\in H_0$  as previously stated. Suppose that for a given $j$, we have chosen $v_1\in H_0,v_2\in H_1,\cdots,v_j\in H_{j-1}$, then we define the $j$-th space $H_j$ via
$$
H_j=\{v\in H_{j-1}: \|\operatorname{P}_{\text{span}((v_1)_I,\cdots,(v_{j})_I)}v_I\|_2\leq \frac{1}{16\sqrt{l}}\|(v_j)_I\|_2\}.
$$ The subspace $H_j$ can alternatively be written as 
$$
H_j=\{v\in E\cap W_I:\|\operatorname{P}_{\text{span}((v_1)_I,\cdots,(v_{s})_I)}v_I\|_2\leq \frac{1}{16\sqrt{l}}\|(v_s)_I\|_2\text{ for each }s=1,\cdots,j\}.
$$
Then whenever $H_j\neq\emptyset$, we select $v_{j+1}\in H_j$ to be the vector in $H_j$ which minimizes $\|(v_{j+1})_I\|_2$. Then by our method of selection, $\|(v_{j+1})_I\|_2\geq \|(v_j)_I\|_2\geq\cdots\geq \|(v_1)_I\|_2$.

Suppose that the above selection procedure proceeds in $l$ steps and we have constructed $v_1,\cdots,v_l$. From the construction for each $j\in[l]$,
$$
\|P_{\operatorname{span}((v_1)_I,\cdots,(v_{j-1})_I)}(v_j)_I\|_2\leq \frac{1}{16\sqrt{l}}\|(v_{j-1})_I\|_2\leq\frac{1}{16\sqrt{l}}\|(v_j)_I\|_2.
$$ Then the claim 1(a) holds thanks to Lemma \ref{lemmaorthogonalchoices}. To check claim (1b), suppose that there exists some $\theta\in\mathbb{R}^{l}$ with $\|\theta\|_2\leq\frac{1}{20\sqrt{l}}$ with \begin{equation}
    \sum_{i=1}^{l}\theta_iv_i\in W_I.
\end{equation}
    For each $j=1,\cdots,l$ denote by $V^j$ the $|I|\times j$ matrix with columns $(v_1)_I,\cdots,(v_j)_I$. Then the verified condition 1(a) implies $\|V^j\|\leq\frac{9}{8}\|(v_j)_I\|_2$. Using $v_{j+1}\in H_j$, we get
$$\begin{aligned}
&\|P_{\operatorname{Span}((v_1)_I,\cdots,(v_j)_I)}(\sum_{i=1}^{l}\theta_i(v_i)_I)\|_2\leq\|\sum_{i=1}^j\theta_i(v_i)_I\|_2+\sum_{s=j+1}^l|\theta_{s}|\cdot\|P_{\operatorname{span}(v_1)_I,\cdots,(v_j)_I}(v_{s})_I\|_2 \\&\leq\|V^j\|\cdot\|\theta\|_2+(\sum_{s=1}^l|\theta_s|)\cdot\frac{1}{16\sqrt{l}}\|(v_j)_I\|_2\\&\leq(\frac{9}{8}+\frac{1}{16})\|\theta\|_2\cdot\|(v_j)_I\|_2<\frac{1}{16\sqrt{l}}\|(v_j)_I\|_2
\end{aligned}$$ where the second line uses Cauchy-Schwartz inequality and the last step uses $\|\theta\|_2\leq\frac{1}{20\sqrt{l}}$. As the inequalities hold for each $j$, we have $\sum_{i=1}^l\theta_i(v_i)_I\in H_l$. Then by definition of $H_j$ we must have $\|\sum_{i=1}^l\theta_i(v_i)_I\|_2\geq \|(v_l)_I\|$. But we also have
$$
\|\sum_{i=1}^{l}\theta_i(v_i)_I\|_2\leq\|V^{l}\|\cdot\|\theta\|_2\leq\frac{9}{8}\|(v_{l})_I\|_2\cdot\|\theta\|_2\leq \frac{1}{16\sqrt{l}}\|(v_{l})_I\|_2,
$$ leading to a contradiction. Thus 1(b) is also checked.

On the other hand, if $H_{l_0}=\emptyset$ for some $l_0\leq l$, then let $F\subset E$ denote the subspace $$F:=\{v\in E:\langle v_I,(v_s)_I\rangle=0\text{ for each }s=1,2,\cdots,l_0\}.$$ Then $F$ is a linear subspace of $E$, $F\cap W_I=\emptyset$ and $\dim F =\dim E-l_0=k-l_0$. The last equality follows from linear algebra, since $F\cap \operatorname{span}(v_1,\cdots,v_{l_0})=0$, and $E=F+\operatorname{span}(v_1,\cdots,v_{l_0})$ as a sum of linear subspaces. This justifies the conclusion in case (2).    
\end{proof}

\subsection{Arithmetic structures and inverse Littlewood-Offord theorems}

We will introduce a family of notations from \cite{rudelson2008littlewood} and \cite{rudelson2016no}:

\begin{Definition}\label{wherearecompressiblevectors}
    Fix $\delta,\rho\in(0,1)$. Then a vector $x\in\mathbb{R}^n$ is called $\delta$-sparse if $|\text{supp}(x)|\leq\delta n$, and $x$ is called $(\delta,\rho)$-compressible, denoted by $x\in\operatorname{Comp}(\delta,\rho)$, if $x$ is within Euclidean distance  at most $\rho$ to a vector supported only on $\delta n$ coordinates. 
A vector $x$ is called $(\delta,\rho)$-incompressible, denoted $x\in \operatorname{Incomp}(\delta,\rho)$, if $x$ is not $(\delta,\rho)$-compressible. 

    For a given $d\leq n$ and a subset $I\subset[n]$, $|I|=d$, let $\mathbb{R}^I\subset\mathbb{R}^n$ be the vectors supported only on coordinates indexed in $I$. Then for any $x\in\mathbb{R}^I$ we define in exactly the same way the notion of $x\in\operatorname{Comp}(\delta,\rho)$ or $\operatorname{Incomp}(\delta,\rho)$, by interpreting $x$ as a vector in $\mathbb{R}^n$.  That is, we say $x\in\mathbb{R}^I$ is $\delta$-sparse if $|\operatorname{supp}(x)|\leq \delta n$.
\end{Definition}

This definition of (in)compressible vectors is slightly different from \cite{rudelson2008littlewood} in that we define $\operatorname{Incomp}(\delta,\rho)$ for all vectors in $\mathbb{R}^n$ instead of $\mathbb{S}^{n-1}$, and that we will use the same notation $\operatorname{Incomp}(\delta,\rho)$ while we keep implicit whether $x$ is supported on all of $\mathbb{R}^d$, or only supported on some subspace $\mathbb{R}^I\subset\mathbb{R}^d$.

The essential least common denominator (LCD) of vector pairs is defined as follows:
\begin{Definition}\label{definition4.33.4}
    For a $m\times n$ matrix $V$ and two constants $L_0>0,\alpha_0\in(0,1)$ we define 
    $$
D_{L_0,\alpha_0}(V):=\inf\left(\|\theta\|_2:\theta\in\mathbb{R}^m, \|V^T\theta\|_\mathbb{Z}\leq L_0\sqrt{\log_+\frac{\alpha_0\|V^T\theta\|_2}{L_0}}
\right).
    $$

If $E\subset\mathbb{R}^n$ is a linear subspace, then we can similarly define
$$
D_{L_0,\alpha_0}(E):=\inf\{y\in E:\operatorname{dist}(y,\mathbb{Z}^n)\leq L_0\sqrt{\log_+\frac{\alpha_0\|y\|_2}{L_0}}\}.
$$

\end{Definition}

We also recall a slightly different notion of essential LCD for one vector:
\begin{Definition}\label{essentiallcdone} For $v\in\mathbb{R}^n$ and $\rho,\gamma\in(0,1)$ we define the essential LCD $D_{\rho,\gamma}(v)$ as
$$
D_{\rho,\gamma}(v):=\inf\{\theta\geq 0:\|\theta v\|_\mathbb{Z}\leq\max(\gamma\|\theta v\|_2,\sqrt{\rho n})\},
$$
    where for $v\in\mathbb{R}^n$ we define 
    $$
\|v\|_\mathbb{Z}=\inf_{p\in\mathbb{Z}^n}\|v-p\|_2.
    $$
\end{Definition}

The version of inverse Littlewood-Offord inequality we use here can be found in \cite{rudelson2016no}:

\begin{theorem}\label{littlewoodincomp}
    Consider $\zeta\in\Gamma_B$ and a random vector $\zeta=(\zeta_1,\cdots,\zeta_n)$ of i.i.d. coordinates with distribution $\zeta$. Consider a matrix $V\in\mathbb{R}^{m\times n}$, then for any $L_0\geq\sqrt{m/(1-\mathcal{L}(1,\zeta))}$, any $\alpha_0>0$ and any $t>0$, we have
    $$
\mathcal{L}(V^T\zeta,t\sqrt{m})\leq \frac{(CL_0/(\alpha_0\sqrt{m}))^m}{\det(VV^T)^{1/2}}\left(t+\frac{\sqrt{m}}{D_{L_0,\alpha_0}(V)}\right)^m
$$ for a constant $C>0$ depending only on $B$.

We similarly have a version for projection onto a subspace: let $E\subset\mathbb{R}^n$ be a linear subspace with $\dim E=m$, and let $P_E$ be the orthogonal projection onto the subspace $E$, then for any $L_0\geq\sqrt{m/(1-\mathcal{L}(1,\zeta))}$, we have
$$
\mathcal{L}(P_E\zeta,t\sqrt{m})\leq (\frac{CL_0}{\alpha_0\sqrt{m}})^m(t+\frac{\sqrt{m}}{D_{L_0,\alpha_0}(E)})^m.
$$

\end{theorem}

We also use the following formulation that incompressible vectors have large LCD:
\begin{lemma}\label{lcdimcomps}
(\cite{rudelson2024large}, Lemma 3.11)
    For a given $s,\alpha_0\in(0,1)$, let $U$ be some $n\times l$ matrix with $U\mathbb{R}^l\cap\mathbb{S}^{n-1}\subset \operatorname{Incomp}(sn,\alpha_0)$. Then for any $\theta\in \mathbb{R}^l$ with $\|U\theta\|_2\leq \sqrt{sn}/2$, we have that for any $L_0>0$,
    $
\|U\theta\|_\mathbb{Z}\geq L_0\sqrt{\log_+\frac{\alpha_0\|U\theta\|_2}{L_0}}.
    $ We actually have that $\|U\theta\|_\mathbb{Z}\geq \alpha_0\|U\theta\|_2$.
\end{lemma}

\section{Ruling out compressible vectors}\label{compressiblevectors3}
In this section we first prove that the kernel of $A\sim\operatorname{Sym}_n(\zeta)$ does not contain too many compressible vectors with very high probability. Although it is well known (see for example \cite{vershynin2014invertibility}, Proposition 4.2) that with probability $1-e^{-\Omega(n)}$, all unit vectors in $\ker A$ are $(\delta,\rho)$-incompressible for some $\delta,\rho\in(0,1)$, the failure probability $e^{-\Omega(n)}$ cannot be made super-exponentially small (since when $A$ has a zero row and column, then $\ker A$ has a compressible vector). Thus this estimate is not sufficient for our purpose. To get around this, we adopt an idea from \cite{rudelson2024large}. The idea is to prove that with probability $1-e^{-\Omega(nk)}$, we can find a linear subspace of $\ker A$ of codimension at most $k/4$, such that in this subspace, all the unit vectors are incompressible. This is stated in the following.

\begin{Proposition}\label{compressibleprop} Let $A\sim\operatorname{Sym}_n(\zeta)$ with $\zeta\in\Gamma_B$ and fix $l\leq n$. Then we can find $c>0,\tau\in(0,1)$ depending only on $B$ such that, with probability at least $1-\exp(-cln)$, there exists a linear subspace $W_1$ of $\ker A$ of codimension $\frac{1}{4}l$ such that $W_1\cap (\operatorname{Comp}(\tau^2n,\tau^4)\cap\mathbb{S}^{n-1})=\emptyset$.
\end{Proposition}
For example, if $\dim\ker A\leq \frac{1}{4}l$ then the conclusion is vacuous. In the large deviation event where  $\dim\ker A\geq l$, then by this proposition, we can find $W_1\subset\ker A$ of dimension at least $\frac{3}{4}l$ containing no vectors from $\operatorname{Comp}(\tau^2n,\tau^4)\cap\mathbb{S}^{n-1}$, with probability $1-e^{-\Omega(ln)}$.

In the proof of Proposition \ref{compressibleprop}, we will take a further partitioning of the support of vectors $v\in\ker A$. This is because, 
in contrast to a square matrix $B$ with i.i.d. entries, the entries of $A\in\operatorname{Sym}_n(\zeta)$ are not independent and we need to consider certain submatrices of $A$ to restore independence. The first estimate is stated as follows:

\begin{lemma}\label{lemma3.1} Let $A\sim\operatorname{Sym}_n(\zeta)$ with $\zeta\in\Gamma_B$. Consider three intervals $I_1:=[1,\lfloor n/3\rfloor]$, $I_2:=[\lceil n/3\rceil,\lfloor 2n/3\rfloor]$ and $I_3:=[\lceil 2n/3\rceil, n]$. Let $l\leq n$ and $v_1,\cdots,v_l\in\mathbb{R}^{n}$ be an $l$-tuple of vectors such that for some $i\in\{1,2,3\}$, the $l$-tuple $(v_1)_{I_i},\cdots,(v_l)_{I_i}$ satisfies $$\frac{3}{2}\geq \|(v_j)_{I_i}\|_2\geq \frac{1}{4} \text{ for each } j\in[l]$$ and that the vector pair $((v_1)_{I_i},\cdots,(v_l)_{I_i})$ is $\frac{1}{2}$-almost orthogonal. Then we can find some $c_{\ref{lemma3.1}}>0$ depending only on $B$ such that \begin{equation}
    \mathbb{P}(\|Av_j\|_2\leq c_{\ref{lemma3.1}}\sqrt{n}\quad\text{for all }j\in[l])\leq \exp(-c_{\ref{lemma3.1}}ln).
\end{equation}
    
\end{lemma}

\begin{proof}
    We consider the case where $i=3$ as the case of $I_1,I_2$ is similar. Let $C$ be the submatrix $A_{[1,\lfloor 2n/3\rfloor]\times [\lceil 2n/3\rceil,n]}$ so that $C$ has i.i.d. coordinates of distribution $\zeta$, so we have the block form $A=\begin{bmatrix}E&C\\C^T&F\end{bmatrix}$ where $E$ and $F$ are square matrices. 

   We only use the first $\lfloor 2n/3\rfloor$ rows of $Av_j$ and condition on the realization $E$. Then the probability in question is bounded from above by, for a given $c>0$, 
   \begin{equation}\label{supssupssups}
\sup_{w_1,\cdots,w_l\in\mathbb{R}^{\lfloor 2n/3
\rfloor}}\mathbb{P}\left(\sum_{j=1}^l\|C(v_j)_{I_3}-w_j\|_2^2\leq c^2ln\right).
   \end{equation}
   Let $\eta\in\mathbb{R}^{\lfloor n/3\rfloor}$ be a random vector with i.i.d. coordinates of law $\xi$, and let $$V=((v_1)_{I_3},\cdots,(v_l)_{I_3})$$ be an $\lfloor n/3\rfloor\times l$ matrix with columns $(v_1)_{I_3},\cdots,(v_l)_{I_3}$. By assumption that $(v_1)_{I_3},\cdots,(v_l)_{I_3}$ are $\frac{1}{2}$-orthogonal, we have that $\|V\|\leq 2\sup_j \|v_j\|_2\leq 3$. Since $\|(v_s)_{I_3}\|_2\geq\frac{1}{4}$ for $s=1,\cdots,l$, we have $\|V\|_{HS}^2\geq \frac{l}{16}$.

   Then by Hanson-Wright inequality (see \cite{rudelson2013hanson}, Corollary 2.4) we have, for any $y\in\mathbb{R}^l$, 
   \begin{equation}\label{prodforaconst}
\mathbb{P}\left(\|V^T\eta-y\|_2^2\leq\frac{1}{64}l\right)\leq\mathbb{P}\left(\|V^T\eta-y\|_2^2\leq\frac{1}{4}\|V\|_{HS}^2\right)\leq \exp(-cln),
   \end{equation} for a constant $c>0$ depending only on $B$.

   Then we let $\eta_1,\cdots,\eta_{\lfloor 2n/3\rfloor}$ be i.i.d. copies of $\eta$, we have for any $y_1,\cdots,y_{\lfloor 2n/3\rfloor}\in\mathbb{R}^l$, 
   \begin{equation}\label{tensorizationeq3.4}
\mathbb{P}\left(\sum_{i=1}^{\lfloor 2n/3/\rfloor}\|V^T\eta_i-y_i\|_2^2\leq \frac{1}{504}ln\right)\leq\exp(-\frac{c}{3}ln).\end{equation} 
   This is because, if $\sum_{i=1}^{\lfloor 2n/3\rfloor}\|V^T\eta_i-y_i\|_2^2\leq\frac{1}{504}ln$, then we can find at least $n/2$ indices $j$ such that $\|V^T\eta_j-y_j\|_2^2\leq \frac{1}{64}l$. Then we use the independence of $\eta_j$ and estimate \eqref{prodforaconst}.

   Then we get the upper bound for the probability in \eqref{supssupssups} and complete the proof.
\end{proof}

Then we take a union bound over all compressible vectors with certain properties. We need a lemma on the operator norm of a random matrix:
\begin{lemma}\label{operatornorms} Let $Q$ be an $m\times n$ random matrix where $m\leq n,m,n\in\mathbb{N}$. Assume that $Q$ has independent coordinates $q_{ij}$ that are centered and $|q_{ij}|\leq 1$ almost surely. Then we can find constants $C_{\ref{operatornorms}}>0,c_{\ref{operatornorms}}>0$ such that 
$$
\mathbb{P}(\|Q\|\geq C_{\ref{operatornorms}}\sqrt{n})\leq \exp(-c_{\ref{operatornorms}}n).
$$

\end{lemma}
This lemma can be derived from standard operator norm bounds, see for instance \cite{rudelson2008littlewood}.

We use this lemma to carry out the following approximation procedure:

\begin{Proposition}\label{proposi3.2}
Let $A\sim\operatorname{Sym}_n(\zeta)$ with $\zeta\in\Gamma_B$, and take $l\in\mathbb{N}_+$ be such that $l/n$ is bounded by some constants depending only on $B$.

    Then there exists $\tau>0$ such that, for any given $i\in\{1,2,3\}$, the probability that there exists $l$-tuples $x_1,\cdots,x_l\in \mathbb{S}^{n-1}\cap \operatorname{Comp}(\tau^2n,\tau^4)$ satisfying the following conditions:
    \begin{enumerate} \item 
    $\|(x_j)_{I_i}\|_2\geq\frac{1}{3}$ for each $j\in[l]$, \item $((x_1)_{I_i},\cdots,(x_l)_{I_i})$ are $\frac{1}{4}$-almost orthogonal,  \item They satisfy 
    $
\|Ax_j\|_2\leq\tau\sqrt{n}\quad\text{for all }j\in[l],
    $\end{enumerate} is no more than $\exp(-cln)$. The intervals $I_i,i=1,2,3$ are defined in Lemma \ref{lemma3.1}.
\end{Proposition}

\begin{proof}

    Let $\tau\in(0,\frac{1}{2})$ and set 
    $$
T=\left\{v\in \frac{\tau}{\sqrt{n}}\mathbb{Z}^n:\|v\|_2\in[\frac{1}{2},\frac{3}{2}]\right\}.
    $$ Then we can estimate, using Lemma \ref{numberofintegralppoints}, 
    \begin{equation}\label{estimatecompressibles}
        |T\cap\operatorname{Sparse}(4\tau^2n)|\leq\binom{n}{4\tau^2n}\cdot(2+\frac{C}{\tau})^{4\tau^2n}\leq  (\frac{C'}{\tau^3})^{4\tau^2n}.
    \end{equation}

    For each vector $x_j\in\mathbb{R}^n$ we use $x_j(1),\cdots,x_j(n)\in\mathbb{R}$ to denote its $n$ coordinates. For an $\ell$-tuple of vectors $x_1,\cdots,x_l$ as in the statement, since each $x_j$ is in $\operatorname{Comp}(\tau^2n,\tau^4)$, for each $j$ we can find a subset of coordinates $I_1(j)\subset[n]$ with $|I_1(j)|\leq \tau^2n$ so that 
    $$
\sum_{i\in [n]\setminus I_1(j)} x_j^2(i)\leq\tau^8,
    $$
    so we can further find a subset $I_2(j)\subset I_1(j)$ with $|I_2(j)|\leq 2\tau^2n$ and
    $$ |x_j(i)|\leq\frac{\tau^3}{\sqrt{n}}\text{ for any }i\in[n]\setminus I_2(j).
    $$ For the given vector $x_j$ we define a rounded version $w_j=(w_j(1),\cdots,w_j(n))$ via
    $$
w_j(i)=\frac{\tau}{\sqrt{n}}\lfloor \frac{\sqrt{n}}{\tau}|x_j(i)|\rfloor\operatorname{sign}(x_j(i)),\quad i=1,\cdots,n,
    $$ so that the coordinates of $x_j$ with small absolute value are set to be zero. 
    
    Then we use a technique called random rounding, which was proposed by Livshyts \cite{livshyts2021smallest} in the setting of random matrices, to approximate $x_j$ by some vectors in $T$. Define a family of random variables
    $\epsilon_{i,j}$, $i=1,\cdots,n,j=1,\cdots,l$ via
    $$\mathbb{P}(\epsilon_{i,j}=w_j(i)-x_j(i))=1-\frac{\sqrt{n}}{\tau}|x_j(i)-w_j(i)|,
    $$

    $$\mathbb{P}\left(\epsilon_{i,j}=w_j(i)-x_j(i)+\frac{\tau}{\sqrt{n}}\operatorname{sign}(x_j(i))\right)=\frac{\sqrt{n}}{\tau}|x_j(i)-w_j(i)|.
    $$ Then we have by definition that $\mathbb{E}[\epsilon_{ij}]=0$ and $|\epsilon_{ij}|\leq \frac{\tau}{\sqrt{n}}$ almost surely.
We define the random approximation $v_j,j=1,\cdots,l$ of $x_j$ as follows:
$$
v_j(i):=x_j(i)+\epsilon_{i,j},\quad i=1,\cdots,n.
$$
Now we set $I_3(j)=I_2(j)\cup\{i\in[n]\setminus I_2(j):v_j(i)\neq 0\}$. Then
\begin{equation}\label{probabilitychernoff}
    \mathbb{P}(\forall j\in[l]:|I_3(j)|\leq4\tau^2n)\geq 1-l\exp(-c\tau^2n). 
\end{equation} 
To check this, note that for any $i\in[n]\setminus I_2(j)$, then $v_j(i)=0$ is equivalent to $\epsilon_{i,j}=v_j(i)-x_j(i)$, which has probability at least $1-\tau^2$. Since $\epsilon_{i,j}$ are independent, the probability estimate \eqref{probabilitychernoff} follows from Chernoff's inequality. If for each $j\in[l]$ we have $|I_3(j)|\leq 4\tau^2n$, then all $v_1,\cdots,v_l$ belong to $T\cap\operatorname{Sparse}(4\tau^2n)$.

Let $X_{I_i}$ and $V_{I_i}$ be the matrices with columns $(x_1)_{I_i},\cdots,(x_l)_{I_i}$ and $(v_1)_{I_i},\cdots,(v_l)_{I_i}$. Then by Lemma \ref{operatornorms}, 
$$\mathbb{P}(\|V_{I_i}-X_{I_i}\|\leq C_{\ref{operatornorms}}\tau)\geq 1-\exp(-c_{\ref{operatornorms}}n).$$ Taking $\tau$ small enough and using that $(x_1)_{I_i},\cdots,(x_l)_{I_i}$ are $\frac{1}{4}$-almost orthogonal, we claim that this implies $(v_1)_{I_i},\cdots,(v_l)_{I_i}$ are $\frac{1}{2}$-almost orthogonal. Indeed, let $D_{V_{I_i}}$ be the diagonal matrix $\operatorname{diag}(\|(v_1)_{I_i}\|_2,\cdots,\|(v_l)_{I_i}\|_2)$, and let $D_{X_{I_i}}$ be $\operatorname{diag}(\|(x_1)_{I_i}\|_2,\cdots,\|(x_l)_{I_i}\|_2)$, then  
\begin{equation}\label{repeatcomputations}
\begin{aligned}
&s_l(V_{I_i}D_{V_{I_i}}^{-1})\geq s_l(X_{I_i}D_{V_{I_i}}^{-1})-\|V_{I_i}-X_{I_i}\|_2\cdot\|D_{V_{I_i}}^{-1}\|\\&\geq s_l(X_{I_i}D_{X_{I_i}}^{-1})s_l(D_{X_{I_i}}\cdot D_{V_{I_i}}^{-1})-C_{\ref{operatornorms}}\tau\cdot (\frac{1}{3}-C_{\ref{operatornorms}}\tau)^{-1}\\&\geq \frac{3}{4}(1-\|D_{X_{I_i}}-D_{V_{I_i}}\|\cdot \|D_{V_{I_i}}\|^{-1})-C_{\ref{operatornorms}}\tau\cdot (\frac{1}{3}-C_{\ref{operatornorms}}\tau)^{-1}\\&\geq \frac{3}{4}(1-C_{\ref{operatornorms}}\tau\cdot (\frac{1}{3}-C_{\ref{operatornorms}}\tau)^{-1})-C_{\ref{operatornorms}}\tau\cdot (\frac{1}{3}-C_{\ref{operatornorms}}\tau)^{-1}\geq\frac{1}{2}
\end{aligned} 
\end{equation} whenever $\tau\leq\tau_0$ for a fixed $\tau_0>0$. A similar argument shows that $s_1(V_{I_i}D_{V_{I_i}}^{-1})\leq\frac{3}{2}$.

Thus we have proven that
\begin{equation}
    \mathbb{P}\left((v_1)_{I_i},\cdots (v_l)_{I_i}\text{ are }\frac{1}{2}-\text{almost orthogonal}
    \right)\geq 1-\exp(-c_{\ref{operatornorms}}n).
\end{equation}

We denote by $\mathcal{E}_{HS}$ the event that $\|A\|_{HS}\leq 2Kn$. Then $\mathbb{P}(\mathcal{E}_{HS})\geq 1-\exp(-c_{\ref{lemma2.1}}n^2)$ by Lemma \ref{lemma2.1}. Using independence of $\epsilon_{ij}$, we check that for each $j\in[l]$,
$$
\mathbb{E}\|A(x_j-v_j)\|_2^2=\mathbb{E}\|\sum_{i=1}^n\epsilon_{i,j}Ae_i\|_2^2=\sum_{i=1}^n\mathbb{E}\epsilon_{ij}^2\|Ae_i\|_2^2\leq(\frac{\tau}{\sqrt{n}})^2\|A\|_{HS}^2\leq 4K^2\tau^2n,
$$ where $e_1,\cdots,e_n$ denote the standard basis vectors of $\mathbb{R}^n$,
so that for each $j\in[l]$ we have by Chebyshev's inequality that 
$$
\mathbb{P}[\|A(x_j-v_j)\|_2\leq 3K\tau\sqrt{n}\mid \mathcal{E}_{HS}]\geq\frac{1}{2}.
$$
    Then by independence of $\epsilon_{ij}$, we deduce that 
    \begin{equation}
        \mathbb{P}[\forall j\in[l],\|A(x_j-v_j)\|_2\leq 3K\tau\sqrt{n}\mid\mathcal{E}_{HS}]\geq 2^{-l}.
    \end{equation}

Since $1-2^{-l}+\exp(-c_{\ref{lemma2.1}}n^2)+l\exp(-c\tau^2n)+\exp(-c_{\ref{operatornorms}}n)<1$, we can always find a realization of $\epsilon_{ij},i\in[n],j\in[l]$ such that the following statements hold at once:
 
   \begin{enumerate}
       \item $(v_1)_{I_i},\cdots,(v_l)_{I_i}$ are $\frac{1}{2}$-almost orthogonal and $\frac{1}{4}\leq\|(v_j)_{I_i}\|_2\leq\frac{5}{4}$ for all $j\in[l]$,
       \item $v_1,\cdots,v_l\in   T\cap \operatorname{Sparse}(4\tau^2n),$
       \item $\|A(x_j-v_j)\|_2\leq 3K\tau\sqrt{n}$ for each $j\in[l]$.
   \end{enumerate}
   Then the probability stated in Proposition \ref{proposi3.2} is bounded by the probability that there exists $v_1,\cdots,v_l$ satisfying conditions (1), (2), (3), which is bounded by 
   $$
(\frac{C'}{\tau^3})^{4\tau^2nl}\cdot e^{-c_{\ref{lemma3.1}}nl}\leq e^{-cnl}
   $$ for some $c>0$ depending only on $B$, provided that $\tau>0$ is small enough. The first factor is the number of choices (See \eqref{estimatecompressibles}) of $v_1,\cdots,v_l\in T\cap\operatorname{Sparse}(4\tau^2n)$ and the second factor is from Lemma \ref{lemma3.1} where we assume that $(3K+1)\tau\leq c_{\ref{lemma3.1}}$. This completes the proof.
\end{proof} 

Now we can complete the proof of Proposition \ref{compressibleprop}.

\begin{proof}[\proofname\ of Proposition \ref{compressibleprop}]
    For each unit vector $v$, at least one of $v_{I_1},v_{I_2},v_{I_3}$ must have $\ell^2$ norm at least $\frac{1}{3}$ by triangle inequality. We denote by $\operatorname{Comp}(\tau^2n,\tau^4)(I_i)$ the subset of $v\in \operatorname{Comp}(\tau^2n,\tau^4)\cap\mathbb{S}^{n-1}$ with $\|v_{I_i}\|_2\geq\frac{1}{3}$. The set $\operatorname{Comp}(\tau^2n,\tau^4)(I_i)$ satisfies the conditions in Lemma \ref{lemma2.4} with $I=I_i$. Then by Lemma \ref{lemma2.4} with $I=I_1$, (I) either we can find a $\frac{l}{12}$-tuple $v_1\cdots,v_{l/12}$ of vectors in $\ker A\cap \operatorname{Comp}(\tau^2n,\tau^4)(I_1)$ whose restriction $(v_1)_{I_1},\cdots,(v_{l/12})_{I_1}$ to $I_1$ is $\frac{1}{8}$-almost orthogonal, (II) or $\operatorname{Ker}A$ has a subspace of codimension at most $\frac{l}{12}$ which is disjoint from vectors in $\operatorname{Comp}(\tau^2n,\tau^4)(I_1)$. By Proposition \ref{proposi3.2}, the probability for case (I) to happen is at most $\exp(-cln)$. Applying this process further to $I_2$ and $I_3$, we deduce that with possibility at least $1-3\exp(-cln)$ we can find a subspace $W_1$ of $\ker A$ of codimension at most $l/4$ which is disjoint from $\cup_{i=1}^3\operatorname{Comp}(\tau^2n,\tau^4)(I_i)$, and hence $\ker A$ is disjoint from $\operatorname{Comp}(\tau^2n,\tau^4)\cap\mathbb{S}^{n-1}$.
\end{proof}

\section{No-gaps delocalization in the large deviation regime}\label{nogapsdev}

Thanks to Proposition \ref{compressibleprop}, we only need to consider a linear subspace $W$ of $\ker A$ where each unit vector in $W$ is incompressible. As the entries of $A$ are not independent, we shall use the ideas in the recent breakthrough \cite{campos2025singularity}, consisting of an ``inversion of randomness’' technique and conditioned inverse Littlewood-Offord theorems.
When applying this technique, we need to restrict the incompressible vectors $v_1,\cdots,v_l$ to a small interval $I\subset[n]$ of length $\Theta(n)$. Although this is a pretty benign operation for one vector $v\in\mathbb{R}^n$, serious problems arise when we have a growing number of vectors $v_1,\cdots,v_l$. While we may assume that $(v_1)_{I},\cdots,(v_l)_I$ are incompressible, there is no guarantee that $\operatorname{Span}((v_1)_I,\cdots,\operatorname{Span}(v_l)_I)\cap\mathbb{S}^{|I|-1}$ are incompressible \footnote{In this section, whenever $\operatorname{Comp}(\delta,\rho)$ and $\operatorname{Incomp}(\delta,\rho)$ are applied to a vector restricted on $I$, the ambient dimension of the vector is still $n$ when (in)compressibility is measured, as in Definition \ref{wherearecompressiblevectors}.}, and this property is fundamentally important when we apply inverse Littlewood-Offord theorems (even if we can assume that each $v\in\operatorname{Span}(v_1,\cdots,v_l)\cap\mathbb{S}^{n-1}$ is incompressible.). Meanwhile, the intervals $I$ we restrict to can be very short, say we may take $|I|\leq 2^{-50}n$.

To solve this problem, we prove the following proposition, which is a much stronger version of Proposition \ref{compressibleprop}:
\begin{Proposition}\label{mainpropsection4} Let $l\in\mathbb{N}$ with $l=o(n)$ and $A_n\in\operatorname{Sym}_n(\zeta)$ with $\zeta\in\Gamma_B$. For any $c_0>0$ we can find a constant $\tau_{c_0}>0$ and $c_{c_0}>0$ depending only on $c_0$ and $B$ such that with probability at least $1-\exp(-c_{c_0}ln)$, there exists a linear subspace $F_1'$ of $\ker A$ of codimension $\frac{1}{2}l$ such that 
$F_1'\cap (\operatorname{Comp}((1-c_0^2/10)n,\tau_{c_0})\cap\mathbb{S}^{n-1})=\emptyset.$\footnote{The assumption $l=o(n)$ is not necessary and is imposed only to simplify some computations. In our main result, Theorem \ref{maintheorem1.1}, we take $l=O(\sqrt{n})$, so this restriction does not weaken our main result.
}
\end{Proposition}

 Proposition \ref{mainpropsection4} states that we can consider a linear subspace of $\ker A$ spanned by unit vectors that are delocalized in the no-gaps sense. In a qualitative form, this notion of no-gaps delocalization can be defined as the following: for any $\epsilon>0$ and any unit eigenvector $v$ of $A$, for any interval $I$ of length $\epsilon n$ we have $\|v_I\|_2\geq c_\epsilon>0$ for some constant $c_\epsilon>0$ (see \cite{rudelson2016no} for a more quantitative characterization). No-gaps delocalization is also known to hold for some structured random matrices with independent entries, see \cite{rudelson2016singular} and \cite{cook2018lower}.
 
 However, the proof of no-gaps delocalization in \cite{rudelson2016no} for a random symmetric matrix does not seem to work in the large deviation regime, and neither do the methods in \cite{rudelson2016singular} and \cite{cook2018lower} apply here. The main reason is as follows: suppose that we begin with a linear subspace $F\subset\ker A$ such that $F\cap\mathbb{S}^{n-1}\subset\operatorname{Incomp}(u_kn,s_k)$ for some $u_k,s_k\in(0,1)$. We wish to select $l$ orthogonal unit vectors $v_1,\cdots,v_l\in F$ and show that the probability for $Av_i=0$ for each $i\in[l]$ is very small. Here $u_k$ is any value in $(0,1)$ and is in particular not sufficiently small, so that the argument in Section \ref{compressiblevectors3} based on Hanson-Wright inequality no longer applies. We wish to apply instead the inverse Littlewood-Offord theorem, Theorem \ref{littlewoodincomp}, combined with a bootstrap strategy inspired by \cite{rudelson2016singular} and \cite{cook2018lower}
 to iteratively increase $u_k$ and decrease $s_k$ in finitely many steps until we can achieve any $u_k\in(0,1)$. 
 
 In our case, as the entries of $A$ are not independent, we should extract a rectangular submatrix $H$ from $A=\begin{bmatrix}
     E&H^T\\H&G
 \end{bmatrix}$ having independent entries and apply inverse Littlewood-Offord theorem using incompressibility of $v_1,\cdots,v_l$ restricted to the columns of $H$. Unfortunately, we cannot determine exactly for any $v\in\operatorname{Span}(v_1,\cdots,v_l)\cap\mathbb{S}^{n-1}$ where does the essential support of $v$ (coordinates of $v$ that are not too small, which has cardinality at least $u_kn$ by assumption) lie among the $n$ coordinates: to ensure that an inverse Littlewood-Offord type theorem can be applied to $H$, $H$ must have at least $(1-u_k+\epsilon)n$ columns (the collection of columns we denote by $I_H$) for some $\epsilon>0$ for us to guarantee that $\operatorname{Span}((v_1)_{I_H},\cdots,(v_l)_{I_H})\cap\mathbb{S}^{|I_H|-1}\subset\operatorname{Incomp}(\epsilon n,s_k)$ so we can apply inverse Littlewood-Offord theorem. But in this case $H$ has only $(u_k-\epsilon )n$ rows, so the small ball probability given by the inverse Littlewood-Offord theorem, Theorem \ref{littlewoodincomp} is at most $(C\epsilon)^{(u_k-\epsilon)n}$ where we approximate $v_1,\cdots,v_l$ by vectors in $\epsilon n^{-1/2}\cdot \mathbb{Z}^n$. However, as each $v_i\in\operatorname{Incomp}(u_kn,s_k)$ we need at least $(\frac{1}{\epsilon})^{u_kn}$ approximations for $v_i$, so that the bound from inverse Littlewood-Offord theorem is not strong enough. When we consider two, or finitely many vectors $v_1,v_2,\cdots,v_d$, there might be some ad hoc methods to fix the problem, but we believe the situation is a lot more complicated when we have a growing number of vectors.

The inefficiency of the above approach is that we have ignored the anticoncentration effect due to the inner product of $H^T$ with $(v_i)_{[n]\setminus I_H},i\in[l]$. To make use of this information, we will use the idea of ``inversion of randomness’' from \cite{campos2025singularity} where we use randomness from a random selection of approximate vectors $v_1',\cdots,v_l'$ from certain boxes in the rescaled integer lattice\footnote{It may come as a surprise that we need these sophisticated techniques to prove no-gaps delocalization in the large deviation regime, as the original proof of no-gaps delocalization in \cite{rudelson2016no} is much simpler.}. For this purpose, we also need a novel version of the conditioned inverse Littlewood-Offord inequality, Theorem \ref{compressibletheorem12.234567}, which states that we can maintain an inverse Littlewood-Offord inequality for the small ball probability of $H(v_i)_{I_H},i\in[l]$, \textit{simultaneously} with a control on the small singular values of many submatrices of $H$. Since the integer lattices we sample the vectors $v_i'$ from can be fairly arbitrary, we need  a control on the small singular values of many rectangular submatrices all at once, and this is our added feature of Theorem \ref{compressibletheorem12.234567} compared to \cite{han2025repeated}, Section 12, which itself is a multidimensional generalization, for the part on inverse Littlewood-Offord theorem, of \cite{campos2025singularity}, Theorem 6.1.

This section is organized as follows. In Section \ref{initialreducion1}, we restate Proposition \ref{mainpropsection4} in a form where we use a bootstrap argument to increase the support of vectors, and we approximate the candidate vectors by vectors $v_i'$ sampled from a rescaled integer lattice such that $\|Av_i'\|_2,i\in[l]$ are small simultaneously. In Section \ref{implementation1}, we set up all the technical tools for an application of  the ``inversion of randomness’' technique. In Section \ref{implementation2}, we implement the inversion procedure and bound the cardinality of candidate vectors via a first moment method. Finally, in Section \ref{finalbootstrap} we combine the previous estimates to complete the proof of Proposition \ref{mainpropsection4} via a bootstrap argument.

\subsection{Initial reduction and approximation by sparse vectors}\label{initialreducion1}
In this section we let $c_0\in(0,1)$ be a fixed constant satisfying $c_0\leq\min(2^{-50},\tau^4/2)$ for the constant $\tau$ in Proposition \ref{compressibleprop}.
The proof of Proposition \ref{mainpropsection4} can be derived from the following proposition:

\begin{Proposition}\label{proposition4.22}
    Let $\zeta\in\Gamma_B$, $A\in \operatorname{Sym}_n(\zeta)$ and fix $u_k,s_k\in(0,1)$ with $2c_0\leq u_k\leq 1-c_0/10$. Let $F\subset\ker A$ be a linear subspace such that $F\cap(\operatorname{Comp}(u_k n,s_k)
    \cap\mathbb{S}^{n-1})=\emptyset. 
    $ Then for any $l\in\mathbb{N}_+,l=o(n)$ we can find a constant $c>0$ and $s_{k+1}\in(0,1)$ depending only on $B$, $c_0$ and $s_k$ such that, with probability at least $1-e^{-cln}$, we can find a linear subspace $F_1\subset F$ of codimension at most $l$, such that $F_1\cap(\operatorname{Comp}(u_{k+1}n,s_{k+1})
    \cap\mathbb{S}^{n-1})=\emptyset 
    $ where $$u_{k+1}=u_k+c_0^4/40.$$ 
\end{Proposition}
One only needs to apply Proposition \ref{proposition4.22} for at most $40/c_0^4$ times to complete the proof of Proposition \ref{mainpropsection4}. 
Now we take $$\epsilon_0=c_0^4/40,\quad d=(1-u_k+\epsilon_0)n.$$ Consider a partition of $A$ by 
$$
A=\begin{bmatrix}
    E&H^T\\H&G
\end{bmatrix} 
$$ where $E$ and $G$ are square matrices of size $d$ and $n-d$ respectively, and let $I_H\subset[n]$ be the columns of $H$, so that $|I_H|=d$. For any $v\in\operatorname{Incomp}(u_kn,s_k)\cap\mathbb{S}^{n-1}$ we must have $v_{I_H}\in\operatorname{Incomp}(\epsilon_0 n,s_k)$, otherwise $v$ has distance at most $s_k$ to a vector supported on $\epsilon_0 n+n-d=u_kn$ coordinates, a contradiction to $v\in\operatorname{Incomp}(u_kn,s_k)$. Then we have $v_{I_H}/\|v_{I_H}\|_2\in\operatorname{Incomp}(\epsilon_0 n,s_k)$ since $\|v_{I_H}\|_2\leq\|v\|_2=1$.

Next, we need to fix the location of the effective coordinates of a sparse vector $v$. 
\begin{notation}
 For any $v\in\operatorname{Comp}(u_{k+1}n,s_{k+1})\cap\mathbb{S}^{n-1}$, there exists a subset $I_v^{k+1}\subset[n]$ such that $|I_v^{k+1}|=u_{k+1}n$ and  $\|v_{[n]\setminus I_v^{k+1}}\|_2\leq s_{k+1}$ by definition, and we call $I_v^{k+1}$ the set of \textit{effective coordinates} of $v$ with respect to $s_{k+1}$ (if $I_v^{k+1}$ is not uniquely defined, we then choose any possible $I_v^{k+1}$. We use the superscript $k+1$ to signify its dependence on $s_{k+1}$).
\end{notation}
Define a family of intervals $J_i=[(i-1)c_0n,ic_0n]\cap\mathbb{Z}$ for all $i=1,\cdots,\lfloor 1/c_0\rfloor$. We then define $J_{\lfloor1/c_0\rfloor+1}=[n-c_0n,n]\cap\mathbb{Z}$. Then by pigeon-hole principle there must exist some $i\in[\lfloor 1/c_0\rfloor+1]$ such that $|I_v^{k+1}\cap J_i|\geq c_0^2n/2$, otherwise $|I_v^{k+1}|\leq\sum_{i\in[\lfloor 1/c_0\rfloor+1]} |I_v^{k+1}\cap J_i|\leq c_0n$, which contradicts the assumption that $|I_v^{k+1}|= u_{k+1}n\geq c_0n$.

 Thanks to this reduction, we can reduce the proof of Proposition \ref{proposition4.22} to the following:

 \begin{fact}\label{fact4.44}To prove Proposition \ref{proposition4.22} one only needs to prove the following alternative conclusion for each $i\in\lfloor 1/c_0\rfloor+1$:
     we can find $c$ and $s_{k+1}$ such that with probability at least $1-e^{-cln}$, there is a linear subspace $F_{1i}\subset F$ of codimension at most $l$ such that $F_{1i}\cap\{v\in\mathbb{S}^{n-1}\cap \operatorname{Comp}(u_{k+1}n,s_{k+1}):|I_v^{k+1}\cap J_i|\geq c_0^2n/2\}=\emptyset$, where $u_{k+1}=u_k+c_0^4/40$.
 \end{fact}
 \begin{proof}
     By the reasoning above, each $v\in\mathbb{S}^{n-1}\cap\operatorname{Comp}(u_{k+1}n,s_{k+1})$ satisfies that $|I_v^{k+1}\cap J_i|\geq c_0^2n/2$ for some $i$. We then replace the constant $l$ in the statement of Fact \ref{fact4.44} by $l/2c_0$ and define $F_1=\cup_{i=1}^{\lfloor 1/c_0\rfloor+1}F_{1i}$, so that $F_1$ has codimension at most $l$, and by the first sentence of this proof, we see that $F_1$ satisfies the requirement of Proposition \ref{proposition4.22} once the alternative conclusion in Fact \ref{fact4.44} is proven. 
 \end{proof}

 Without loss of generality, we only consider the case $i=\lfloor 1/c_0\rfloor+1$ and we denote by $J_*=[n-c_0n,n]=J_{\lfloor 1/c_0\rfloor+1}$. Then by our definition of $d$ and $\epsilon_0$, we have $J_*\subset [d+1,n]$.

We will then apply Lemma \ref{lemma2.4} to select $l$ almost orthogonal vectors $v_1,\cdots,v_l$ from $F\subset\ker A$ satisfying $|I_{v_i}^{k+1}\cap J_*|\geq c_0^2n/2$ for each $i\in[l]$. The selected vector tuples satisfy the following four conditions, as outlined in Fact \ref{notation4.72}:

We now record the consequence of Lemma \ref{lemma2.4} for the specific choice of parameters used in the proof of Proposition \ref{proposition4.22}.

\begin{fact}\label{notation4.72} Let $F$ be a linear subspace such that $$F\subset\ker A: F\cap\operatorname{Comp}(u_kn,s_k)\cap\mathbb{S}^{n-1}=\emptyset.$$ Let $I=[d]$ where $d=(1-u_k+\epsilon_0)n$ and let $W_{[d]}$ be the following subset
$$
W_{[d]}:=\{v\in\mathbb{S}^{n-1}\cap \operatorname{Comp}(u_{k+1}n,s_{k+1})\setminus\operatorname{Comp}(u_kn,s_k):|I_v^{(k+1)}\cap J_*|\geq c_0^2n/2\}.
$$ Then Lemma \ref{lemma2.4} applies with $E=F,I=[d]$ and $W_I=W_{[d]}$. Consequently, if alternative (1) of Lemma \ref{lemma2.4} can be applied for $l$ times yielding an $l$-tuple of vectors $v_1,\cdots,v_l\in F\cap W_{[d]}$, then they must satisfy
\begin{enumerate}
   
    \item $(v_1)_{[d]},\cdots,(v_l)_{[d]}$ are $\frac{1}{8}$-orthogonal with respect to the Euclidean norm on $\mathbb{R}^{d}$.
    \item $|\|(v_i)_{[d]}\|_2\geq s_k$ for each $i\in[l]$.
    \item $\operatorname{Span}((v_1)_{[d]},\cdots,(v_l)_{[d]})\cap\mathbb{S}^{d-1}\subset\operatorname{Incomp}(\epsilon_0 n,s_k)$.
    \item $v_1,\cdots,v_l\in\operatorname{Comp}(u_{k+1}n,s_{k+1})$ and  $|I_{v_i}^{k+1}\cap J_*|\geq c_0^2/2,\forall i\in[l].$
\end{enumerate}
\end{fact}
\begin{proof}
For every $v\in F\cap\mathbb{S}^{n-1}$, we have $v\in\operatorname{Incomp}(u_kn,s_k)$. Since $d=(1-u_k+\epsilon_0)n$, we have $v_{[d]}\in\operatorname{Incomp}(\epsilon_0 n,s_k)$ and in particular $1\geq \|v_{[d]}\|_2\geq s_k$, so $v_{[d]}/\|v_{[d]}\|_2\in\operatorname{Incomp}(\epsilon_0 n,s_k)$. Thus Lemma \ref{lemma2.4} applies with the above choice of $E=F,I$ and $W_I$. The condition (2) and (3) follow from incompressibility of $v_{[d]}$. Condition (1) is alternative (1)(a) of Lemma \ref{lemma2.4} and condition (4) is built into the definition of $W_{[d]}$.
\end{proof}

Then we prove that we can approximate the unit vectors $v_1,\cdots,v_l$ satisfying the aforementioned conditions in Fact \ref{notation4.72} by vectors from a grid of small cardinality:
\begin{Proposition}\label{lemma4.6stated}
    On the event $\|A\|_{HS}\leq 2Kn$,
    let $v_1,\cdots,v_l\in\mathbb{S}^{n-1}$ be an $\ell$-tuple of vectors in $\ker A$ satisfying the assumptions (1) to (4) of Fact \ref{notation4.72}.    
    
   Then we can find some constant $s_0'>0$ depending only on $c_0$ and $\tau$ such that, whenever $s_{k+1}<s_0'$ and \begin{equation}\label{relationsk+1}(s_{k+1})^{1/4}\leq \frac{1}{40C_{\ref{operatornorms}}}s_k^2,\end{equation} we can find vectors $v_1',\cdots,v_l'$ satisfying the following conditions \begin{enumerate}
       \item  $\|v_j'\|_2\in[\frac{1}{2},\frac{3}{2}]$, $\|(v_j')_{[d]}\|_2\geq \frac{s_k}{2}$ and $\|v_j'-v_j\|_\infty\leq \frac{(s_{k+1})^{1/4}}{\sqrt{n}}$ for each $j\in[l].$
         \item $v_j'\in\frac{(s_{k+1})^{1/4}}{\sqrt{n}}\mathbb{Z}^n$ for each $j\in[l]$, and\\
$v_j'\in\operatorname{Sparse}((u_{k+1}+\epsilon_0)n,s_{k+1})$, and $|\operatorname{supp}(v_j')\setminus I_{v_j}^{k+1}|\leq \epsilon_0n$ for each $j\in[l]$.
        \item  $(v_1')_{[d]},\cdots,(v_l')_{[d]}$ are $\frac{1}{4}$-orthogonal with respect to the norm on $\mathbb{R}^{d}$.
          \item $\operatorname{Span}((v_1')_{[d]},\cdots,(v_l')_{[d]})\cap\mathbb{S}^{d-1}\subset\operatorname{Incomp}(\epsilon_0 n,s_k/2)$.
        \item   $
\|Av_j'\|_2\leq 6K(s_{k+1})^{1/4}\sqrt{n}\quad\text{for each }j=1,\cdots,l.
    $ 
    \end{enumerate} 
\end{Proposition}

\begin{proof}

The argument is similar to the proof of Proposition \ref{proposi3.2}. By definition, for each $v_j\in\mathbb{S}^{n-1}\cap \operatorname{Incomp}(u_{k+1}n,s_{k+1})$ we can find $I_1(j):=I_{v_j}^{k+1}\subset[n]$ with $|I_1(j)|=u_{k+1}n$ such that $\sum_{i\in[n]\setminus I_1(j)}|v_j(i)|^2\leq s_{k+1}^2$. Then we can find a subset $I_2(j)\subset[n]$, $I_2(j)\supset I_1(j)$  with $|I_2(j)\setminus I_i(j)|\leq s_{k+1}n$ such that for any $i\in [n]\setminus I_2(j)$ we have $|v_j(i)|\leq \frac{(s_{k+1})^{1/2}}{\sqrt{n}}.$ Now we define vectors $w_j(i)$ approximating $v_j(i)$ as follows:
$$
w_j(i)=\frac{(s_{k+1})^{1/4}}{\sqrt{n}}\cdot
\lfloor \frac{\sqrt{n}}{(s_{k+1})^{1/4}}  v_j(i)\rfloor
\operatorname{sign}(v_j(i)),
$$and define random variables 
$\epsilon_{ij}$ via 
$$
\mathbb{P}(\epsilon_{ij}=w_j(i)-v_j(i))=1-\frac{\sqrt{n}}{(s_{k+1})^{1/4}}|v_j(i)-w_j(i)|
$$ and that
$$
\mathbb{P}(\epsilon_{ij}=w_j(i)-v_j(i)+\frac{(s_{k+1})^{1/4}}{\sqrt{n}}\operatorname{sign}(v_j(i)))=\frac{\sqrt{n}}{(s_{k+1})^{1/4}}|v_j(i)-w_j(i)|.
$$
Then we define random vectors $v_1',\cdots,v_l'$ via $$v_j'(i)=v_j(i)+\epsilon_{ij}\text{  for each  }j\in[l],i\in[n].$$ Then by definition, we have $v_j'\in \frac{(s_{k+1})^{1/4}}{\sqrt{n}}\mathbb{Z}^n$ and $\|v_j'-v_j\|_\infty\leq \frac{(s_{k+1})^{1/4}}{\sqrt{n}}$ for each $j\in[l]$.

By Lemma \ref{operatornorms}, with probability at least $1-e^{-c_{\ref{operatornorms}n}}$
we have $\|\epsilon_{ij}\|_{op}\leq C_{\ref{operatornorms}}(s_{k+1})^{1/4}$, so that whenever $s_{k+1}<s_0'$ for a given $s_0'>0$ depending only on $c_0$ and $\tau$, we have that $\|v_i'\|_2\in[\frac{1}{2},\frac{3}{2}]$ and $\|(v_i')\|_{[d]}\|_2\geq s_k-C_{\ref{operatornorms}}(s_{k+1})^{1/4}\geq \frac{s_k}{2}$ by our assumption on $s_{k+1}$. This justifies condition (1) whenever $\|\epsilon_{ij}\|_{op}\leq C_{\ref{operatornorms}}(s_{k+1})^{1/4}$. In the following we let $s_0'>0$ be a constant that may change from line to line, but $s_0'$ only depends on $c_0$ and $\tau$.

We define $I_3(j)=I_2(j)\cup \{i\in[n]\setminus I_2(j): v'_j(i)\neq 0\}$. Since for each $i\in[n]\setminus  I_2(j)$ we have $w_j(i)=0$ and thus $\mathbb{P}(\epsilon_{ij}\neq 0)\leq (s_{k+1})^{1/4}$, we apply Chernoff's inequality to get 
$$
\mathbb{P}(|I_3(j)\setminus I_1(j)|\leq (s_{k+1}+2(s_{k+1})^{1/4})n\text{ for each }j\in[l])\leq 1-l\exp(-c(s_{k+1})^{1/2}n).
$$ Thus with probability at least $1-l\exp(-c(s_{k+1})^{1/2}n)$, the randomly chosen vectors $v_i',i\in[l]$ satisfy condition (2), whenever $s_{k+1}<s_0'$ for some $s_0'>0$ so that $s_{k+1}+2(s_{k+1})^{1/4}\leq \epsilon_0$. 

To check condition (3), we only need to repeat the computations in \eqref{repeatcomputations} which holds whenever $\|\epsilon_{ij}\|_{op}\leq C_{\ref{operatornorms}}(s_{k+1})^{1/4}$ and whenever $s_{k+1}$ satisfies \eqref{relationsk+1}. The latter estimate on $\|\epsilon_{ij}\|_{op}\leq C_{\ref{operatornorms}}(s_{k+1})^{1/4}$ holds on an event with probability $1-e^{-c_{\ref{operatornorms}n}}$, thanks to Lemma \ref{operatornorms}.

Then we check that condition (4) follows from conditions (1),(2),(3). Let $U$ be the $d\times l$ matrix with columns $(v_1)_{[d]},\cdots,(v_l)_{[d]}$ and $V$ be the $d\times l$ matrix with columns $(v_1')_{[d]},\cdots,(v_l')_{[d]}$. Then for any $\theta\in\mathbb{R}^l$ such that $\|V\theta\|_2=1$, using the conditions (1) and (3) of this proposition, we get that $s_k/4\cdot \|\theta\|_2\leq  \|V\theta\|_2=1,$ so that $\|\theta\|_2\leq 4/s_k$. Then we have that on the event $\|\epsilon_{ij}\|_{op}\leq C_{\ref{operatornorms}}(s_{k+1})^{1/4}$, 
$$
\|U\theta\|_2\geq \|V\theta\|_2-\|\theta\|_2\|U-V\|\geq 1-C_{\ref{operatornorms}}(s_{k+1})^{1/4}\cdot\frac{4}{s_k}\geq\frac{3}{4} 
$$ by assumption \ref{relationsk+1}. Then for any $y\in\operatorname{Sparse}(\epsilon_0 n)$, we have
$$
\|U\theta-y\|_2\geq \|U\theta\|_2\|\frac{U\theta}{\|U\theta\|_2}-\frac{y}{\|U\theta\|_2}\|_2\geq\frac{3}{4}s_k
$$ since $\frac{U\theta}{\|U\theta\|_2}\in\operatorname{Incomp}(\epsilon_0 n,s_k)$ by our assumption (3) in Notation \ref{notation4.72}.

Then we have 
$$
\|V\theta-y\|_2\geq \|U\theta-y\|_2-\|U-V\|\|\theta\|_2\geq \frac{3}{4}s_k-\frac{4}{s_k}C_{\ref{operatornorms}}(s_{k+1})^{1/4}\geq \frac{1}{2}s_k
$$ whenever $s_{k+1}$ satisfies \eqref{relationsk+1}. This verifies condition (4).

Finally, we check condition (5). 
On the event that $\|A\|_{HS}\leq 2Kn$, using independence of $\epsilon_{ij}$, we check that for each $j\in[l]$,
$$
\mathbb{E}\|A(v_j'-v_j)\|_2^2=\mathbb{E}\|\sum_{i=1}^n\epsilon_{i,j}Ae_i\|_2^2=\sum_{i=1}^n\mathbb{E}\epsilon_{ij}^2\|Ae_i\|_2^2\leq(\frac{(s_{k+1})^{1/4}}{\sqrt{n}})^2\|A\|_{HS}^2\leq 4K^2(s_{k+1})^{1/2}n,
$$ where $e_1,\cdots,e_n$ denote the standard basis vectors of $\mathbb{R}^n$,
so that for each $j\in[l]$ we have by Chebyshev's inequality that 
$$
\mathbb{P}[\|A(v_j'-v_j)\|_2\leq 6K(s_{k+1})^{1/4}\sqrt{n}\mid \mathcal{E}_{HS}]\geq\frac{1}{2}.
$$
    Then by independence of $\epsilon_{ij}$, we deduce that 
    \begin{equation}
        \mathbb{P}[\forall j\in[l],\|A(v_j'-v_j)\|_2\leq 6K(s_{k+1})^{1/4}\sqrt{n}\mid\mathcal{E}_{HS}]\geq 2^{-l}.
    \end{equation}

We sum up the probabilities and get $1-2^{-l}+l\exp(-c\tau^2n)+\exp(-c(s_{k+1})^{1/2}n)<1$ whenever $n$ is large enough, noting that $l=o(n)$. Thus, there exists a realization of $\epsilon_{ij}$ where the claims (1) to (5) hold simultaneously. This completes the proof.
\end{proof}

\subsection{Implementing the ``inversion of randomness’' counting procedure I}\label{implementation1}
To implement the ``inversion of randomness’' argument of \cite{campos2025singularity}, we need to replace the matrix $A$ by the following zeroed out version, where $H$ has size $(n-d)\times d$ and has i.i.d. coordinates\footnote{Actually, the distribution of entries of $H$ in $M$ is not $\zeta$ but a certain centered lazier version of $\zeta$, see Definition \ref{definition4.74.7}. This minor change of distribution does not affect our informal discussion here.}: 
$$
M_0=\begin{bmatrix}
    0&H^T\\H&0
\end{bmatrix},
$$ where we use small ball probability from the inverse Littlewood-Offord theorems in the $H$ component, and small ball probability from a random selection of coordinate vectors on the columns associated to $H^T$. By assumption, our candidate vectors $v_i'$ are $(u_{k+1}+\epsilon_0)n$-sparse, and $|I_{v_i}^{k+1}\cap[d+1,n]|\geq c_0^2n/2$. Then we need the following further partitioning of $H$:
$$
H=\begin{bmatrix}
    H_{I\times J}&*\\**&*
\end{bmatrix}
$$ where $H_{I\times J}$ is the submatrix indexed by $I\subset[n-d]$ and $J\subset[d]$. We will use the submatrix $H_{I\times J}^T$ to reuse the randomness from a randomly selected vector $v_i'$ where $I_{v_i'}^{k+1}\cap[d+1,n]\supset\{d+i:i\in I\}$. we require the size of $J$ to be $|J|\geq c_0^4n/8$ and $|I|\geq c_0^2n/2$. Then from this inversion of randomness argument, we can work as if we have $n-d+|J|\geq (u_k+c_0^4/10)n$ rows to generate randomness, which is larger than $(u_{k+1}+\epsilon_0)n\leq (u_k+c_0^4/20)n$, the number of coordinates to generate $v_j'$. Then we can set $s_{k+1}$ small enough to prove Proposition \ref{proposition4.22}.

This discussion motivates the following conditioned inverse Littlewood-Offord theorem, Theorem \ref{compressibletheorem12.234567}. Before stating the theorem, we introduce some additional notations:

\begin{Definition}\label{definition4.74.7}
    Recall that $\zeta$ is a mean 0, variance 1 subgaussian random variable with $\zeta\in\Gamma_B$ for some $B>0$, and $\zeta'$ an independent copy of $\xi$. We define 
$$
\tilde{\zeta}=\zeta-\zeta'.
$$ Define $I_B=(1,16B^2)$ and let $\bar{\zeta}$ be $\tilde{\zeta}$ conditioned on the event $|\tilde{\zeta}|\in I_B.$

We let $p:=\mathbb{P}(|\widetilde{\xi}|\in I_B)$. Then by \cite{campos2024least}, Lemma II.1, we have \begin{equation}\label{whatdoesbhave?}p\geq\frac{1}{2^7B^4}.\end{equation}

For given $\nu\in(0,1)$ let $Z_\nu$ be an independent Bernoulli variable with mean $\nu$, and define 
$$
\xi_\nu:=\mathbf{1}\{|\tilde{\zeta}\in I_B\}\tilde{\zeta}Z_\nu.
$$
We write $X\sim\Phi_\nu(d;\zeta)$ to mean a $d$-dimensional random vector with i.i.d. coordinates of distribution $\tilde{\zeta}Z_\nu$, and we write $X\sim\Xi_\nu(d;\zeta)$ to mean a $d$-dimensional random vector with i.i.d. coordinates of distribution $\xi_\nu$.

For a given $d\in\mathbb{N}_+,d\leq n$ define the following zeroed-out matrix 
\begin{equation}\label{zeroedoutmatrixM}
    M=\begin{bmatrix}
        \mathbf{0}_{[d]\times[d]}&H^T\\H&\mathbf{0}_{[n-d]\times[n-d]}
    \end{bmatrix} 
\end{equation} where $H$ is an $(n-d)\times d$ random matrix with i.i.d. entries of distribution $\tilde{\zeta}Z_\nu$ (This definition differs from the $M_0$ given at the beginning of Section \ref{implementation1}). 

\end{Definition}

\begin{theorem}\label{compressibletheorem12.234567}
    For given $n\in\mathbb{N}$, $0<c_0\leq 2^{-50}B^{-4}$, let $d\leq n$ and fix $\alpha\in(0,1)$.  Consider an $\ell$-tuple of $\frac{1}{4}$-orthogonal vectors $X_1,\cdots,X_\ell\in\mathbb{R}^d$ that satisfy $D_{L,\alpha}(\frac{c_0}{32\sqrt{n}}\mathbf{X})\geq 256B^2\sqrt{\ell}$ and that $\|X_i\|_2\geq \sqrt{n}/x_i$ for each $i\in[\ell]$ and some $x_i>0,i\in[\ell]$ (where we denote $\frac{c_0}{32\sqrt{n}}\mathbf{X}:=(\frac{c_0}{32\sqrt{n}}X_1,\cdots,\frac{c_0}{32\sqrt{n}}X_\ell)$) and where we choose $L=(\frac{8}{\sqrt{\nu p}}+\frac{256B^2}{\sqrt{c_0}})\sqrt{\ell}$.

    Let $H$ be an $(n-d)\times d$ random matrix with i.i.d. rows of distribution $\Phi_\nu(d;\zeta)$ with $\nu=2^{-15}$ (and we recall $p\geq\frac{1}{2^7B^4}$ from \eqref{whatdoesbhave?}.

    Let $n_e,d_e\in\mathbb{N}$ satisfy that $c_0^2n_e/4\leq d_e\leq c_0^2n_e$ and $n_e= c_0^2n/2$. Let $J$ be the interval $[1,d_e]\cap\mathbb{N}$.
    
    Then whenever $k\leq 2^{-32}B^{-4}c_0^3n$ and $\prod_{i=1}^\ell Rt_i\geq\exp(-2^{-32}B^{-4}c_0^6n)$, we have:
    \begin{equation}\begin{aligned}\label{finalwehaveestimates}
    \mathbb{P}_H&(\sigma_{d_e-k+1}(H_{I\times J})\leq c_0^22^{-4}\sqrt{n}\text{ for some interval } I\subset[n-d] \text{ with  } |I|\geq n_e, \\&
        \text{ and }\|HX_i\|_2\leq c_0 n\quad\text{for all }i\in[\ell])\\&\leq 2^n e^{-c_0^3nk/48}(\prod_{i=1}^\ell\frac{R x_i}{\alpha})^{n-d},
    \end{aligned}\end{equation}
    where $R=2^{46}B^2c_0^{-3}(\frac{8}{\sqrt{\nu p}}+\frac{256B^2}{\sqrt{c_0}})$.\end{theorem}

 The proof of Theorem \ref{compressibletheorem12.234567} is an adaptation of similar results in \cite{campos2025singularity}, \cite{han2025repeated} and is deferred to Section \ref{proofinverse}.

To reuse the randomness in the component $H_{I\times J}^T$,
we need the following theorem from \cite{rudelson2015small}, Corollary 1.4 and Remark 2.3:
\begin{theorem}\label{theorem4.9newold}
    Let $N>0$, and let $d_e,n_e,k\in\mathbb{N}$ be such that $n_e\geq d_e\geq k$. Let $P$ be an orthogonal projection of $\mathbb{R}^{n_e}$ onto a $d_e-k$-dimensional subspace, and take $X=(X_1,\cdots,X_{n_e})$ be a random vector with independent coordinates that satisfy 
    $$
\mathcal{L}(X_i,1/2)\leq N^{-1}\quad\forall i\in[n_e].
    $$
    Then for any $K\geq 1$ we have
    $$
\max_{y\in\mathbb{R}^{n_e}}\mathbb{P}_X(\|PX-y\|_2\leq K\sqrt{d_e-k})\leq(\frac{CK}{N})^{d_e-k},
    $$
    where $C>0$ is a universal constant.
\end{theorem}
   Theorem \ref{theorem4.9newold} has the following corollary, which is similar to \cite{campos2025singularity}, Lemma 7.5:
\begin{corollary}\label{corollary4.11} Let $N\in\mathbb{N}$, $n_e,d_e,k\in\mathbb{N}$ be such that $n_e\geq d_e\geq 2k$. 
    Let $H$ be a $d_e\times n_e$ matrix with $\sigma_{d_e-k}(H)\geq c_0^2\sqrt{n}/16$. For a fixed vector $v=(v_1,\cdots,v_{n_e})\in\mathbb{R}^{n_e}$, let $B_1,\cdots, B_{n_e}\subset\mathbb{R}$ be such that $B_1+v_1,\cdots,B_{n_e}+v_{n_e}\subset\mathbb{Z}$ and $|B_i|\geq N$ for each $i$. Let $X$ be a random vector uniformly chosen from $\mathcal{B}=B_1\times \cdots\times B_{n_e}$, then for any $K\geq 1$, 
    $$
\sup_{y\in\mathbb{R}^{d_e}}\mathbb{P}_X(\|HX-y\|_2\leq 6K n)\leq (\frac{CKn}{d_ec_0^2N})^{d_e-k}
    $$ for a universal constant $C>0$.
\end{corollary}
We note that $H$ in this corollary corresponds to the matrix $H_{I\times J}^T$ in Theorem \ref{compressibletheorem12.234567}.
\begin{proof}
    We first observe that we only need to prove the estimate for $y=0$: this can be achieved by first projecting $y$ onto the image of $H\mathbb{R}^{n_e}$, and then translate the coordinates of $X$ by a fixed vector, and the latter change does not change the assumptions on $B_i$.
    
    Then for $y=0$, consider the spectral decomposition $H^TH=\sum_{i=1}^{d_e}\sigma_i(H)^2v_iv_i^T$ where $v_1,\cdots,v_{d_e}\in\mathbb{R}^{n_e}$ are orthogonal vectors. We define a projection $P=\sum_{i=1}^{d_e-k}v_iv_i^T$, then 
    $$
\|HX\|_2^2=\sum_{i=1}^{d_e}\sigma_j(H)^2\langle X,v_j\rangle^2\geq \sigma_{d_e-k}(H)^2\sum_{j=1}^{2d-k}\langle X,v_j\rangle^2\geq 2^{-8}c_0^4n\|PX\|_2^2.
    $$
    Then 
    $$
\mathbb{P}(\|HX\|_2\leq 6Kn)\leq\mathbb{P}(\|PX\|_2\leq 6K\cdot 16c_0^{-2}\sqrt{n}). 
    $$ Now we take ${k'}= 6K.16c_0^{-2}\sqrt{n/(d_e-k)}$ in the setting of Theorem \ref{theorem4.9newold} and get that $\mathbb{P}(\|PX\|_2\leq {k'}\sqrt{d_e-k})\leq (\frac{CKn}{d_ec_0^2N})^{d_e-k}$. This completes the proof.
\end{proof}

In Proposition \ref{lemma4.6stated}, we approximate the vectors $v_1,\cdots,v_l$ by $v_1',\cdots,v_l'\in B_n(0,2)\cap \epsilon n^{-1/2}\cdot\mathbb{Z}^n\cap\operatorname{Sparse}((u_{k+1}+\epsilon_0)n)$ with the choice $\epsilon=\frac{(s_{k+1})^{1/4}}{\sqrt{n}}$. As the latter set has complicated geometry, we will cover the latter set by a family of boxes, each box being a Cartesian product of $(u_{k+1}+\epsilon_0)n$ intervals. Then we generate vectors from integer lattices, allowing us to use Corollary \ref{corollary4.11} more efficiently. This construction of boxes was first introduced by Tikhomirov in \cite{tikhomirov2020singularity}, and this notion yields a probabilistic interpretation for the enumeration of vectors from a given lattice. This covering step is detailed as follows:

\begin{lemma}\label{howlarageisthebox}
    Fix $\epsilon>0$ and $n,d_0\in\mathbb{N}$ with $d_0\geq c_0^4n/4$, we define 
    $$
\Lambda_{\epsilon}[d_0]:=B_n(0,2)\cap (4\epsilon n^{-1/2}\cdot\mathbb{Z}^n)\cap\operatorname{Sparse}(d_0).
    $$ Also, for given $N\in\mathbb{N}$ and $\kappa>0$ we say that $\mathcal{B}=B_1\times\cdots B_n\subset\mathbb{Z}^n$ is an $(N,d_0,\kappa)$-sparse box if we can find a subset $I_\mathcal{B}\subset[n]$ with $|I_\mathcal{B}|= d_0$ such that  $|B_i|\geq N$ for each $i\in I_\mathcal{B}$, $B_i=\{0\}$ for $i\in[n]\setminus I_\mathcal{B}$ and
    $|\mathcal{B}|\leq (\kappa N)^{d_0}$.

    Then for all $\epsilon>0$  we can find a family $\mathcal{F}$ of $(N,d_0,\kappa)$-sparse boxes so that 
    $$
\Lambda_\epsilon[d_0]\subset \cup_{B\in\mathcal{F}}(4\epsilon n^{-1/2})\cdot\mathcal{B},
    $$  where we take $N=1/(4\epsilon)$ and $\kappa=32/c_0^4$, and the family $\mathcal{F}$ has cardinality $|\mathcal{F}|\leq2^{10n}$.
\end{lemma}

\begin{proof}
    For $\ell\geq 1$ define the interval $I_\ell=[-2^\ell N,2^\ell N]\setminus[-2^{\ell-1}N,2^{\ell-1}N]$, and define $I_0=[-N,N]$. We also define $I_{-1}=\{0\}$.

    Then for $(\ell_1,\cdots,\ell_n)\in\mathbb{Z}_{\geq -1}^n$ such that $|\{i\in[n]:\ell_i\geq 0\}|=d_0$, we define the product box
    $\mathcal{B}(\ell_1,\cdots,\ell_n):=\prod_{j=1}^n I_{\ell_j}$. We then define the family of boxes 
    $$
\mathcal{F}:=\left\{B(\ell_1,\cdots,\ell_n):\sum_{j:\ell_j\geq 0}2^{2\ell_j}\leq 8n,\quad |\{i\in[n]:\ell_i\geq 0\}|=d_0\right\}.
    $$

We then check that $\mathcal{F}$ is the desired family of boxes. For $v\in\Lambda_\epsilon[d_0]$, define $X:=(4\epsilon)^{-1}n^{1/2}v\in\mathbb{Z}^n\cap\operatorname{Sparse}(d_0)$. For each $i\in \{i\in[n]:X_i\neq 0\}$ define $\ell_i\in\mathbb{Z}_{\geq 0}$ such that $X_i\in I_{\ell_i}$. If $N_X:=|\{i\in[n]:X_i\neq 0\}|<d_0$, then we arbitrarily choose an additional $d_0-N_x$ coordinates $j$ from $[n]$ and set $\ell_j=0$. For all remaining $n-d_0$ coordinates $k$, set $\ell_k=-1$. Then we claim that $X\in\mathcal{B}(\ell_1,\cdots,\ell_n)$: we only need to check that 
$$
\sum_{j:\ell_j> 0}2^{2(\ell_j-1)}N^2\leq \sum_{j=1}^n X_j^2\leq n/(4\epsilon)^2(\sum_i v_i^2)\leq 2nN^2,
$$ which verifies the conditions stated in the definition of $\mathcal{F}$.

 Then we check that $|\mathcal{F}|\leq 2^{10n}$.
For each $t\geq 0$ there are at most $8n/4^t$ tuples $i\in[d_0]$ such that $\ell_i=t$, and the there are at most $\binom{n}{\leq 8n/4^t}$ choices for the location of these symbols $i$. Then the total number of tuples is at most 
$$
\prod_{t\geq 0} \binom{n}{\leq 8n/4^t}\leq(1/4)^{-4n}\leq 2^{10n}.
$$

Finally we compute the cardinality of $\mathcal{B}(\ell_1,\cdots,\ell_n)\in\mathcal{F}$. Indeed, for such a box in $\mathcal{F}$,
$$
|\mathcal{B}(\ell_1,\cdots,\ell_n)|\leq  N^{d_0} 2^{d_0+\sum_{j\geq 1}\ell_j}\leq (\kappa N)^{d_0}
$$
    where we recall that there are only $d_0$ coordinates of the box which is not the zero set $\{0\}$, and for the second inequality we use the inequality
    $$
\prod_{j:\ell_j\geq 0} 2^{\ell_j}\leq (\frac{\sum_j 2^{2\ell_j}}{d_0})^{d_0}\leq (8n/d_0)^{d_0}\leq (32/c_0^4)^{d_0} 
    $$ by our assumption that $d_0\geq c_0^4n/4$.
\end{proof}

Finally we need the following Fourier replacement lemma, which transfers anticoncentration with respect to $M$ to anticoncentration with respect to the original matrix $A$. 

\begin{lemma}\label{lemma4.145}
For an $l$-tuple of vectors $X_1,\cdots,X_l\in\mathbb{R}^{n}$ with $l\leq n$ and $t>0$, let $p\in\mathbb{N}_+$, $c_0^4nl\leq p\leq nl$ be an integer such that for some given $L>1$ and a fixed $t>2^{-c_0^5n}$ we have 
    \begin{equation}\label{smallballprob1}
\mathbb{P}(\|(MX_1,MX_2,\cdots,MX_l)\|_2\leq t\sqrt{nl}) \leq (Lt)^{p},
\end{equation} where $(MX_1,\cdots,MX_l)\in\mathbb{R}^{nl}$ is the concatenation of $MX_i\in\mathbb{R}^n$, for each $i\in[l]$. Then we have the following estimate:
\begin{equation}\label{transferlemmas}
\mathcal{L}((AX_1,AX_2,\cdots,AX_l),t\sqrt{nl})\leq \left(10\exp(2/c_0^4)Lt\right)^p.
\end{equation}
\end{lemma}

A similar version of Fourier replacement lemma can be found in \cite{campos2025singularity}, Lemma 8.1. The version we prove here has two major differences: first, we can consider a small ball probability estimate of the form $(Lt)^p$ where $p<nl$, that is, the exponent $p$ can be smaller than the full dimension. Second and more importantly, we only need the estimate \eqref{smallballprob1} to hod for a fixed $t>0$ rather than for all $s\geq t$: this will be very useful later. To compensate for this flexibility, we require that $p\in[c_0^4nl,nl]$ and have a worse constant dependence in \eqref{transferlemmas}.

The proof of Lemma \ref{lemma4.145} is deferred to Section \ref{prooffourierreplace}.

\subsection{Implementing the procedure: cardinality of candidate vectors}\label{implementation2} We have made all the technical preparations for applying the ``inversion of randomness’' procedure, and here we implement the computations.

Let $\mathbf{X}=(X_1,\cdots,X_l)$ with each $X_i\in\mathbb{R}^{n}$ be an $\ell$-tuple of vectors. Let $H$ be an $(n-d)\times d$ random matrix with i.i.d. entries of distribution $\tilde{\zeta}Z_\nu$. Define the following two events $\mathcal{A}_1=\mathcal{A}_1(\mathbf{X})$ 
$$
\mathcal{A}_1:=\{ H: \|\left(H(X_1)_{[d]},\cdots,H(X_l)_{[d]}\right)\|_2\leq 6Kn\sqrt{l}
\} 
$$ where $(H(X_1)_{[d]},\cdots,H(X_l)_{[d]})$ is the $dl$-dimensional vector as the concatenation of $H (X_i)_{[d]}$, $i\in[l]$.
Then define the event $\mathcal{A}_2=\mathcal{A}_2(\mathbf{X})$ via
$$
\mathcal{A}_2:=\{H:\|(H^T(X_1)_{[d+1,n]},\cdots, H^T(X_l)_{[d+1,n]})\|_2\leq 6Kn\sqrt{l}\}.
$$
Then we have the following simple observation:
\begin{fact} For any fixed $X_1,\cdots,X_l\in\mathbb{R}^n$ and $\mathcal{A}_1,\mathcal{A}_2$ as above, we have
    $$\mathbb{P}_M(\|MX_i\|_2\leq 6Kn\text{ for each }i\in[l])\leq\mathbb{P}_H(\mathcal{A}_1\cap\mathcal{A}_2).$$
\end{fact}
We define a robust version of the rank of $H$ with respect to the rank of submatrices: for each $k\in[d_e]$ let $\mathcal{E}_k$ denote the event 
$$\begin{aligned}
\mathcal{E}_k:=&\{H:\sigma_{d_e-k}(H_{I\times J})\geq c_0^2\sqrt{n}/16\text{ for each } I\subset[n-d]:|I|\geq n_e,\\&\text{and }\sigma_{d_e-k+1}(H_{I\times J})< c_0^2\sqrt{n}/16\text{ for some } I\subset[n-d]:|I|\geq n_e,\end{aligned}
$$ where we recall that $J=[d_e]$. Then exactly one of the events in $\mathcal{E}_0,\cdots,\mathcal{E}_{d_e}$ should hold.

We take $4\epsilon=(s_{k+1})^{1/4}$, $N=1/(4\epsilon)$, and consider a family of $(N,(u_{k+1}+\epsilon_0)n,\kappa)$ sparse boxes 
$\mathcal{B}_1,\cdots,\mathcal{B}_l$. We assume that for each box $\mathcal{B}_i$, there must exist at least $n_e\geq c_0^2n/2$ coordinates $j\in [d+1,n]$ of $\mathcal{B}_i$ such that the $j$-th coordinate of $B_i$ is not the zero set $\{0\}$. Let $v_1',\cdots,v_l'$ be the vectors satisfying assumptions (1) to (5) of Proposition \ref{lemma4.6stated}. Then we must be able to find boxes $\mathcal{B}_1,\cdots,\mathcal{B}_l$ satisfying the aforementioned properties, such that $(s_{k+1})^{-1/4}\sqrt{n}\cdot v_j'\in\mathcal{B}_i$ for each $i$, by Lemma \ref{howlarageisthebox}.

In the following, we fix a collection of $(N,(u_{k+1}+\epsilon_0)n,\kappa)$ sparse boxes $\mathcal{B}_1,\cdots,\mathcal{B}_l$ as in the previous paragraph. We define the set of typical vectors associated to these boxes via
\begin{equation}
    \begin{aligned}T_0=T_0(\mathcal{B}_1\cdots,\mathcal{B}_l)&:=\{X_i\in\mathcal{B}_i\forall i\in[l],\quad  \|(X_i)_{[d]}\|_2\geq \frac{s_k\sqrt{n}}{2\cdot s_{k+1}^{1/4}}\forall i\in[l] ,\\&
    (X_1)_{[d]},\cdots,(X_l)_{[d]}\text{ are }\frac{1}{4}-\text{orthogonal by the norm on }\mathbb{R}^d,
    \\&\operatorname{Span}((X_1)_{[d]},\cdots,(X_l)_{[d]})\cap\mathbb{S}^{d-1}\subset\operatorname{Incomp}(\epsilon_0n,s_k/2)\}.\end{aligned}
\end{equation}

The set $T_0$ contains the vector pairs in the product box that can arise from the approximation procedure of Proposition \ref{lemma4.6stated}:

\begin{fact}\label{fact4.132}
Let $v_1',\cdots,v_l'$ satisfy the assumptions (1) to (5) of Proposition \ref{lemma4.6stated}, and let $v_i'\in\mathcal{B}_i$ for one $(N,(u_{k+1}+\epsilon_0)n,\kappa)$ sparse box box as previously defined.

If we set  $X_i'=s_{k+1}^{-1/4}\sqrt{n}\cdot v_i'$ for each $i\in[l]$, we have $(X_1',\cdots,X_l')\in T_0$ and that 
$$
\{A:\|Av_j'\|_2\leq 6K(s_{k+1})^{1/4}\sqrt{n}\forall j\in[l]\}=\{A:\|AX_j'\|_2\leq 6Kn\forall j\in[l]\}.
$$    
\end{fact}
We also estimate the essential LCD of vectors from this approximation:
\begin{fact}\label{fact4.15approx}
Let $U$ be the $d\times l$ matrix with columns $(X_1')_{[d]},\cdots,(X_l')_{[d]}$ for the vectors specified in Fact \ref{fact4.132}. Then for any $L_0>0$ and $\alpha_0=s_k/2$ we have
$$D_{L_0,s_k/2}(U)\geq \frac{\sqrt{\epsilon_0}s_{k+1}^{1/4}}{4}.$$
\end{fact}
\begin{proof}
    By definition, $\|X_i'\|_2\leq \frac{2\sqrt{n}}{s_{k+1}^{1/4}}$ for each $i\in[l]$ and $(X_1')_{[d]},\cdots,(X_l')_{[d]}$ are $\frac{1}{4}$-orthogonal. Then we get that $\|U\theta\|_2\leq \sqrt{\epsilon_0n}/2$ for all $\theta\in\mathbb{R}^l$ with $\|\theta\|_2\leq \frac{\sqrt{\epsilon_0}s_{k+1}^{1/4}}{4}$. Then by Lemma \ref{lcdimcomps}, we have 
    $D_{L_0,s_k/2}(U)\geq \frac{\sqrt{\epsilon_0}s_{k+1}^{1/4}}{4}.$
\end{proof}

Then we make a refinement of the set of typical vectors: for any fixed $L_0>0$, define
\begin{equation}\label{settypicalvectors}
T=T(\mathcal{B}_1,\cdots,\mathcal{B}_l):=T_0\cap\left\{X_i\in\mathcal{B}_i\forall i\in[l]: D_{L_0,s_k/2}(U)\geq\frac{\sqrt{\epsilon_0}s_{k+1}^{1/4}}{4}\right\},
\end{equation} where $U$ is the $d\times l$ matrix with columns given by $(X_i)_{[d]},i\in[l]$. 

A crucial observation is that the definition of $T$ does not impose any restriction on the value of coordinates of $\mathcal{B}_i$ in $[d+1,n]$, and all restrictions are on the first $d$ coordinates.

Now we implement the main computation in the inversion of randomness argument:

\begin{Proposition}\label{prop1stmoment} We define the following function for $\mathbf{X}=(X_1,\cdots,X_l)\in(\mathbb{Z}^n)^l$:
$$f(\mathbf{X}):=\mathbb{P}_M(\|(MX_1,\cdots,MX_l)\|_2\leq 6Kn\sqrt{l})\mathbf{1}(\mathbf{X}\in T).$$
Then for fixed $s_k,s_{k+1}>0$, whenever $n$ is sufficiently large, the expectation of $f(\mathbf{X})$ can be estimated as follows:
$$
\mathbb{E}_{(X_1,\cdots,X_l)\in\mathcal{B}_1\times\cdots\times\mathcal{B}_l}  f(\mathbf{X})\leq  (\frac{R\cdot s_{k+1}^{1/4}}{s_k})^{(u_k+c_0^4/10)nl}
$$ for a constant $R>0$ depending only on $c_0,p,\nu$.

\end{Proposition}

\begin{proof}
We take a conditioning on $H$ and get 
\begin{equation}
f(\mathbf{X})\leq\sum_{{k'}=0}^{d_e}\mathbb{P}_H(\mathcal{A}_2\mid\mathcal{A}_1\cap\mathcal{E}_{k'})\mathbb{P}_H(\mathcal{A}_1\cap\mathcal{E}_{k'})\mathbf{1}(\mathbf{X}\in T).
\end{equation}

Set $\alpha'=2^{-32}B^{-4}c_0^3$. Then for $k'\leq \alpha'n$,
for any $\mathbf{X}\in T$ we can apply Theorem \ref{compressibletheorem12.234567} to upper bound $\mathbb{P}_H(\mathcal{A}_1\cap\mathcal{E}_{k'})$ by
\begin{equation}\label{firstrounds}
\mathbb{P}_H(\mathcal{A}_1\cap\mathcal{E}_{k'})\leq 2^n
\exp(-c_0^3n{k'}/48)(\frac{R
\cdot s_{k+1}^{1/4}}{s_k})^{(n-d)l},
\end{equation}
whee we apply Theorem \ref{compressibletheorem12.234567} with $\overline{X}_i=\frac{c_0}{6K}X_i$ and $\alpha=s_{k}/2$, so that we have $x_i\geq \frac{12K\cdot s_{k+1}^{1/4}}{c_0s_k}$ for each $i\in[l]$. We also check that $D_{L,s_k/2}(\frac{c_0}{32\sqrt{n}}\overline{X}_1,\cdots,\frac{c_0}{32\sqrt{n}}\overline{X}_l)\geq 256B^2\sqrt{l}$ where the constant $L$ was given in Theorem \ref{compressibletheorem12.234567}, and this estimate holds whenever $n$ is sufficiently large by the assumption $l=o(n)$. The constant $R>0$ in \eqref{firstrounds} depends only on $c_0,\nu,p$.

Moreover, we can compute using Theorem \ref{compressibletheorem12.234567} that, for large values of ${k'}$:
$$\begin{aligned}
\sum_{{k'}\geq \alpha'n}\mathbb{P}_H(\mathcal{A}_1\cap\mathcal{E}_{k'})&\leq\mathbb{P}_H(\mathcal{A}_1\cap\{\sigma_{d_e-\alpha'n}(H_{I\times J})\leq c_0^2\sqrt{n}/16\text{ for some }|I|=n_e\})\\&\leq\exp(-2^{-34}c_0^6B^{-4}n^2).\end{aligned}
$$

Now we take the inversion step with randomness generated from the coordinates. We define a family of functions $g_{k'}(\mathbf{X}):=\mathbb{P}_H(\mathcal{A}_2\mid\mathcal{A}_1\cap\mathcal{E}_{k'})$, ${k'}\in[\alpha'n]$. Then we have
$$
\mathbb{E}_\mathbf{X}[g_{k'}(\mathbf{X})]=\mathbb{E}_{\mathbf{X}_{[d]}}\mathbb{E}_H[\mathbb{E}_{\mathbf{X}_{[d+1,n]}}\mathbf{1}[\mathcal{A}_2]\mid\mathcal{A}_1\cap\mathcal{E}_{k'}],
$$  where we use the symbol $\mathbf{X}_{[d]}$ to mean the vectors $(X_1)_{[d]},\cdots,(X_l)_{[d]}$ and similarly for $\mathbf{X}_{[d+1,n]}$. Now we consider some fixed $H\in\mathcal{A}_1\cap\mathcal{E}_{k'}$, so that $\sigma_{d_e-{k'}}(H_{I\times J})\geq c_0^2\sqrt{n}/16$ for any interval $I\subset[n-d]$ with $|I|\geq n_e=c_0^2n/2$. Then we apply both Corollary \ref{corollary4.11} and the tensorization lemma, Lemma \ref{tensorizationlemma} to see that for this fixed $H$, 
\begin{equation}\label{secondexpectationfirst}
\mathbb{E}_{\mathbf{X}_{[d+1,n]}}\mathbf{1}[\mathcal{A}_2]\leq (R\cdot s_{k+1}^{1/4})^{(c_0^4n/8-{k'})l}
\end{equation} for a constant $R>0$ that only depends on $c_0$.
More precisely, we use the independence of the coordinates of $X_1,\cdots,X_l$, and for each fixed $i\in[l]$ we have that 
\begin{equation}\label{inversionone}
\mathbb{P}_{{(X_i)}_{[d+1,n]}}(\| H^T(X_i)_{[d+1,n]}\|_2\leq 6Kn)\leq \mathcal{L}((H_{I_i\times J})
^T(X_i)_{i\in I_i},6Kn),\end{equation}
where we let $I_i$ denote the coordinates of the box $\mathcal{B}_i$ in $[d+1,n]$ that are not the zero set $\{0\}$ and hence are well-anticoncentrated, and we have $|I_i|\geq c_0^2n/2$ for each $i$ by assumption, so that $\sigma_{d_e-{k'}}(H_{I_i\times J})\geq c_0^2\sqrt{n}/16$ by our assumption on $H$. Then we apply  Corollary \ref{corollary4.11} to upper bound \eqref{inversionone}, and then use Lemma \ref{tensorizationlemma} to tensorize all the $l$ dimensions in \eqref{secondexpectationfirst}.

Now we combine estimates \eqref{firstrounds} and \eqref{secondexpectationfirst} to deduce that 
$$\begin{aligned}
\mathbb{E}_\mathbf{X} f(\mathbf{X})&\leq  2^n\sum_{{k'}=1}^{\alpha'n}(\frac{R\cdot s_{k+1}^{1/4}}{s_k})^{
((u_k+c_0^4/10)n-{k'})l}\exp(-c_0^3n{k'}/48)+\exp(-2^{-34}c_0^6B^{-4}n^2)\\&\leq (\frac{R\cdot s_{k+1}^{1/4}}{s_k})^{(u_k+c_0^4/10)nl},
\end{aligned}$$
where we use $u_1\geq c_0^4n$, and throughout the lines $R>0$ is a constant depending only on $c_0,\nu,p$ but whose value changes from line to line. This completes the proof.
\end{proof}

We also need a crude upper bound of $f$ for any $\mathbf{X}\in T$: this follows by using only the randomness in the first $d$ columns of $M$.
\begin{lemma}\label{basecase}
    For any $\mathbf{X}\in T$ we have that
    $$
f(\mathbf{X})\leq (\frac{R\cdot s_{k+1}^{1/4}}{s_k})^{(u_k-c_0^4/40)nl}.
    $$ for some $R>0$ depending only on $c_0,\nu,p$.
\end{lemma}

    \begin{proof}
        We bound $f(\mathbf{X})$ by $\mathbb{P}_H(\mathcal{A}_1)$, and use the (standard) inverse Littlewood-Offord inequality, Theorem \ref{littlewoodincomp}, to compute the small ball probability as in the previous proof.
    \end{proof}

Then we return to the original matrix $A$:

\begin{corollary}\label{pilethingsupfinalnew}
For a fixed $s_{k}>0$ we can choose $s_{k+1}>0$ depending only on $s_k,B,c_0$ such that, whenever $n$ is sufficiently large,
$$\mathbb{P}\left(\text{There exists an $l$-tuple } (X_1,\cdots,X_l)\in T \text{ such that } \|AX_i\|_2\leq 6Kn \forall i\in[l]\right)\leq  2^{-100nl}.    $$
\end{corollary}

\begin{proof}
  We consider the following partition of vectors in $T$ for each $k'=0,1,\cdots,d_e$:
$$\begin{aligned}&E_{k'}:=\{\mathbf{X}\in  T:
\mathbb{P}_M(\|(MX_1,\cdots,MX_l)\|_2\leq 6Kn\sqrt{l})\geq (\frac{R\cdot s_{k+1}^{1/4}}{s_k})^{(u_k+c_0^4/10)nl-k'l}\}
\end{aligned}$$ for the constant $R>0$ given in Proposition \ref{prop1stmoment}, 
and we denote by $E_{-1}=T\subseteq(\mathcal{B}_1\times\cdots\times \mathcal{B}_l)$.
These subsets are nested in the sense that $E_{-1}\supset E_0\supset E_1\supset E_2\supset\cdots\supset E_{d_e}$. We also have $E_{k'}=\emptyset$ for all $k'\geq c_0^4n/8$ by Lemma \ref{basecase}. 

Suppose that $(X_1,\cdots,X_l)\in E_{k'}\setminus E_{k'+1}$ for some $k'\geq -1$, then by definition of $E_{k'+1}$ we have $\mathbb{P}_M(\|(MX_1,\cdots,MX_l)\|_2\leq 6Kn\sqrt{l})\leq (\frac{R\cdot s_{k+1}^{1/4}}{s_k})^{(u_k+c_0^4/10)nl-(k'+1)l}$, so by Lemma \ref{lemma4.145}, 
\begin{equation}\label{singlepairestimate}\mathbb{P}_A(\|(AX_1,\cdots,AX_l)\|_2\leq 6Kn\sqrt{l})\leq (\frac{C'R\cdot s_{k+1}^{1/4}}{s_k})^{(u_k+c_0^4/10)nl-(k'+1)l},\end{equation} 
where $C'=10\exp(2/c_0^4).$

Next, we compute the cardinality of $E_{k'}$. Applying Markov's inequality to the conclusion of Proposition \ref{prop1stmoment}, we get that
$$\begin{aligned}
&\mathbb{P}_\mathbf{X}(\mathbf{X}\in E_{k'})\\&
= \mathbb{P}_{\mathbf{X}}(\mathbb{P}_M(\|(MX_1,\cdots,MX_l)\|_2\leq 6Kn\sqrt{l})\geq (\frac{R\cdot s_{k+1}^{1/4}}{s_k})^{(u_k+c_0^4/10)nl-k'l})\mathbf{1}(\mathbf{X}\in T)\\&\leq (\frac{s_k}{R\cdot s_{k+1}^{1/4}})^{(u_k+c_0^4/10)nl-k'l}\mathbb{E}_\mathbf{X}f(\mathbf{X})\leq  (\frac{R\cdot s_{k+1}^{1/4}}{s_k})^{k'l}.
\end{aligned}
$$ 
Recall by Lemma \ref{howlarageisthebox} that each box $\mathcal{B}_i$ has cardinality at most $|\mathcal{B}_i|\leq (\frac{\kappa} {s_{k+1}^{1/4}})^{(u_{k+1}+\epsilon_0)n}$, then the cardinality of $E_{k'}$ is upper bounded by 
\begin{equation}\label{thisbound}
|E_{k'}|\leq (\frac{\kappa} {s_{k+1}^{1/4}})^{(u_{k+1}+\epsilon_0)nl}(\frac{R\cdot s_{k+1}^{1/4}}{s_k})^{k'l},
\end{equation} so we can compute the following probability by combining \eqref{thisbound} with \eqref{singlepairestimate}:
$$\begin{aligned}&
\mathbb{P}_A(\exists\mathbf{X}\in E_{k'}\setminus E_{k'+1}:\|(AX_1,\cdots,AX_l)\|_2\leq 6Kn\sqrt{l})\\&\leq 
(\frac{C'R}{s_k})^{(u_k+c_0^4/10)nl}\cdot s_{k+1}^{c_0^4nl/160}
\end{aligned}$$ recalling that $u_k+c_0^4/10= u_{k+1}+\epsilon_0+c_0^4/20$. Then for fixed $s_k$, whenever $s_{k+1}>0$ is chosen to be sufficiently small, the last expression can be smaller than $2^{-200nl}$.
Finally, we take the sum of the individual probabilities for $k'=0,1,\cdots,d_e$ to conclude the proof.
\end{proof}

\subsection{Completing the proof via an iterative bootstrap argument}\label{finalbootstrap}

Now we can complete the proof of Proposition \ref{mainpropsection4}. We first prove Proposition \ref{proposition4.22}.  

To highlight the main ideas, we first outline the main arguments in the proof of Proposition \ref{proposition4.22}.
We first apply Lemma \ref{lemma2.4} with the set $W_{[d]}$ defined in Fact \ref{notation4.72}. This implies that, unless we are done, there exists an $l$-tuple of vectors in $F\cap\ker A$ whose restrictions to $[d]$ are almost orthogonal and satisfy the required incompressibility and support conditions. Then we use Proposition \ref{lemma4.6stated} to replace these vectors by approximations $v_1',\cdots,v_l'$ that lie in a controlled sparse family while maintaining the structural properties demanded by the inversion argument. Then we apply Lemma \ref{howlarageisthebox} to place these approximations into union of sparse boxes that have manageable cardinalities. Next, we apply  Fact \ref{fact4.132} and \ref{fact4.15approx} to see that after a suitable rescaling, the resulting tuple should lie in a set of typical vectors $T$ introduced in \eqref{settypicalvectors}, but then Corollary \ref{pilethingsupfinalnew} rules out, with very high probability, the event that such a tuple should simultaneously satisfy all the required smallness conditions on $AX_1,\cdots,AX_l$. This contradiction implies that alternative (1) of Lemma \ref{lemma2.4} cannot occur too often and thus Fact \ref{fact4.44} yields the desired subspace $F_1$.

\begin{proof}[\proofname\ of Proposition \ref{proposition4.22}]
By Fact \ref{fact4.44}, we only need to find the subspace $F_{1i}$ as claimed in Fact \ref{fact4.44} for each $i\in[\lfloor 1/c_0\rfloor+1]$. Without loss of generality we  consider $i=\lfloor 1/c_0\rfloor+1$. We work on the event $\|A\|_{HS}\leq 2Kn$, which has probability $1-\exp(-\Omega(n^2))$.

By Fact \ref{notation4.72}, suppose that we can apply alternative (1) of Lemma \ref{lemma2.4} for $l$ times and 
get an $l$-tuple of vectors $v_1,\cdots,v_l\in\ker A\cap \mathbb{S}^{n-1}$ satisfying assumptions (1) to (4) of Fact \ref{notation4.72}, then by Proposition \ref{lemma4.6stated} we can find vectors $v_1',\cdots,v_l'$ satisfying assumptions (1) to (5) of Proposition \ref{lemma4.6stated}, whenever $s_{k+1}$ is small enough satisfying \eqref{relationsk+1}. By Lemma \ref{howlarageisthebox}, these approximate vectors can be contained in at most $|\mathcal{F}|^l=2^{10nl}$ such $(1/s_{k+1}^{1/4},(u_{k+1}+\epsilon_0)n,\kappa)$ sparse boxes. According to Fact \ref{fact4.132} and Fact \ref{fact4.15approx}, a scaled version $X_1,\cdots,X_l$ of the vectors $v_1',\cdots,v_l'$ must lie in the set $T$ of the typical vectors defined in \eqref{settypicalvectors} and satisfy $\|AX_i\|_2\leq 6Kn\forall i\in[l]$, but the probability for the latter event to happen for some $l$-tuple in $T$ is at most $2^{-100nl}$ as long as $s_{k+1}>0$ is set sufficiently small, using Corollary \ref{pilethingsupfinalnew}.

Thus, the probability that we can apply alternative (1) of Lemma \ref{lemma2.4} for $l$ times is at most $2^{10nl}\cdot 2^{-100nl}=2^{-90nl}$, so that by alternative (2) of Lemma \ref{lemma2.4}, with probability at least $1-2^{-90nl}$, we can find a subspace $F_{1i}$ of codimension at most $l$ satisfying Fact \ref{fact4.44}.
\end{proof}

Now the proof of Proposition \ref{mainpropsection4} is immediate:
\begin{proof}[\proofname\ of Proposition \ref{mainpropsection4}] This follows by applying Proposition \ref{proposition4.22} finitely many times (the number of applications is bounded by $40/c_0^4$), where for the given $s_{k}>0$ we find $s_{k+1}>0$ sufficiently small by Proposition \ref{proposition4.22}, and continue the process. 
\end{proof}

\section{Incompressible vectors: threshold and cardinality of the net}\label{thresholdcardinality}

In this section, we study incompressible vectors in $\ker A$. The goal is to find a subspace of $\ker A$ such that all the unit vectors in this subspace have large enough essential LCD. 

The main result of this section is the following proposition:

\begin{Proposition}\label{propositionlcds}
Let $A_n\in\operatorname{Sym}_n(\zeta)$ with $\zeta\in\Gamma_B$.
Let $F_1'\subset\ker A$ be a linear subspace such that $F_1'\cap\mathbb{S}^{n-1}\subset \operatorname{Incomp}((1-c_0^2/10)n,\tau_{c_0})$ for some $\tau_{c_0}\in(0,1)$.

Then we can find constants $\rho\in(0,1)$, $\gamma\in(0,\tau_{c_0})$ and $C_{\rho,\gamma,c_0}>0$ such that, whenever \begin{equation}\label{rangeofls}1
\leq l\leq c_{c_0,B}\sqrt{n}\end{equation} for some $c_{c_0,B}>0$, then the following holds with probability at least $1-\exp(-\Omega(ln))$:

There exists a linear subspace $F_2'$ of $F_1'$ of codimension at most $\frac{1}{4}l$ such that, for any $v\in F_2'\cap\mathbb{S}^{n-1}$, we have that (recalling the definition of $D_{\rho,\gamma}(v)$ from Definition \ref{essentiallcdone}):
$$
D_{\rho,\gamma}(v)\geq \exp(\frac{C_{\rho,\gamma,c_0} n}{4l}).
$$
    
\end{Proposition}

The source for the restriction $l\leq C_{c_0,B}\sqrt{n}$ in \eqref{rangeofls} is  explained in Remark \ref{remark5.12onlcd}.

Our strategy to prove Proposition \ref{propositionlcds} is to work with the contrapositive statement, and try to extract an $l$-tuple of almost orthogonal vectors in $\ker A$ with essential LCD in the intermediate range $[c\sqrt{n},\exp(Cn/l)]$. Then one tries to use the inverse Littlewood-Offord theorem similar to Theorem \ref{littlewoodincomp}. However, a few major technical challenges will arise:
\begin{enumerate}
    \item The essential LCD of these $l$ vectors can vary significantly in $[c\sqrt{n},\exp(Cn/l)]$, leading to complications as we approximate them by vectors from the integer lattice. To solve this problem, we will re-scale the vectors based on their essential LCD, in a way that the vector $v_i\in\mathbb{S}^{n-1}$ having essential LCD roughly $t_i$ is rescaled to have $\ell^2$ norm $1/t_i$, and then they can be approximated with a uniform error. 
    \item Although we may have a control for the essential LCD of the re-scaled versions of $v_1,\cdots,v_l$ individually, the essential LCD of their linear span does not follow from this information. The method to solve this is to use Lemma \ref{lemma2.4}, case (1)(b). This application of Lemma \ref{lemma2.4} leads to the restriction $l=O(\sqrt{n})$, but is currently the only available method that achieves this goal.
    \item For the random symmetric matrix $A$, we shall apply the inversion of randomness procedure in \cite{campos2025singularity}, which requires us to restrict $v_1,\cdots,v_l$ to a small fraction of $d\leq c_0^2n$ coordinates. This leads to an issue that $\operatorname{span}((v_1)_{[d]},\cdots,(v_l)_{[d]})\cap\mathbb{S}^{d-1}$ may not consist entirely of incompressible vectors, but the no-gaps delocalization result of Proposition \ref{mainpropsection4} successfully solves this problem. Further, we shall transfer the fact that the essential LCD of $(v_i)_{[n]}$ is small to the smallness of essential LCD of $(v_i)_{[d]}$, and use a net with small essential LCD for $(v_i)_{[d]}$ to approximate.
    \item In \cite{campos2022singularity}, Lemma 7.4 (see also \cite{han2025repeated}) they used the fact that a uniformly sampled vector from an integer lattice has large essential LCD with super-exponentially high probability. This lemma is not applicable in the large deviation regime here because by this lemma, the fraction of vectors with small essential LCD is at most smaller than any exponentially decaying function in $n$ but still larger than $\exp(-\Omega(nl))$ for any $l$ growing in $n$. We  introduce an alternative version of typical vectors $\hat{T}$ in \eqref{typicalhatt} and use randomly generated coordinates only in the last $n-d$ dimensions.
\end{enumerate}

In this section we will formalize all these ideas and prove Proposition \ref{propositionlcds}.
This section is organized as follows. In Section \ref{approximationstep2}, we reduce Proposition \ref{propositionlcds} to a simpler problem, and approximate the candidate vectors in $\ker A$ by those vectors from an integer lattice. In Section \ref{implementation2.1}, we recall some preparatory technical tools. In Section \ref{implementation2.2}, we implement the ``inversion of randomness’' procedure similarly to Section \ref{implementation2}. Finally, in Section \ref{implementation3} we conclude the proof of Proposition \ref{propositionlcds}.

\subsection{The approximation procedure}\label{approximationstep2}
 Similarly to the previous section, we let $c_0\in(0,1)$ be a sufficiently small constant that will be fixed at the end of this section (say, we take $c_0=2^{-100}$), and let $d\in\mathbb{N}$ be such that $c_0^2n/4\leq d\leq c_0^2n$.

The following fact enables us to consider the essential LCD of restricted vectors:

\begin{fact}\label{fact3.4}
    Let $v\in\mathbb{S}^{n-1}\cap \operatorname{Incomp}((1-c_0^2/10)n,\tau_{c_0})$. This time we define a family of intervals $J_{i}=[(i-1)d+1,id]$ for all $1\leq i\leq \lfloor n/d\rfloor$ and $J_{\lfloor n/d\rfloor+1}=[n-d,n]$.  Then,
    \begin{enumerate}
    \item For each $i\in[\lfloor n/d\rfloor+1]$ we have that  $v_{J_i}\in \operatorname{Incomp}(c_0^2n/8,\tau_{c_0})$ and in particular $\|v_{J_i}\|_2\geq \tau_{c_0}>0.$ 
    
    \item Suppose that $D_{\rho,\gamma}(v)\leq N$ for some $\rho\in(0,1)$, $\gamma\in(0, \tau_{c_0})$ and some $N>0$, where $D_{\rho,\gamma}(v)$ is defined in Definition \ref{essentiallcdone}. Then for each $i\in[\lfloor n/d\rfloor+1]$,
    $$\operatorname{dist}(\theta v_{J_i},\mathbb{Z}^d)\leq \sqrt{\rho n}\quad  \text{for some }\sqrt{n}/4\leq\theta\leq N.$$
    
    \end{enumerate}
\end{fact}

\begin{proof}
For claim (1), suppose that for some $i\in[\lfloor n/d\rfloor+1]$, $v_{J_i}\in\operatorname{Comp}(c_0^2n/8,\tau_{c_0})$. Then since $d\geq c_0^2n/4$ we have $v\in\operatorname{Comp}((1-c_0^2/10)n,\tau_{c_0})$, a contradiction.

For claim (2), by Lemma \ref{lcdimcomps}, we have $\|\theta v\|_\mathbb{Z}\geq\gamma\|\theta v\|_2$ whenever $\theta\leq \sqrt{(1-c_0^2/10)n}/2$ so that whenever $\theta\leq\sqrt{n}/4$. Then $D_{\rho,\gamma}(v)\leq N$ implies that $\|\theta v\|_\mathbb{Z}\leq\sqrt{\rho n}$ for some $\sqrt{n}/4\leq |\theta|\leq N$ and $\|\theta v\|_\mathbb{Z}\leq\gamma\|\theta v\|_2$. Then for each $i$, $\|\theta v_{J_i}\|_\mathbb{Z}\leq \sqrt{\rho n}$ for some $\sqrt{n}/4\leq\theta\leq N$.
\end{proof}

 We define the following set of vectors, for some $L\geq\sqrt{l}$:
\begin{equation}\label{anumberofsubsets}\begin{aligned}
W_{1,i}:=&\{v\in \mathbb{R}^n:\frac{\tau_{c_0}}{4}\sqrt{n}\leq \|v_{[d]}\|_2\leq \exp(\frac{\rho n}{4L^2}),\\&\operatorname{dist}(v_{[d]},\mathbb{Z}^{d})\leq\sqrt{\rho n}\},
\end{aligned}\end{equation}
where $\rho\in(0,1)$ will be fixed at the end of proof. The parameter $l\in\mathbb{N}$ stands for the corank of our matrix $A\in\operatorname{Sym}_n(\zeta)$. We assume that $\tau_{c_0}$ is sufficiently small such that $\tau_{c_0}\leq c_0/8$.

Then we estimate the cardinality of a net approximating vector $(v_i)_{[d]}$: the following lemma is a reformulation of \cite{rudelson2024large}, Lemma 5.2. One should be careful that we are only approximating the first $d$ coordinates of the vectors $v_i$ in the following lemma:
\begin{lemma}\label{lemma5.2nts}
    Consider $\mathbf{t}=(t_1,\cdots,t_l)$ a tuple satisfying that $t_j\in[\frac{\tau_{c_0}}{4}\sqrt{n},\exp(\frac{\rho n}{4L^2})]$ for each $j\in[l]$. Let $\delta>0$ satisfy that $\delta\leq\rho$ and $\delta^{-1}\in\mathbb{N}$. Consider $\mathcal{N}_\mathbf{t}\in(\delta \mathbb{Z}^d)^l$ be the set of all $l$-tuples of vectors $\bar{u}_1,\cdots,\bar{u}_l\in\mathbb{R}^d$ such that  
$$ \operatorname{dist}(\bar{u}_j,\mathbb{Z}^d)\leq2\sqrt{\rho n},
$$ and 
$$
\|\bar{u}_j\|_2\in[\frac{1}{2}t_j,4t_j]\quad\forall j\in[l].
$$
Then 
$$
|\mathcal{N}_\mathbf{t}|\leq (\frac{C_{c_0}\sqrt{\rho}}{\delta})^{ld}(\prod_{j=1}^{l}\frac{t_j}{\sqrt{d}})^d
$$ for a constant $C_{c_0}>0$ depending only on $c_0$, $\tau_{c_0}$ and $B$.
\end{lemma}

\begin{proof}
We shall rewrite the proof in \cite{rudelson2024large} since all the parameters have changed.  We define $\mathcal{M}_j=\mathbb{Z}^d\cap 2t_jB_2^d$, where $B_2^d$ is the unit ball in $\mathbb{R}^d$, then by Lemma \ref{numberofintegralppoints},
$$
|\mathcal{M}_j|\leq (2+\frac{Ct_j}{\sqrt{d}
})^d\leq (C'_{c_0})^d (\frac{t_j}{\sqrt{d}})^d
$$ for some $C'_{c_0}>0$ depending only on $\tau_{c_0}$ and $c_0$, using the lower bound $t_j\geq\frac{\tau_{c_0}}{4}\sqrt{n}\geq  \frac{\tau_{c_0}}{4}\sqrt{d/c_0^2}$
. 

We then define $\mathcal{M}=\delta\mathbb{Z}^d\cap \sqrt{\rho n}B_2^d$. Then again by Lemma \ref{numberofintegralppoints}, we have
$$
|\mathcal{M}|\leq (\frac{2C\sqrt{\rho}/c_0}{\delta})^d
$$ since $n\leq 4d/c_0^2$. Then we define $\mathcal{N}_j=\mathcal{M}_j+\mathcal{M}\subset\delta\mathbb{Z}^d$, and finally set $\mathcal{N}_\mathbf{t}=\prod_{j=1}^l\mathcal{N}_j$. Then $\mathcal{N}_\mathbf{t}$ satisfies 
$$
|\mathcal{N}_\mathbf{t}|\leq\prod_{j=1}^l (\frac{2C\sqrt{\rho}/c_0}{\delta})^d\cdot (C'_{c_0})^d (\frac{t_j}{\sqrt{d}})^d\leq (\frac{2C_{c_0}''\sqrt{\rho}}{\delta})^{ld}(\prod_{j=1}^l\frac{t_j}{\sqrt{d}})^d
$$ for another constant $C_{c_0}''>0$ depending only on $c_0,\tau_{c_0}$ and $B$.
\end{proof}

We now consider a dyadic partitioning of the length of the vectors $(v_j)_{[d]}.$
\begin{Definition}
\label{definitionforbigwt}    
For a given vector $\mathbf{t}=(t_1,\cdots,t_l)$ with each $t_i\in[\frac{\tau_{c_0}}{4}\sqrt{n},\exp(\frac{\rho n}{4L^2})]$ for each $i\in[l]$, we define $W_\mathbf{t}$ as the set of $l$-tuples of vectors $v_1,\cdots,v_l\in\mathbb{R}^n$ satisfying 
\begin{enumerate}
    \item $\|(v_i)_{[d]}\|_2\in[t_i,2t_i]\text{ for each } i\in[l]$, and $\|v_i\|_2\leq 2t_i/\tau_{c_0}$ for each $i\in[l]$.
    \item  The vectors $(v_1)_{[d]},\cdots,(v_l)_{[d]}$ are $\frac{1}{8}$-almost orthogonal.
    \item $\operatorname{dist}((v_i)_{[d]},\mathbb{Z}^d)\leq\sqrt{\rho n}$ for each $i\in[l]$.
    \item  $\operatorname{span}((v_1)_{[d]},\cdots,(v_l)_{[d]})\cap\mathbb{S}^{d-1}\subset\operatorname{Incomp}(c_0^2n/8,\tau_{c_0})$.
    \item The $d\times l$ matrix $V$ with columns $(v_1)_{[d]},\cdots,(v_l)_{[d]}$ satisfy that $\operatorname{dist}(V\theta,\mathbb{Z}^d)\geq \sqrt{\rho n} $ whenever $\|\theta\|_2\leq\frac{1}{20\sqrt{l}}$ and $\|V\theta\|_2\geq \frac{\tau_{c_0}}{4}\sqrt{n}$.
    
\end{enumerate}

\end{Definition}

Then we prove a crucial approximation result similar to Proposition \ref{lemma4.6stated}. 

\begin{Proposition}\label{proposition5.42} Let $l\leq cn$, $\delta\in(0,\rho)$ with $\delta^{-1}\in\mathbb{N}$. Let $A$ be an $n\times n$ matrix with $\|A\|_{HS}\leq Kn$. For any sequence of vectors $(v_1,\cdots,v_l)\in W_\mathbf{t}\cap\ker A$, where $\mathbf{t}$ is as in Definition \ref{definitionforbigwt}, we can find a sequence of vectors $(u_1,\cdots,u_l)\in\delta \mathbb{Z}^n$ satisfying the following
\begin{enumerate}
    \item $((u_1)_{[d]},\cdots,(u_l)_{[d]})\in \mathcal{N}_\mathbf{t}$ and  $\|u_i-v_i\|_\infty\leq\delta$ for all $i\in[l]$.
    \item Let $U$ and $V$ be $d\times l$ matrices with columns $(u_1)_{[d]},\cdots,(u_l)_{[d]}$ and $(v_1)_{[d]},\cdots,(v_l)_{[d]}$. Then $\|U-V\|_{op}\leq C_{\ref{operatornorms}}\delta\sqrt{d}$.
    \item We have $\|u_i\|_2\leq 2t_i/\tau_{c_0}+C_{\ref{operatornorms}}\delta\sqrt{n}\leq 3t_i/\tau_{c_0}$ for each $i\in[l]$.
    \item The system $(u_1)_{[d]},\cdots,(u_l)_{[d]}$ is $\frac{1}{4}$-almost orthogonal.
    \item Let $U$ be as in (2), then $\operatorname{dist}(U\theta,\mathbb{Z}^d)\geq\frac{1}{2}\sqrt{\rho n}$ whenever $\|\theta\|_2\leq\frac{1}{20\sqrt{l}}$ and $\|U\theta\|_2\geq 2\tau_{c_0}\sqrt{n}$.
    \item $\operatorname{span}((u_1)_{[d]},\cdots,(u_l)_{[d]})\cap\mathbb{S}^{d-1}\subset \operatorname{Incomp}(c_0^2n/8,\tau_{c_0}/2)$.
    \item $\|Au_j\|_2\leq 2K\delta n$ for all $j\in[l]$,
    
\end{enumerate}
whenever $\delta>0$ is chosen sufficiently small relative to the other constants $\rho,c_0,\tau_{c_0}>0$.
\end{Proposition}
Note that $\delta>0$ is small relative to $\rho$: we will later fix $\rho$ and eventually fix $\delta$ at the end.
\begin{proof}
For a given pair $(v_1,\cdots,v_l)\in W_\mathbf{t}$, choose $(v_1',\cdots,v_l')\in\delta\mathbb{Z}^n$ such that    
$$
v_j'\in v_j+\delta [0,1]^n\text{ for all } j\in[l].
$$ We then define a family of independent random variables $\epsilon_{ij},i\in[n],j\in[l]$ via
$$
\mathbb{P}(\epsilon_{ij}=v_j'(i)-v_j(i))=1-\frac{v_j(i)-v_j'(i)}{\delta},
$$
$$
\mathbb{P}(\epsilon_{ij}=v_j'(i)-v_j(i)+\delta)=\frac{v_j(i)-v_j'(i)}{\delta}.
$$ Then $\mathbb{E}\epsilon_{ij}=0$ and $|\epsilon_{ij}|\leq\delta$.
 Then we define random approximations for each $j\in[l]$:
$$
u_j=v_j+\sum_{i=1}^n\epsilon_{ij}e_i\in\delta\mathbb{Z}^n
$$ where $e_1,\cdots,e_n$ are standard coordinate vectors. Then $\mathbb{E}u_j=v_j$ and $|u_j-v_j|_\infty\leq\delta$.

Let $U$ and $V$ be as in claim (2), then by Lemma \ref{operatornorms} and the value of $d$, we have 
\begin{equation}\label{operatornormbounds}
\mathbb{P}(\|U-V\|_{op}\leq C_{\ref{operatornorms}}\delta\sqrt{d})\geq1-\exp(-c_{\ref{operatornorms}}c_0^2n/4),  
\end{equation} so claim (2) holds with high probability. 
Further, let $\underline{U}$ and $\underline{V}$ be $n\times l$ matrices with columns $u_1,\cdots,u_l$ and $v_1,\cdots,v_l$, then also by Lemma \ref{operatornorms}, 
$$
\mathbb{P}(\|\underline{U}-\underline{V}\|_{op}\leq C_{\ref{operatornorms}}\delta\sqrt{n})\geq1-\exp(-c_{\ref{operatornorms}}n). 
$$ Then claim (3) holds with probability at least $1-\exp(-c_{\ref{operatornorms}}n)$ for the random vector $u$, and whenever $\delta>0$ is chosen small enough.

To check claim (1), by definition of $\mathcal{N}_\mathbf{t}$ we need to verify that for each $j\in[l]$, we have 
$$ \|(v_j)_{[d]}
\|_2\in[\frac{1}{2}t_j,4t_j],\quad \operatorname{dist}((v_j)_{[d]},\mathbb{Z}^d)\leq2\sqrt{\rho n}. 
$$ The first claim is verified when \eqref{operatornormbounds} holds and whenever $\delta>0$ is chosen small enough so that $ C_{\ref{operatornorms}}\delta\sqrt{n}\leq \frac{\tau_{c_0}}{8}\sqrt{n}$, as the right hand side is smaller than $\frac{1}{2}t_i$ for each $i$ and $\|(v_j-u_j)_{[d]}\|_2$ is smaller than $C_{\ref{operatornorms}}\delta\sqrt{n}$. The second estimate is a simple consequence of triangle inequality and the fact that $\|(u_j-v_j)_{[d]}\|_2\leq C_{\ref{operatornorms}}\delta\sqrt{n}$, whenever $\delta$ is chosen small enough.

To check claim (4), let $D_U$ be the matrix $\operatorname{diag}(\|(u_1)_{[d]}\|_2,\cdots,\|(u_l)_{[d]}\|_2)$ and similarly for $D_V$. Then by $\frac{1}{8}$-orthogonality of $(v_1)_{[d]},\cdots,(v_l)_{[d]}$, we see that whenever \eqref{operatornormbounds} holds,
$$\begin{aligned}
\|UD_U^{-1}\|&\leq [\|VD_V^{-1}\|+\|U-V\|\cdot\|D_V^{-1}\|]\cdot \|D_VD_U^{-1}\|\\&\leq [\frac{9}{8}+C_{\ref{operatornorms}}\delta\sqrt{n}\cdot \frac{1}{\tau_{c_0}\sqrt{n}/4}](1-\frac{C_{\ref{operatornorms}}\delta\sqrt{n}}{\tau_{c_0}\sqrt{n}/4})^{-1}\leq \frac{5}{4}\end{aligned}
$$ if $\delta>0$ is sufficiently small relative to $\tau_{c_0}$. We similarly have $s_l(UD_U^{-1})\geq \frac{3}{4}$ under \eqref{operatornormbounds}.

We check that condition (5) follows from (1),(2),(4). Since both $(v_i)_{[d]},i\in[l]$ and $(u_i)_{[d]},i\in[l]$ are $\frac{1}{4}$-almost orthogonal and $\|(v_i)_{[d]}\|_2\geq \frac{1}{2}\|(u_i)_{[d]}\|_2\forall i\in[l]$, we have that 
$$
\|V\theta\|_2^2\geq\frac{1}{4} \sum_{i=1}^l \theta_i^2\|(v_i)_{[d]}\|_2^2\geq\frac{1}{16} \sum_{i=1}^l \theta_i^2\|(u_i)_{[d]}\|_2^2\geq \frac{1}{64}\|U\theta\|_2^2\geq (\frac{\tau_{c_0}\sqrt{n}}{4})^2.
$$
Then by triangle inequality and assumption (5) in Definition \ref{definitionforbigwt}, we have $\operatorname{dist}(U\theta,\mathbb{Z}^d)\leq\operatorname{dist}(V\theta,\mathbb{Z}^d)+\|(U-V)\theta\|_2\leq\sqrt{\rho n}+\|U-V\|_{op}\leq \sqrt{\rho n}+C_{\ref{operatornorms}}\delta\sqrt{d}\leq  2\sqrt{\rho n}$ by conclusion (2), whenever $\delta>0$ is small enough relative to $\rho$. This verifies (5).

Condition (6) can be checked similarly to condition (4) in Proposition \ref{lemma4.6stated}. If $\theta\in\mathbb{R}^l$ satisfies $\|U\theta\|=1$, then $\|\theta\|_2\leq 1/s_l(U)\leq \frac{16}{\tau_{c_0}\sqrt{n}
}$ since the columns of $U$ are $\frac{1}{4}$-almost orthogonal and each column has $\ell^2$- norm at least $\frac{\tau_{c_0}}{8}\sqrt{n}$. Then we have
$$
\|V\theta\|_2\geq\|U\theta\|_2-\|U-V\|\|\theta\|_2\geq 1-C_{\ref{operatornorms}}\delta\sqrt{n}\cdot \frac{16}{\tau_{c_0}\sqrt{n}}=1-\frac{16C_{\ref{operatornorms}}\delta}{\tau_{c_0}}.
$$
Then for any $y\in\operatorname{sparse}(c_0^2n/8)$ we have
$$
\|V\theta-y\|_2\geq(1-\frac{16C_{\ref{operatornorms}}\delta}{\tau_{c_0}})\|\frac{V\theta}{\|V\theta\|_2}-\frac{y}{\|V\theta\|_2}\|_2\geq (1-\frac{16C_{\ref{operatornorms}}\delta}{\tau_{c_0}})\tau_{c_0}  
$$
by assumption (4) of Definition \ref{definitionforbigwt}. Then we have 
$$
\|U\theta-y\|_2\geq\|V\theta-y\|_2-\|U-V\|\|\theta\|_2\geq (1-\frac{16C_{\ref{operatornorms}}\delta}{\tau_{c_0}})\tau_{c_0}-\frac{16C_{\ref{operatornorms}}\delta}{\tau_{c_0}}\geq\frac{1}{2}\tau_{c_0} 
$$
where the last inequality holds whenever $\delta>0$ is small enough. This verifies (6).

To check condition (7), we verify that for each $j\in[l]$,
$$
\mathbb{E}\|A(u_j-v_j)\|_2^2=\sum_{i=1}^n \mathbb{E}\epsilon_{ij}^2\|Ae_i\|_2^2\leq\delta^2\|A\|_{HS}^2\leq K^2\delta^2n^2.
$$

Then using the independence of the events at different $j$, we have
$$
\mathbb{P}(\forall j\in[l]:\|A(u_j-v_j)\|_2\leq 2K\delta n)\geq 2^{-l}.
$$

Finally, since $\exp(-c_{\ref{operatornorms}}n)+\exp(-c_{\ref{operatornorms}}c_0^2n/4)+1-2^{-l}\leq 1$, there exists a realization of $u_1,\cdots,u_l$ such that conditions (1) to (7) can hold at once. This completes the proof.
\end{proof}

We extract the following fact which will be very useful later on:
\begin{fact}\label{fact5.555}
   Let $U$ be the $d\times l$ matrix defined in item (2) of Proposition \ref{proposition5.42}. Then whenever  $L\geq\sqrt{l}$ and  $\alpha_0\in(0,\tau_{c_0}/4)$, we have that
       $$ D_{L,\alpha_0}(U)\geq\frac{1}{20\sqrt{l}}.
       $$
   
\end{fact}
\begin{proof}
    Note that by Proposition \ref{proposition5.42}, (5), for any $\theta\in\mathbb{R}^l$ with $\|\theta\|_2\leq\frac{1}{20\sqrt{l}}$, either one of the following  holds: $$\|U\theta\|_2\leq 2\tau_{c_0}\sqrt{n}, \text{ or } \|U\theta\|_\mathbb{Z}\geq \frac{1}{2}\sqrt{\rho n}.$$
    Then for all such $\theta$ with $\|\theta\|_2\leq\frac{1}{20\sqrt{l}}$ and $L\geq\sqrt{l}$, we have that for any $\alpha_0\in(0,\frac{1}{4})$,
    $$
L\sqrt{\log_+\frac{\alpha_0\|U\theta\|_2}{L}}\leq L\sqrt{\log_+\exp(\frac{\rho n}{4L^2})}\leq\frac{1}{2}\sqrt{\rho n}
    $$ where we use the $\frac{1}{4}$-orthogonality of columns of $U$ to deduce that $\|U\theta\|_2\leq\|U\|_{op}\leq 4\exp(\frac{\rho n}{4L^2})$. This implies that, whenever $\|U\theta\|_2\geq2\tau_{c_0}\sqrt{n}$, 
    \begin{equation}\label{secondclaims}
\|U\theta\|_\mathbb{Z}\geq L\sqrt{\log_+\frac{\alpha_0\|U\theta\|_2}{L}}.
    \end{equation}
    On the other hand, when $\|U\theta\|_2\leq 2\tau_{c_0}\sqrt{n}$ and $\alpha_0\leq \tau_{c_0}/4$, then by Lemma \ref{lcdimcomps} and (6) of Proposition \ref{proposition5.42}, the estimate \eqref{secondclaims} is also valid for any $L\geq 1$ (recall that we assume $\tau_{c_0}\leq c_0/8$). Thus, we have checked that \eqref{secondclaims} holds for all $\theta\in\mathbb{R}^l:\|\theta\|_2\leq\frac{1}{20\sqrt{l}}$, which completes the proof.
\end{proof}

\subsection{The inversion procedure: some preparations}\label{implementation2.1}
The constructions and computations in this and the next section are similar to those in Section \ref{implementation1} and \ref{implementation2}.

We first need a different version of Lemma \ref{howlarageisthebox} as the geometry of the candidate vectors are different here. In the following we only consider the last $n-d$ coordinates of a vector:
\begin{lemma}\label{lemma1186}
Fix $d\leq n/2$. We say $\mathcal{B}=\mathcal{B}_{d+1}\times\mathcal{B}_{d+2}\times\cdots\times\mathcal{B}_n\subset\mathbb{Z}^{n-d}$ is an $(N,d,\kappa)$ box for some $N\in\mathbb{N},\kappa>0$ if $|B_i|\geq N$ for each $i\in [d+1,n]$ and $|\mathcal{B}|\leq(\kappa N)^{n-d}$. For any $t>0$ and $\delta>0$ we define the following subset of the integer lattice (where $B_{n-d}(t)$ is the Euclidean ball of radius $t$ in $\mathbb{R}^{[d+1,n]}$, centered at 0)
$$
\Lambda_{t,\delta}=B_{n-d}(0,t)\cap (\delta\cdot\mathbb{Z}^{[d+1,n]}).
$$ Then whenever $\kappa\leq 6$, there exists a family $\mathcal{F}$ of $(N,d,\kappa)$ boxes such that 
$$
\Lambda_{t,\delta}\subset\cup_{\mathcal{B}\in\mathcal{F}} (\delta\cdot\mathcal{B})
$$ where we take $N= \lceil t(\delta \sqrt{n})^{-1}\rceil$ and the family $\mathcal{F}$ has cardinality $|\mathcal{F}|\leq 2^{10n}$.
\end{lemma}
 The proof of Lemma \ref{lemma1186} is deferred to Section \ref{theotherproofs}.

We also need a version of Theorem \ref{compressibletheorem12.234567} for singular values of the whole matrix $H$:

\begin{theorem}\label{compressibletheorem12.234567new}
    For given $n\in\mathbb{N}$, $0<c_0\leq 2^{-50}B^{-4}$, let $d\leq c_0^2n$ and fix $\alpha\in(0,1)$.  Consider an $\ell$-tuple of $\frac{1}{4}$-orthogonal vectors $X_1,\cdots,X_\ell\in\mathbb{R}^d$ that satisfy $D_{L,\alpha}(\frac{c_0}{32\sqrt{n}}\mathbf{X})\geq 256B^2\sqrt{\ell}$ and that $\|X_i\|_2\geq t_i$ for each $i\in[\ell]$ where $t_i>0$ for each $i\in[\ell]$ (where we denote $\frac{c_0}{32\sqrt{n}}\mathbf{X}:=(\frac{c_0}{32\sqrt{n}}X_1,\cdots,\frac{c_0}{32\sqrt{n}}X_\ell)$) and where we choose $L=(\frac{8}{\sqrt{\nu p}}+\frac{256B^2}{\sqrt{c_0}})\sqrt{\ell}$.

    Let $H$ be an $(n-d)\times d$ random matrix with i.i.d. rows of distribution $\Phi_\nu(d;\zeta)$ with $\nu=2^{-15}$ (and we recall $p\geq\frac{1}{2^7B^4}$ from \eqref{whatdoesbhave?}.

    Then whenever $k\leq 2^{-21}B^{-4}c_0^2n$ and $\prod_{i=1}^\ell R\frac{\sqrt{n}}{t_i}\geq\exp(-2^{-23}B^{-4}d)$, we have:
    \begin{equation}\begin{aligned}\label{finalwehaveestimatesnew}
    \mathbb{P}_H&(\sigma_{d-k}(H)\leq c_02^{-4}\sqrt{n}
        \text{ and }\|HX_i\|_2\leq  n\quad\text{for all }i\in[\ell])\leq  e^{-c_0nk/48}(\prod_{i=1}^\ell\frac{R\sqrt{n} }{t_i\alpha})^{n-d},
\end{aligned}\end{equation}
   where we take $R=2^{46}B^2c_0^{-3}(\frac{8}{\sqrt{\nu p}}+\frac{256B^2}{\sqrt{c_0}})$.\end{theorem}
Theorem \ref{compressibletheorem12.234567new} can be adapted from \cite{han2025repeated}, Theorem 12.23. See Section \ref{technicaldeduction} for  details.

\subsection{Implementing the ``inversion of randomness’' procedure}\label{implementation2.2}
We first introduce a convenient notion of typical vector pairs:
\begin{Definition}\label{typicalvectorboxes}
For a tuple of parameters $\mathbf{t}=(t_1,\cdots,t_l)$ such that $t_i\in[\frac{\tau_{c_0}}{4}\sqrt{n},\exp(\frac{\rho n}{4L^2})]$, we define \begin{equation}\label{DefinitionofMi}N_i=\lceil \frac{3t_i}{\tau_{c_0}}\cdot (\delta\sqrt{n})^{-1}\rceil \text{ for each }i\in[l].\end{equation} Then let $\mathcal{B}_1,\cdots,\mathcal{B}_l$ be a family of $(N_i,d,6)$-boxes for each $i\in[l]$ as in Lemma \ref{lemma1186}.

We define the following set of typical vectors relative to these fixed boxes:
\begin{equation}\label{typicalhatt}\begin{aligned}
\hat{T}&=\{(X_1,\cdots,X_l)\subset (\mathbb{R}^n)^l:  
((X_1)_{[d]},\cdots,(X_l)_{[d]})\in\mathcal{N}_\mathbf{t} \text{ (see   Lemma \ref{lemma5.2nts}),}
\\&(X_i)_{[d+1,n]}\in \delta\cdot  \mathcal{B}_i\forall i\in[l],\\& \text{and the tuples satisfy assumptions (1),(4),(5),(6) of Proposition \ref{proposition5.42}}\},
\end{aligned}\end{equation} where in assumption (5), the matrix $U$ has columns given by $(X_1)_{[d]},\cdots,(X_l)_{[d]}$. \end{Definition}

In a nutshell, Definition  \ref{typicalvectorboxes} uses the net $\mathcal{N}_\mathbf{t}$ based on the essential LCD (see Lemma \ref{lemma5.2nts}) for the first $d$ coordinates of the vectors, and uses the boxes $\mathcal{B}_{1},\cdots,\mathcal{B}_l$ for the remaining $n-d$ coordinates of the vectors. It moreover requires that the assumptions in Proposition \ref{proposition5.42} are justified. By Lemma \ref{lemma1186}, we can check that any pair of vectors $(u_1,\cdots,u_l)$ satisfying assumptions (1) to (6) of Proposition \ref{proposition5.42} should lie in the typical set $\hat{T}$ for some choices of $(N_i,d,6)$ boxes $\mathcal{B}_1,\cdots,\mathcal{B}_l$. The conditions (1),(4),(5),(6) of Proposition \ref{proposition5.42} are only imposed on the first $d$ coordinates of the vectors, and the other $n-d$ coordinates can be freely chosen from the box. We again denote by $\mathbf{X}=(X_1,\cdots,X_l)$ for an $\ell$-tuple of vectors in $\mathbb{R}^n$.

Recall that $M$ is the zeroed-out matrix defined in \eqref{zeroedoutmatrixM}. We shall also use the following notion of a robust rank for $H$, similarly to Section \ref{implementation2}:

For each $k\in[d]$ let $\mathcal{E}_k$ denote the event 
$$\begin{aligned}
\mathcal{E}_k:=&\{H:\sigma_{d-k}(H)\geq c_0\sqrt{n}/16\text{ and }\sigma_{d-k+1}(H)< c_0\sqrt{n}/16\}\end{aligned}
$$ where we recall that $J=[d_e]$. Then exactly one of the events in $\mathcal{E}_0,\cdots,\mathcal{E}_{d}$ should hold.

We similarly define the following two events: the first event is $\mathcal{A}_1=\mathcal{A}_1(\mathbf{X})$ via
$$
\mathcal{A}_1:=\{ H: \|\left(H(X_1)_{[d]},\cdots,H(X_l)_{[d]}\right)\|_2\leq 2K\delta n\sqrt{l}
\} 
$$ and the second event $\mathcal{A}_2=\mathcal{A}_2(\mathbf{X})$ is via
$$
\mathcal{A}_2:=\{H:\|(H^T(X_1)_{[d+1,n]},\cdots, H^T(X_l)_{[d+1,n]})\|_2\leq 2K\delta n\sqrt{l}\}, 
$$ where $H$ is the submatrix of $M$ in \eqref{zeroedoutmatrixM}.
Then we have the following simple observation:
\begin{fact} For any fixed $X_1,\cdots,X_l\in\mathbb{R}^n$ and $\mathcal{A}_1,\mathcal{A}_2$ as above, we have
    $$\mathbb{P}_M(\|MX_i\|_2\leq 2K\delta n\text{ for each }i\in[l])\leq\mathbb{P}_H(\mathcal{A}_1\cap\mathcal{A}_2).$$
\end{fact}
Throughout the following proof we assume that, as in Theorem \ref{compressibletheorem12.234567new},
\begin{equation}\label{valueofl}
L=(\frac{8}{\sqrt{\nu p}}+\frac{256B^2}{\sqrt{c_0}})\sqrt{\ell}.
\end{equation}

The following is the main step in the ``inversion of randomness’'  computation procedure:

\begin{Proposition}\label{prop1stmomentnewvassar} We define the following function for $\mathbf{X}=(X_1,\cdots,X_l)\in(\mathbb{Z}^n)^l$:
$$f(\mathbf{X}):=\mathbb{P}_M(\|(MX_1,\cdots,MX_l)\|_2\leq 2K\delta n\sqrt{l})\mathbf{1}(\mathbf{X}\in \hat{T}).$$
Then whenever $l\leq c_{c_0,B}\delta\sqrt{n}$ for some $c_{c_0,B}>0$ depending only on $c_0$ and $B$, and whenever $\delta>0$ is sufficiently small relative to $c_0$, the expectation of $f(\mathbf{X})$ can be estimated as follows: 
$$
\mathbb{E}_{(X_1,\cdots,X_l)\in\mathcal{N}_\mathbf{t}\times (\delta\mathcal{B}_1\times\cdots\times\delta\mathcal{B}_l)}  f(\mathbf{X})\leq (\prod_{i=1}^l\frac{R'\delta\sqrt{n}}{t_i}) ^{n}
$$ for a constant $R'>0$ depending only on $c_0,p,\nu,B$.

The expectation $\mathbb{E}_{(X_1,\cdots,X_l)\in\mathcal{N}_\mathbf{t}\times (\delta\mathcal{B}_1\times\cdots\times\delta\mathcal{B}_l)}$ means that $((X_1)_{[d]},\cdots,(X_l)_{[d]})$ is uniformly chosen from $\mathcal{N}_\mathbf{t}$ and $((X_1)_{[d+1,n]},\cdots,(X_l)_{[d+1,n]})$ is uniformly chosen over $\delta\mathcal{B}_1\times\cdots\times \delta\mathcal{B}_l$.

\end{Proposition}

\begin{remark}\label{remark5.12onlcd}
    The restriction $l\leq c_{c_0,B}\delta\sqrt{n}$ in this result is the place where the restriction $k\leq c\sqrt{n}$ in Theorem \ref{maintheorem1.1} arises. This restriction is due to the use of Fact \ref{fact5.555} in the inverse Littlewood-Offord inequality (either Theorem \ref{compressibletheorem12.234567new} or Theorem \ref{littlewoodincomp}), which itself originates from the use of Lemma \ref{lemma2.4}, case (1)(b). The same restriction already exists for the large deviation of the rank of a square random matrix with i.i.d. entries, see \cite{rudelson2024large}.
\end{remark}

Before the detailed proof we outline a brief roadmap for Proposition \ref{prop1stmomentnewvassar}.
The purpose of Proposition \ref{prop1stmomentnewvassar} is to obtain a first moment estimate for the random variable $f(\mathbf{X})$
 on the typical set $\hat{T}$. The proof first splits the event according to the robust rank condition on the block $H$ and these events are denoted by the sets $E_{k'}$. For relatively small $k'$, we use Theorem \ref{compressibletheorem12.234567new} to control the event $\mathcal{A}_1$, namely the smallness of $H(X_i)_{[d]}$, relying on the lower bounds for the LCD in Fact \ref{fact5.555}. For the remaining coordinates, we condition on $H$ and use Corollary \ref{corollary4.11} to bound the event $\mathcal{A}_2$, exploiting the independence of coordinates of $(X_i)_{[d+1,n]}$ inside the boxes defining $\hat{T}$. Summing over $k'$ then yields the desired upper bound on $\mathbb{E}f(\mathbf{X})$. This average estimate is later converted, via Markov's inequality and the upper bound on $|\hat{T}|$, into a point-wise probability estimate in Corollary \ref{pilethingsupfinal}.

\begin{proof}[\proofname\ of Proposition \ref{prop1stmomentnewvassar}]
We take a conditioning on $H$ and get 
\begin{equation}
f(\mathbf{X})\leq\sum_{{k'}=0}^{d}\mathbb{P}_H(\mathcal{A}_2\mid\mathcal{A}_1\cap\mathcal{E}_{k'})\mathbb{P}_H(\mathcal{A}_1\cap\mathcal{E}_{k'})\mathbf{1}(\mathbf{X}\in \hat{T}).
\end{equation}

Set $\alpha'=2^{-21}B^{-4}c_0^2$. Then whenever $k'\leq \alpha'n$ and $l\leq c_{c_0,B}\delta\sqrt{n}$ for some $c_{c_0,B}>0$, we can apply Theorem \ref{compressibletheorem12.234567new} for any $\mathbf{X}\in \hat{T}$ to upper bound $\mathbb{P}_H(\mathcal{A}_1\cap\mathcal{E}_{k'})$ by
\begin{equation}\label{firstroundsnew}
\mathbb{P}_H(\mathcal{A}_1\cap\mathcal{E}_{k'})\leq
\exp(-c_0n{k'}/48)(\prod_{i=1}^l\frac{2KR\cdot \delta \sqrt{n}}{t_i\tau_{c_0}/8})^{(n-d)}.
\end{equation} More precisely, we take $\overline{X}_i=\frac{1}{2K\delta }X_i$ for each $i\in[l]$ and $\alpha=\tau_{c_0}/8$ in Theorem \ref{compressibletheorem12.234567new}. The condition on essential LCD imposed in Theorem \ref{compressibletheorem12.234567new} is satisfied as we note that $$D_{L,\tau_{c_0}/8}(\frac{c_0}{32\sqrt{n}}\overline{X}_1,\cdots,\frac{c_0}{32\sqrt{n}}\overline{X}_l)= \frac{64K\delta\sqrt{n}}{c_0}D_{L,\tau_{c_0/8}}(U)\geq \frac{16K\delta\sqrt{n}}{5c_0\sqrt{l}}\geq 256B^2\sqrt{l}$$ 
where the first inequality follows from Fact \ref{fact5.555} and the last inequality holds
whenever $l\leq c_{c_0,B}\delta \sqrt{n}$ for some constant $c_{c_0,B}>0$ depending only on $C_0,B$.

For larger values of $k'$, we have the following upper bound by Theorem \ref{compressibletheorem12.234567new}:
$$\begin{aligned}
\sum_{{k'}\geq \alpha'n}\mathbb{P}_H(\mathcal{A}_1\cap\mathcal{E}_{k'})&\leq\mathbb{P}_H(\mathcal{A}_1\cap\{\sigma_{d-\alpha'n}(H)\leq c_0\sqrt{n}/16)\leq\exp(-2^{-27}c_0^3B^{-4}n^2).\end{aligned}
$$

Now we take the inversion step with randomness generated from vectors in the box. We define a family of functions $g_{k'}(\mathbf{X}):=\mathbb{P}_H(\mathcal{A}_2\mid\mathcal{A}_1\cap\mathcal{E}_{k'})$, ${k'}\in[\alpha'n]$. Then we have
$$
\mathbb{E}_\mathbf{X}[g_{k'}(\mathbf{X})]=\mathbb{E}_{\mathbf{X}_{[d]}}\mathbb{E}_H[\mathbb{E}_{\mathbf{X}_{[d+1,n]}}\mathbf{1}[\mathcal{A}_2]\mid\mathcal{A}_1\cap\mathcal{E}_{k'}],
$$  where we use the symbol $\mathbf{X}_{[d]}$ to mean the vectors $(X_1)_{[d]},\cdots,(X_l)_{[d]}$ and similarly for $\mathbf{X}_{[d+1,n]}$. Now we consider some fixed $H\in\mathcal{A}_1\cap\mathcal{E}_{k'}$, so that $\sigma_{d-k'}(H)\geq c_0\sqrt{n}/16$. Then we apply both Corollary \ref{corollary4.11} and Lemma \ref{tensorizationlemma} to see that for this fixed $H$, 
\begin{equation}\label{secondexpectationfirstnew}
\mathbb{E}_{\mathbf{X}_{[d+1,n]}}\mathbf{1}[\mathcal{A}_2]\leq (\prod_{i=1}^l\frac{R'}{N_i})^{d-k'}\leq (\prod_{i=1}^l\frac{R'\delta\sqrt{n}}{t_i})^{d-k'}
\end{equation} for a constant $R'>0$ that only depends on $c_0,p,\nu,B$ and whose value changes from line to line.
More precisely, we use independence of the coordinates of $(X_1)_{[d+1,n]},\cdots,(X_l)_{[d+1,n]}$, and for each fixed $i\in[l]$ we have that for any $K\geq 1$, the following estimate holds
\begin{equation}\label{inversiononetwo}
\mathbb{P}_{{(X_i)}_{[d+1,n]}}(\| H^T(X_i)_{[d+1,n]}\|_2\leq 2K\delta n)\leq (\frac{CK}{c_0^4N_i})^{d-k'},\end{equation}
where we apply Corollary \ref{corollary4.11} to deduce \eqref{inversiononetwo}. Finally we use Lemma \ref{tensorizationlemma}, which applies to non i.i.d. variables, to tensorize all the $l$ dimensions in \eqref{inversiononetwo} and get \eqref{secondexpectationfirstnew}.

Now we combine estimates \eqref{firstroundsnew} and \eqref{secondexpectationfirstnew} to deduce that 
$$\begin{aligned}
\mathbb{E}_\mathbf{X} f(\mathbf{X})&\leq  \sum_{{k'}=1}^{\alpha'n}(\prod_{i=1}^l\frac{R'\delta\sqrt{n}}{t_i})^{
n-{k'}}\exp(-c_0n{k'}/48)+\exp(-2^{-27}c_0^3B^{-4}n^2)\\&\leq (\prod_{i=1}^l\frac{R'\delta \sqrt{n}}{t_i})^{n},
\end{aligned}$$ where one can check that the last inequality holds whenever $\rho>0$ is sufficiently small relative to $c_0$, using $t_i\leq\exp(\frac{\rho n}{4L^2})$ and the value of $L$ \eqref{valueofl}.
\end{proof}

We also need a crude upper bound of $f$ for any $\mathbf{X}\in \hat{T}$: this follows by using only the randomness in the first $d$ columns (and hence the last $n-d$ rows) of $M$.
\begin{lemma}\label{basecasenewvassar}
    For any $\mathbf{X}\in \hat{T}$ we have that
    $$
f(\mathbf{X})\leq (\prod_{i=1}^l\frac{R'\delta\sqrt{n}}{t_i})^{n-d},
    $$ for the same constant $R'>0$ as the one showing up in Proposition \ref{prop1stmomentnewvassar}.
\end{lemma}

    \begin{proof}
        We bound $f(\mathbf{X})$ by $\mathbb{P}_H(\mathcal{A}_1)$, and use the (standard) inverse Littlewood-Offord inequality, Theorem \ref{littlewoodincomp}, to compute the small ball probability as in the previous proof. One may enlarge the constant $R'$ in Proposition \ref{prop1stmomentnewvassar} so the same $R'$ works for both results.
    \end{proof}

\subsection{Completing the argument}\label{implementation3}

We return to the small ball probability relative to $A$:

\begin{corollary}\label{pilethingsupfinal} For $A\in\operatorname{Sym}_n(\zeta),\zeta\in\Gamma_B$, we can find $\rho>0$ sufficiently small and $\delta>0$ sufficiently small such that,
whenever $l\leq c_{c_0,B}\delta\sqrt{n}$, we have
$$\mathbb{P}\left(\text{There exists an $l$-tuple } (X_1,\cdots,X_l)\in \hat{T} \text{ such that } \|AX_i\|_2\leq 2K\delta n \forall i\in[l]\right)\leq  2^{-100nl}.    $$
\end{corollary}

\begin{proof}
  We consider the following partition of vectors in $\hat{T}$ for each $k'=0,1,\cdots,d$:
$$\begin{aligned}&E_{k'}:=\{\mathbf{X}\in  T:
\mathbb{P}_M(\|(MX_1,\cdots,MX_l)\|_2\leq 2K\delta n\sqrt{l})\geq (\prod_{i=1}^l\frac{R'\delta\sqrt{n}}{t_i})^{n-k'}\}
\end{aligned}$$ where we denote by $E_{-1}=\hat{T}$.
These subsets are nested in the sense that $E_{-1}\supset E_0\supset E_1\supset E_2\supset\cdots\supset E_{d}$. We also have $E_{k'}=\emptyset$ for all $k'\geq d+1$ by Lemma \ref{basecasenewvassar}. 

Suppose that $(X_1,\cdots,X_l)\in E_{k'}\setminus E_{k'+1}$ for some $k'\geq -1$, then by definition of $E_{k'+1}$ we have $\mathbb{P}_M(\|(MX_1,\cdots,MX_l)\|_2\leq 2K\delta n\sqrt{l})\leq (\prod_{i=1}^l\frac{R'\delta\sqrt{n}}{t_i})^{n-k'-1}$, so by Lemma \ref{lemma4.145}, 
\begin{equation}\label{singlepairestimatenew}\mathbb{P}_A(\|(AX_1,\cdots,AX_l)\|_2\leq 2K\delta n\sqrt{l})\leq (\prod_{i=1}^l\frac{C'R'\delta\sqrt{n}}{t_i})^{n-k'-1}\end{equation} 
for a fixed constant $C'>0.$

Next, we compute the cardinality of $E_{k'}$. Applying Markov's inequality to the conclusion of Proposition \ref{prop1stmomentnewvassar}, we get that
$$\begin{aligned}
&\mathbb{P}_\mathbf{X}(\mathbf{X}\in E_{k'})\\&
= \mathbb{P}_{\mathbf{X}}(\mathbb{P}_M(\|(MX_1,\cdots,MX_l)\|_2\leq 2K\delta n\sqrt{l})\geq (\prod_{i=1}^l\frac{R'\delta\sqrt{n}}{t_i})^{n-k'})\mathbf{1}(\mathbf{X}\in T)\\&\leq (\prod_{i=1}^l\frac{t_i}{R'\delta\sqrt{n}})^{n-k'}\mathbb{E}_\mathbf{X}f(\mathbf{X})\leq  (\prod_{i=1}^l\frac{R'\delta\sqrt{n}}{t_i})^{k'}.
\end{aligned}
$$ 

The cardinality of $\hat{T}$ can be computed by combining Lemma \ref{lemma5.2nts} and Lemma \ref{lemma1186}:
$$
|\hat{T}|\leq (\prod_{j=1}^l\frac{C\sqrt{\rho}t_j}{\delta \sqrt{d}})^d\cdot (\prod_{j=1}^l\frac{18t_j}{\tau_{c_0}\cdot \delta\sqrt{n}})^{n-d}\leq (\prod_{j=1}^l\frac{C''t_j}{\delta\sqrt{n}})^n\cdot \rho^{ld/2},
$$ where $C''>0$ depends only on $c_0,\tau_{c_0}$ and we use the assumption that $c_0^2n/4\leq d\leq c_0^2n$. 

Then the cardinality of $E_{k'}$ is upper bounded by 
\begin{equation}\label{thisboundnew}
|E_{k'}|=|\hat{T}|\cdot\mathbb{P}_\mathbf{X}(\mathbf{X}\in E_{k'}) \leq (\prod_{j=1}^l\frac{R''t_j}{\delta\sqrt{n}})^{n-k'}\cdot \rho^{ld/2},
\end{equation} so we can compute the following probability by combining \eqref{thisboundnew} with \eqref{singlepairestimatenew}:
$$\begin{aligned}&
\mathbb{P}_A(\exists\mathbf{X}\in E_{k'}\setminus E_{k'+1}:\|(AX_1,\cdots,AX_l)\|_2\leq 2K\delta n\sqrt{l})\\&\leq 
\rho^{ld/2}\cdot (\prod_{j=1}^l\frac{R''t_j}{\delta\sqrt{n}})^{n-k'}(\prod_{j=1}^l\frac{C'R'\delta\sqrt{n}}{t_j})^{n-k'-1}\leq 2^{-200nl}.
\end{aligned}$$ where the last inequality holds when we set $\rho>0$ sufficiently small, using the fact that $c_0^2n/4\leq d\leq c_0^2n$ and $t_j\leq\exp(\frac{\rho n}{4L^2})$ for each $j\in[l]$.

Finally, we take the sum of the individual probabilities for $k'=0,1,\cdots,d$ to conclude the proof.
\end{proof}

Now we can complete the proof of Proposition \ref{propositionlcds}. To make the proof more readable, we first present a road map for the proof of Proposition \ref{propositionlcds}. By Fact \ref{fact3.4} we only need to exclude $W_{1,i}$ inside the subset $F_1'$. We get this via applying Lemma \ref{lemma2.4} with $E=F_1',I=[d],W_I=W_{1,i}$. Suppose that alternative (1) of Lemma \ref{lemma2.4} can be applied $l/2$ times, then we can obtain an almost orthogonal tuple of vectors in $F_1'\cap W_{1,i}$ satisfying the minimality condition in   Lemma \ref{lemma2.4} (1)(b). Proposition \ref{proposition5.42} then replaces this tuple by lattice approximations with controlled geometry. Then after rescaling, these approximation vectors lie in the set of typical vectors $\hat{T}$ defined in \eqref{typicalhatt}. Finally, Corollary \ref{pilethingsupfinal} shows that the probability that such a typical tuple further satisfies the required smallness conditions for $AX_1,\cdots,AX_l$ is super-exponentially small. This contradiction implies that alternative (1) should not occur for $l/2$ times with probability larger than $\exp(-cnl)$, and thus Lemma \ref{lemma2.4} yields the desired subspace $F_2'$.  

\begin{proof}[\proofname\ of Proposition \ref{propositionlcds}] By Fact \ref{fact3.4}, case (2), we only need to find a linear subspace $F_2'\subset F_1'$ which is disjoint from the set $W_{1,i}$ defined in \eqref{anumberofsubsets}. We will prove this by applying  Lemma \ref{lemma2.4} with the choice $E=F_{1}'$, $W_I:=W_{1,i}$ and $I=[d]$, and show that the probability that we can apply alternative (1) of Lemma \ref{lemma2.4} for $l/2$ times is at most $e^{-\Omega(ln)}$, so that the alternative case, i.e. case (2) of Lemma \ref{lemma2.4} will immediately yield the linear subspace $F_2'$ as required in Proposition \ref{propositionlcds}. We assume $\|A\|_{HS}\leq K\delta n$, which has probability $1-\exp(-\Omega(n^2))$.

Suppose that we can apply Lemma \ref{lemma2.4} for $l/2$ times, then we can find an $l/2$ tuple of vectors $v_1,\cdots,v_{l/2}$ satisfying Definition \ref{definitionforbigwt} (that is, $v_1,\cdots,v_{l/2}\in W_\mathbf{t}$  for a given tuple $\mathbf{t}=(t_1,\cdots,t_{l/2})$). We take a dyadic partition for the possible range of each $t_i\in[\frac{\tau_{c_0}n}{4},\exp(\frac{\rho n}{4L^2})]$, and first fix a value of $\mathbf{t}=(t_1,\cdots,t_{l/2})$ from this dyadic partition. The total number of points in this partition is at most $n^{l/2}=\exp(l\log n/2)$.

For each such fixed tuple $\mathbf{t}$, suppose that we can apply case (1) of Lemma \ref{lemma2.4} to find vectors $(v_1,\cdots,v_{l/2})\in W_\mathbf{t}$, then by Proposition \ref{proposition5.42}, we can find $(u_1,\cdots,u_{l/2})$ satisfying all the conditions of Proposition \ref{proposition5.42}. By Definition \ref{typicalvectorboxes} and the discussions after it, we can find at most $|\mathcal{F}|^l\leq 2^{10nl/2}$ typical boxes $\hat{T}$ as defined in \eqref{typicalhatt} that contain all such possible vectors $u_1,\cdots,u_{l/2}$. 
By Corollary \ref{pilethingsupfinal}, the probability that there exists a tuple $u_1,\cdots,u_{l/2}$ in a fixed typical box $\hat{T}$ that satisfies the smallness condition on $\|Au_i\|_2$ is at most $2^{-100nl}$. Then the overall probability in question is at most 
$$
\exp(l\log n/2)\cdot 2^{10nl}\cdot 2^{-100nl}=e^{-\Omega(nl)},
$$ completing the proof.
\end{proof}

\section{On the rank of a random symmetric matrix and random graphs}\label{ranklastpart}

Now we have made all the preparations for the proof of Theorem \ref{maintheorem1.1}.

\begin{proof}[\proofname\ of Theorem \ref{maintheorem1.1}]
For small (or any bound) values of $k$, the conclusion of Theorem \ref{maintheorem1.1} already follows from \cite{campos2025singularity}, Theorem 1.1 when $\zeta$ has uniform $\{\pm 1\}$ distribution, or from \cite{campos2024least}, Theorem 1.1 for general subgaussian distribution.
 
 We can thus assume $k$ is divisible by $8$ for simplicity: otherwise, replace $k$ by $8k$.   Suppose that $A\in\operatorname{Sym}_n(\zeta)$ has $\operatorname{Rank}(A)\leq n-k$, then we can find a principal minor $A_{n-k}$ of $A$ of size $n-k$ such that $\operatorname{Rank}(A_{n-k})=\operatorname{Rank}(A)$ (this can be checked via computing the characteristic polynomial of $A$, and expressing its coefficients as linear sums of determinants of principal mimors), and that any principal minor containing $A_{n-k}$ must have the same rank as $A_{n-k}$ does.

    We assume without loss of generality that the upper left $(n-k/8)\times(n-k/8)$ minor of $A$ contains this $A_{n-k}$, with the following decomposition:
    $$
A=\begin{bmatrix}
    A_{n-k/8}&{X}\\X^*&D
\end{bmatrix}
    $$ where $D$ is a square matrix and $X$ has size $(n-k/8)\times (k/8)$ with independent entries.

Then $\operatorname{Rank}(A_{n-k/8})\leq n-k$. Applying Proposition \ref{mainpropsection4} and Proposition \ref{propositionlcds} (where we use $k\leq c\sqrt{n}$ for some $c>0$), we can deduce that on an event $\Omega_k$ with $\mathbb{P}(\Omega_k)\geq 1-e^{-\Omega(nk)}$, we can find a linear subspace $E\subset\ker A_{n-k/8}$ with $\dim E\geq k/4$ such that for any $v\in E\cap\mathbb{S}^{n-k/8-1}$, we have $D_{\rho,\gamma}(v)\geq\exp(\frac{Cn}{4k})$ for some $C>0$. Then by Definition \ref{definition4.33.4}, we readily check that for any $L_0\geq\sqrt{l/(1-\mathcal{L}(1,\zeta)}$ and $\alpha_0>0$ we can find $C>0$ such that 
$$
D_{L_0,\alpha_0}(E)\geq \exp(\frac{Cn}{4k}).
$$

We let the columns of $A_{n-k/8}$ be denoted by $u_1,\cdots,u_{n-k/8}$, then $$E\subset (\operatorname{span}(u_1,\cdots,u_{n-k/8}))^\perp.$$ On the other hand, by definition we have 
$$
\operatorname{Rank}(A_{n-k})\leq\operatorname{Rank}(A_{n-k/8})\leq \operatorname{Rank}(\begin{bmatrix}
    A_{n-k/8}\quad X
\end{bmatrix})\leq \operatorname{Rank}(A)=\operatorname{Rank}(A_{n-k}).  
$$
This implies that, if we denote by $X_1,\cdots,X_{k/8}$ the columns of $X$, then each $X_i$ lies in $\operatorname{span}(u_1,\cdots,u_{n-k/8})$ and in particular
$$\operatorname{Proj}_EX_i=0\quad \forall i\in[k/8].
$$ Since $X_1,\cdots,X_{k/8}$ are independent random vectors with entry distribution $\zeta\in\Gamma_B$, we can apply the inverse Littlewood-Offord theorem, Theorem \ref{littlewoodincomp} to check that 
$$
\mathbb{P}(\operatorname{Proj}_EX_i=0\quad \forall i\in[k/8])\leq e^{-\Omega(nk)}.
$$

Finally, combining all the above discussions, and an additional $2^n$ combinatorial factor for the choice of the principal submatrix $A_{n-k/8}$, we have that 
\begin{equation}\label{rankdeficitlarge}
\mathbb{P}(\operatorname{Rank}(A)\leq  n-k)\leq 2^n[\mathbb{P}(\Omega_k^c)+\mathbb{P}(\operatorname{Proj}_E X_i=0\forall i\in[k/8])]\leq e^{-\Omega(nk)}
\end{equation} as long as $k$ is sufficiently large. This proves the large deviation inequality for the whole range of $k$ as stated.
\end{proof}

\subsection{Adaptation to the rank of random graphs}
In this section we outline the proof of Theorem \ref{corollaryrandomgraph}. We will essentially modify the proof of Theorem \ref{maintheorem1.1}, but we will need to make the following necessary changes. Let $G$ denote the adjacency matrix of the random graph $G(n,p)$, then

\begin{enumerate}
    \item The operator norm $\|G\|$ is no longer of the order $O(\sqrt{n})$, but the Hilbert Schmidt norm $\|G\|_{HS}$ still has order $O(n)$. Since all the approximation procedures in this paper do not rely on the operator norm $\|G\|$ but only rely on the Hilbert-Schmidt norm $\|G\|_{HS}$ via the random rounding approach, no additional effort is needed here.
    \item The diagonal entries of $G$ are identically zero, but this does not affect the proof of Theorem \ref{maintheorem1.1} in any significant ways. We only need straightforward modifications.
    \item The two conditional inverse Littlewood-Offord theorems, Theorem \ref{compressibletheorem12.234567} and \ref{compressibletheorem12.234567new}, are still true for random variables $\zeta$ of a nonzero mean. This is due to the random-rounding technique (Lemma \ref{netofmatrices}) we adopt as we estimate the least singular values.
    \item The only point where additional technical effort is needed is when the corank $k$ is very small. The case $k=1$ of Theorem \ref{maintheorem1.1} used the main result of \cite{campos2025singularity} or \cite{campos2024least}, but no such results are directly available for the matrix $G$ with entries of non-zero mean.
\end{enumerate}

\begin{proof}[\proofname\ of Theorem \ref{corollaryrandomgraph} for large corank $k$] The proof is a straightforward modification of the proof of Theorem \ref{maintheorem1.1} for large corank $k$, taking into account the above bullet points (1),(2),(3). The corank $k$ only needs to be large enough such that \eqref{rankdeficitlarge} holds.
\end{proof}

To prove Theorem \ref{corollaryrandomgraph} for a small corank $k$ for the case $p=\frac{1}{2}$, we need the following technical preparations to address the above bullet point (4). We use ideas from \cite{tao2007singularity} and \cite{jain2021singularity}, and we use the notations in \cite{campos2025singularity}, Appendix C. 

We define a function that measures anticoncentration for a fixed vector: for $v\in\mathbb{R}^n$ define
$$
\rho(v)=\max_{w\in\mathbb{R}}\mathbb{P}(\sum_{i=1}^n\delta_iv_i=w)
$$ where $\delta_1,\cdots,\delta_n$ are i.i.d. random variables of distribution $\operatorname{Ber}(\frac{1}{2})$.

We then define the following function for any $\gamma>0$:

$$
q_n(\gamma):=\max_{w\in\mathbb{R}^n}:\mathbb{P}_G(\exists v\in\mathbb{R}^n\setminus\{0\}:Gv=w,\rho(v)\geq\gamma).\
$$
Then we will prove the following analogue of \cite{campos2025singularity}, Lemma 9.1:

\begin{lemma}\label{lemma6.1graphmodel}
    Let $G$ be as in Theorem \ref{corollaryrandomgraph} with $p=\frac{1}{2}$. Then for all $\gamma>0$ we have
    $$
\mathbb{P}(\det G=0)\leq 16n\sum_{m=n}^{2n-2}(\gamma^{1/8}+\frac{q_{m-1}(\gamma)}{\gamma}).
    $$
\end{lemma}

Assuming the validity of Lemma \ref{lemma6.1graphmodel}, we can complete the proof of Theorem \ref{corollaryrandomgraph}. 
\begin{proof}[\proofname\ of Theorem \ref{corollaryrandomgraph} for small corank $k$ and $p=\frac{1}{2}$] The proof is similar to \cite{campos2025singularity}, Theorem 1.1. but without having $\|G\|\leq 4\sqrt{n}$. Instead, we use that $\|G\|_{HS}\leq 2n$ almost surely, and using the results of Section \ref{compressiblevectors3} and \ref{thresholdcardinality}
we can check that we can find $\rho_0,\gamma_0\in(0,1)$ and $C_{\rho_0,\gamma_0}>0$ such that for any $w\in\mathbb{S}^{n-1}$,
$$
\mathbb{P}\left(\exists v\in\mathbb{S}^{n-1},D_{\rho_0,\gamma_0}(v)\leq e^{C_{\rho_0,\gamma_0}n}: Gv\in\{tw:t\in\mathbb{R}\}\right)=e^{-\Omega(n)}.
$$Indeed, we shall take a $e^{-cn}$-net for $t\in[-2n,2n]$ and use computations in Section \ref{compressiblevectors3}  to rule out compressible vectors and use results in Section \ref{thresholdcardinality}
to rule out incompressible vectors with a small LCD. Note that while Section \ref{thresholdcardinality} studies the structure of vectors in $\ker G$, the use of the Fourier replacement lemma, Lemma \ref{lemma4.145}, yields the same result for the structure of vectors $v$ with $Gv=tw$, and the computations in Section \ref{implementation2.2} actually yield that those $v\in\mathbb{S}^{n-1}$ satisfying $\|Gv-tw\|_2\leq e^{-cn}$ for certain $c>0$ should have $D_{\rho_0,\gamma_0}(v)\geq e^{Cn}$ with probability $1-e^{-\Omega(n)}$. Then $D_{\rho_0,\gamma_0}(v)\geq e^{Cn}$ implies that $\rho(v)\leq e^{-C'n}$ for some $C'>0$ by the inverse Littlewood-Offord inequality, theorem \ref{littlewoodincomp}. This implies an exponentially small upper bound for $q_{m-1}(\gamma)$ for $\gamma\geq e^{-C'n}$, which, taken into Lemma \ref{lemma6.1graphmodel}, completes the proof. We refer to \cite{campos2025singularity}, Section 9 for some computational details.
\end{proof}

\begin{remark}\label{remakrp1/2} For small corank, in particular $k=1$, we present the proof of Theorem \ref{corollaryrandomgraph} only for the special case $p=\frac{1}{2}$ because we can use Lemma \ref{lemma6.1graphmodel} to bound the singularity probability for uniform $\{0,1\}$ entries. For a general $p\in(0,1)$, Lemma \ref{lemma6.1graphmodel} does not apply, but we can still use the argument in the spirit of \cite{campos2024least} which originates from \cite{vershynin2014invertibility} to derive a least singular value estimate for $G$, which implies the singularity probability as a corollary. As the latter argument is too lengthy, we omit its presentation in this note. 
\end{remark}

Then we complete the proof of Lemma \ref{lemma6.1graphmodel}. Let $G_n$ denote the adjacency matrix of the Erdős–Rényi graph on $n$ vertices, where we take $p=\frac{1}{2}$. Then we have the following

\begin{lemma}\label{lemma6.222}
    For any $0\leq k\leq n-1$, 
    $$\mathbb{P}(\operatorname{Rank}(G_n)=k)\leq 2\cdot \mathbb{P}(\operatorname{Rank}(G_{2n-k-1})=2n-k-2).
    $$
\end{lemma}

The proof of this lemma relies on the following fact:
\begin{fact}\label{fact6.3}
    Let $V$ be a subspace of $\mathbb{R}^n$ of dimension at most $k$. Then 
    $
|V\cap\{0,1\}^n|\leq 2^k.
    $
\end{fact}
The proof of Fact \ref{fact6.3} follows the same lines as in \cite{campos2021singularity}, Observation A.5, which was due to an observation of Odlyzko \cite{odlyzko1988subspaces}. Although these results are stated for the intersection $|V\cap\{-1,1\}^n|$, the same argument applies to the intersection $|V\cap\{0,1\}^n|$ here.

\begin{proof}[\proofname\ of Lemma \ref{lemma6.222}] We only need to adapt the proof of \cite{campos2021singularity}, Lemma A.4, and work through the proof using Fact \ref{fact6.3}. Although the proof in \cite{campos2021singularity}, Lemma A.4 was stated for random matrices over a finite field $\mathbb{F}_p$, taking $p$ large enough ($p\gg n^n$) would yield the same estimate for the rank over $\mathbb{R}$. The adaptations are straightforward and thus omitted.
\end{proof}

As a corollary of Lemma \ref{lemma6.222}, we have the following estimate:

\begin{lemma}\label{lemma16.13}
For every $n\in\mathbb{N}$,
$$
\mathbb{P}(\det G_n=0)\leq 4n\sum_{m=n}^{2n-2}\mathbb{P}\left(\operatorname{Rank}(G_m)=m-1,\operatorname{Rank}(G_{m-1})\in \{m-2,m-1\}\right).
$$
\end{lemma}
The proof of Lemma \ref{lemma16.13} can be adapted from \cite{campos2021singularity}, Lemma A.1, using Lemma \ref{lemma6.222}. 

Then we consider two separate cases. The following lemma can be proven via exactly the same argument as in \cite{campos2025singularity}, Lemma C.3:

\begin{lemma}\label{lemma6.523} We have
$$
\mathbb{P}(\operatorname{Rank}(G_n)=n-1,\operatorname{Rank}(G_{n-1})=n-2)\leq q_{n-1}(\gamma)+\gamma.
$$
\end{lemma}

We also have the following analogue of \cite{campos2025singularity}, Lemma C.4:
\begin{lemma}\label{lemma6.78923}
    We have, for each $t\in[n-2]$, 
    $$
\mathbb{P}(\operatorname{Rank}(G_n)=n-1,\operatorname{Rank}(G_{n-1})=n-1)\leq 3^t q_{n-1}(\gamma)+(2^t\gamma+2^{-t})^{1/4}.
    $$
\end{lemma}

\begin{proof}We essentially follow the lines of \cite{campos2025singularity}, Lemma C.4. We take a nontrivial partition of $[n-1]$ by $I\cup J$ and denote by $\mathcal{A}:=\{G_{n-1}:\det G_{n-1}\neq 0\}.$ Then by a well-known decoupling argument (see \cite{ferber2019singularity} or \cite{campos2021singularity}), 
$$\begin{aligned}
\mathbb{P}&(\operatorname{Rank}(G_n)=\operatorname{Rank}(G_{n-1})=n-1)\\&\leq\mathbb{E}_{G_{n-1}}\mathbb{P}(\langle G_{n-1}^{-1}(X-X')_I,(X-X')_J\rangle=0\mid G_{n-1})^{1/4}\cdot \mathbf{1}(G_{n-1}\in\mathcal{A}),
\end{aligned}$$
whee $X,X'$ are two independent random vectors uniformly chosen over $\{0,1\}^{n-1}$.
We write $w=X-X'$ and when $G_{n-1}\in\mathcal{A}$ write $x=G_{n-1}^{-1}w_I$. Then for $G_{n-1}\in\mathcal{A}$, we have $$
\mathbb{P}(x=0\mid G_{n-1})=2^{-|I|}.$$
    Now we set 
    $$
W(I)=\{v\in\{-1,0,1\}^{n-1}:v_i=0\text{ for all }i\notin I\}.
    $$ We then define
    $$
U_\gamma^I:=\{G_{n-1}:
\rho(v)\leq\gamma\text{ for all } v\in\mathbb{R}^{n-1}\setminus\{0\}\text{ such that }G_{n-1}v\in W(I)\}.
    $$
    Then we can check that 
    $$
\mathbb{P}(G_{n-1}\notin U_\gamma^I)\leq 3^{|I|}q_{n-1}(\gamma).
    $$
Finally, for $G_{n-1}\in\mathcal{A}\cap U_\gamma^I$, we have 
$$
\mathbb{P}_x(\langle x,w_J\rangle=0\text{ and }x\neq 0|G_{n-1})\leq 2^{|I|}\gamma.
$$ where we use $\rho(x)\leq\gamma$ by definition and $\rho(x_J)\leq 2^{|I|}\rho(x)$ by \cite{ferber2019singularity}, Lemma 2.9.
    Combining all the above computations, and taking $I=[t]$, completes the proof.
\end{proof}

\begin{proof}[\proofname\ of Lemma \ref{lemma6.1graphmodel}]
    This follows from combining Lemma \ref{lemma16.13}, \ref{lemma6.523} and \ref{lemma6.78923}, simialr to the proof of \cite{campos2022singularity}, Lemma 9.1. The details are omitted.
\end{proof}

\section{Proof of technical results}\label{proofinverse}
This section contains the proof of Theorem \ref{compressibletheorem12.234567}, Lemma \ref{lemma4.145} and other results.

\subsection{A new conditioned inverse Littlewood-Offord theorem}\label{technicaldeduction}
In this section, we prove our inverse Littlewood-Offord theorem, Theorem \ref{compressibletheorem12.234567}, that is used in the proof of no-gaps delocalization in Section \ref{nogapsdev}. 
We first introduce some notations.
For an $2d\times k$ matrix $W$ and an $\ell$-tuple of vectors $Y_1,\cdots,Y_\ell\in\mathbb{R}^d$, we define an augmented matrix $W_\mathbf{Y}$ as follows, where we abbreviate $\mathbf{Y}:=(Y_1,\cdots,Y_\ell)$: 
\begin{equation}\label{WbfY}
W_\mathbf{Y}=\begin{bmatrix}
    W, \begin{bmatrix}
        Y_1\\\mathbf{0}_d
    \end{bmatrix},\begin{bmatrix}
        \mathbf{0}_d\\Y_1
    \end{bmatrix}
   ,\cdots,\begin{bmatrix}
        Y_\ell\\\mathbf{0}_d
    \end{bmatrix},\begin{bmatrix}
        \mathbf{0}_d\\Y_\ell
    \end{bmatrix}
\end{bmatrix}.
\end{equation}

In \cite{han2025repeated}, Section 12 the following version of essential LCD was used:
\begin{Definition}\label{definitionlcd}
Fix some $L_0\geq 1$ and $\alpha_0\in(0,1)$.
Consider nonzero vectors $Y_1,\cdots,Y_\ell\in\mathbb{R}^d$ and a tuple of positive constants $\mathbf{t}=(t_1,\cdots,t_\ell)$ with $t_i>0$. We define the essential least common denominator (LCD) of the vector pair $\mathbf{Y}$ associated with these parameters:
$$
\operatorname{LCD}_{L_0,\alpha_0}^\mathbf{t}(\mathbf{Y}):=\inf\left\{\|\theta\|_2:\theta\in \mathbb{R}^\ell,\| \sum_{i=1}^\ell \theta_iY_i\|_\mathbb{T}\leq L_0\sqrt{\log_+\frac{\alpha_0\sqrt{\sum_{i=1}^\ell \theta_i^2/t_i^2}}{L_0}}\right\},
$$
where $\log_+(x)=\max(\log x,0)$.\end{Definition} 
This definition of essential LCD is equivalent (up to a change of constants) to Definition \ref{definition4.33.4} for almost orthogonal vectors $Y_1,\cdots,Y_\ell$:

\begin{fact}\label{notionsoflcd}
    Assume that $Y_1,\cdots,Y_\ell\in\mathbb{R}^n$ are $\frac{1}{4}$-almost orthogonal and that $\|Y_i\|_2= 1/t_i$ for each $i\in[l]$. Let $\mathbf{Y}$ be the $d\times l$ matrix with columns $Y_1,\cdots,Y_l$, then we have 
    $$D_{L_0,4\alpha_0}(\mathbf{Y})\leq 
\operatorname{LCD}_{L_0,\alpha_0}^\mathbf{t}(\mathbf{Y})\leq D_{L_0,\alpha_0/4}(\mathbf{Y}).
    $$
\end{fact}
\begin{proof}
    By Definition \ref{definitionalmostorthogonals}, we have $\frac{3}{4}\sqrt{\sum_{i=1}^l\theta_i^2/t_i^2}\leq \|\sum_{i=1}^l\theta_iY_i\|_2\leq \frac{5}{4}\sqrt{\sum_{i=1}^l\theta_i^2/t_i^2}$
\end{proof}

We shall use the following theorem, which is Theorem 12.17 proven in \cite{han2025repeated}, as the main technical input:

\begin{theorem}\label{twolittlewoodofford}
    For given $0<\nu\leq 2^{-15}$, $c_0\leq 2^{-35}B^{-4}\nu$, $d\in\mathbb{N}$ and $\alpha\in(0,1)$. Consider an $\ell$-tuple of vectors $\mathbf{Y}=(Y_1,\cdots,Y_\ell)\in\mathbb{R}^d$ and a tuple of real numbers $\mathbf{t}=(t_1,\cdots,t_\ell)\in\mathbb{R}_+^\ell.$ Let $W$ be some $2d\times k$ matrix satisfying $\|W\|\leq 2$, $\|W\|_{HS}\geq\sqrt{k}/2$. Let $\tau\sim\Phi_\nu(2d;\xi)$. 

    Then if $\operatorname{LCD}_{L,\alpha}^\mathbf{t}(\mathbf{Y})\geq 256B^2\sqrt{\ell}$ with $L=(\frac{8}{\sqrt{\nu p}}+\frac{256B^2}{\sqrt{c_0}})\sqrt{\ell}$, then we must have
    \begin{equation}
        \mathcal{L}(W_\mathbf{Y}^T\tau,c_0^{1/2}\sqrt{k+2\ell})\leq (\frac{R}{\alpha})^{2\ell}(\prod_{i=1}^\ell t_i^2)\exp(-c_0k),
    \end{equation}where we take $R=2^{38}B^2\nu^{-1/2}c_0^{-2}(\frac{8}{\sqrt{\nu p}}+\frac{256B^2}{\sqrt{c_0}}).$
\end{theorem}

To upgrade Theorem \ref{twolittlewoodofford} to a conditioned version on the singular values of submatrices,
we need the following technical preparations.

Let $d_e\in\mathbb{N}$ and let $\mathcal{U}_{d_e,k}$ be the set of all $d_e\times k$ matrices with orthogonal columns (by definition $d_e\geq c_0^4n/8$). Then we have a net for $\mathcal{U}_{d_e,k}$ as a subset of $\mathbb{R}^{[d_e]\times[k]}$:

\begin{lemma}\label{netofmatrices}(\cite{campos2025singularity}, Lemma 6.4) Fix $k\leq d_e$ and $\delta\in(0,\frac{1}{2})$. There exists a net $\mathcal{W}=\mathcal{W}_{d_e,k}\subset\mathbb{R}^{[d_e]\times[k]}$ satisfying $|\mathcal{W}|\leq (64/\delta)^{d_ek}$, such that for any $U\in\mathcal{U}_{d_e,k}$, any $r\in\mathbb{N}$ and any given $r\times d_e$ matrix $A$ we can find $W\in\mathcal{W}$ satisfying
\begin{enumerate}
    \item $\|A(W-U)\|_{HS}\leq\delta(k/d_e)^{1/2}\|A\|_{HS}$,
    \item $\|W-U\|_{HS}\leq\delta\sqrt{k}$, and also
    \item $\|W-U\|_{op}\leq8\delta$.
\end{enumerate}
\end{lemma}

We shall also use the following simple fact:
\begin{fact}\label{factofcorank}
    Take $d,n\in\mathbb{N}$ with $2d\leq n$, and $d_e,n_e\in\mathbb{N}$ with $2d_e\leq n_e$ and $d_e\leq d,n_e\leq n$. 
    
    Consider $H$ an $(n-d)\times d$ matrix, let $I\subset[n-d]$ have cardinality $|I|=n_e$ and let $J=[1,d_e]\cap\mathbb{Z}$. Then if $\sigma_{d_e-k+1}(H_{I\times J})\leq x$ then we can find $k$ orthogonal unit vectors $v_1,\cdots,v_k\in\mathbb{R}^{d_e}$ with $\|H_{I\times J}v_i\|_2\leq x$ for each $i\in[k]$. That is, we can find $W\in\mathcal{U}_{d_e,k}$ satisfying $\|H_{I\times J}W\|_{HS}\leq x\sqrt{k}$.

    Moreover, any $W\in\mathcal{U}_{d_e,k}$ can be identified with an element of $\mathcal{U}_{d,k}$ by appending $d-d_e$ all-zero rows to the matrix $W$. Under this identification, we have $\|H_{I\times [d]}W\|_{HS}\leq x\sqrt{k}$.
\end{fact}
This fact follows immediately from the definition of least singular values.

We also need a control on $\|H\|_{HS}$:
\begin{fact}\label{largedeviationhilbertschmidt} Let $H$ be a random matrix of size $(n-d)\times 2d$ with i.i.d. rows of distribution $\Phi_\nu(2d;\zeta)$. Recall that $J=[d_e]$ and $n_e\geq c_0^2n/2,d_e\geq c_0^4n/8$, then we have
\begin{equation}
    \mathbb{P}(\|H_{I\times J}\|_{HS}\geq 2\sqrt{d_en_e}\text{ for some }I\subset[n-d]:|I|=n_e)\leq2\exp(-2^{-30}B^{-4}c_0^6n^2).
\end{equation}
\end{fact}
\begin{proof}(Sketch)
   We first consider any fixed interval $I$ and derive the estimate for the Hilbert-Schmidt norm of $H_{I\times J}$, and then take a union bound for all $I\subset[n-d]$.
\end{proof}

Finally, we need a version of the tensorization lemma:
\begin{lemma}(\cite{han2025repeated}, Lemma 12.21)\label{tensorization2}
    Fix integers $d<n$, $m\geq 0$ and $k\geq 0$. Consider $W$ a $2d\times(k+2\ell)$ matrix and $H$ a $m\times 2d$ random matrix with i.i.d. rows. Consider $\tau\in\mathbb{R}^{2d}$ a random vector having the same distribution as a row of $H$. Then for $\beta\in(0,\frac{1}{8})$, we have
    $$
\mathbb{P}_H(\|HW\|_{HS}\leq\beta^2\sqrt{(k+2\ell)m})\leq (32e^{2\beta^2(k+2\ell)}\mathcal{L}(W^T\tau,\beta\sqrt{k+2\ell}))^{m}.
    $$
\end{lemma}
In \cite{han2025repeated}, Lemma 12.21, the proof was written for $m=n-d$, but the proof clearly works for any $m\in\mathbb{N}_+$. Now we can complete the proof of Theorem \ref{compressibletheorem12.234567}. We state the inequality in the following form:

\begin{theorem}\label{theorem12.234567}
    For given $n\in\mathbb{N}$, $0<c_0\leq 2^{-50}B^{-4}$, let $d\leq n$ and fix $\alpha\in(0,1)$. Let $\mathbf{t}=(t_1,\cdots,t_\ell)$ be a tuple of positive real numbers. Consider an $\ell$-tuple of vectors $X_1,\cdots,X_\ell\in\mathbb{R}^d$ that satisfy $\operatorname{LCD}_{L,\alpha}^\mathbf{t}(\frac{c_0}{32\sqrt{n}}\mathbf{X})\geq 256B^2\sqrt{\ell}$ (where we denote by $\frac{c_0}{32\sqrt{n}}\mathbf{X}:=(\frac{c_0}{32\sqrt{n}}X_1,\cdots,\frac{c_0}{32\sqrt{n}}X_\ell)$) and where we take $L=(\frac{8}{\sqrt{\nu p}}+\frac{256B^2}{\sqrt{c_0}})\sqrt{\ell}$.

    Let $H$ be an $(n-d)\times 2d$ random matrix with i.i.d. rows of distribution $\Phi_\nu(2d;\zeta)$ with $\nu=2^{-15}$ (and we recall $p\geq\frac{1}{2^7B^4}$ from \eqref{whatdoesbhave?}.

    Let $n_e,d_e\in\mathbb{N}$ satisfy that $c_0^2n_e/4\leq d_e\leq c_0^2n_e$ and $n_e\geq c_0^2n/2$. Let $J$ be the interval $[1,d_e]\cap\mathbb{N}$.
    
    Then whenever $k\leq 2^{-32}B^{-4}c_0^3n$ and $\prod_{i=1}^\ell Rt_i\geq\exp(-2^{-32}B^{-4}c_0^6n)$, we have:
    \begin{equation}\begin{aligned}\label{alignedconditionals}
    \mathbb{P}_H&(\sigma_{d_e-k+1}((H_1)_{I\times J})\leq c_0^22^{-4}\sqrt{n}\text{ for some interval } I\subset[n-d] \text{ with  } |I|\geq n_e, \\&
        \text{ and }\|H_1X_i\|_2,\|H_2X_i\|_2\leq c_0n\quad\text{for all }i\in[\ell])\\&\leq 2^n e^{-c_0^3nk/12}(\prod_{i=1}^\ell\frac{R t_i}{\alpha})^{2n-2d},
    \end{aligned}\end{equation}
    where $H_1=H_{[n-d]\times[d]}$, $H_2=H_{[n-d]\times[d+1,2d]}$ and $R=2^{44}B^2c_0^{-3}(\frac{8}{\sqrt{\nu p}}+\frac{256B^2}{\sqrt{c_0}})$.
\end{theorem}

\begin{proof} By Cauchy interlacing theorem, $\sigma_{d_e-k+1}((H_1)_{I\times J})\geq \sigma_{d_e-k+1}((H_1)_{I'\times J})$ for $I'\subset I$, so we only need to consider intervals $I\subset[n-d]$ with $|I|=n_e$.
Upon taking a union bound over all possible $I$, we shall prove the following estimate for every fixed $I\subset[n-d]$ with $|I|=n_e$ (where we notice that $(H_1)_{I\times J}=H_{I\times J}$):
\begin{equation}\begin{aligned}
    \mathbb{P}_H&(\sigma_{d_e-k+1}(H_{I\times J})\leq c_0^22^{-4}\sqrt{n}
        \text{ and }\|H_1X_i\|_2,\|H_2X_i\|_2\leq c_0n\quad\text{for all }i\in[\ell])\\&\leq  e^{-c_0^3nk/12}(\prod_{i=1}^\ell\frac{R t_i}{\alpha})^{2n-2d}.
    \end{aligned}\end{equation}

Let $Y_i:=\frac{c_0}{32\sqrt{n}}\cdot X_i$ for each $i\in[\ell]$. We use the abbreviation $H_I=H_{I\times[2d]}$ and $H_{I^c}=H_{([n-d]\setminus I)\times[2d]}$. Then by Fact \ref{factofcorank}, 
    $$\begin{aligned}
&\mathbb{P}(\sigma_{d_e-k+1}(H_{I\times J})\leq c_0^22^{-4}\sqrt{n}\text{ and }\|H_1X_i\|_2,\|H_2X_i\|_2\leq c_0n\text{ for each }i\in[\ell])\\&\leq
\mathbb{P}(\exists U\in\mathcal{U}_{d_e,k}\subset\mathcal{U}_{2d,k}:\|H_IU_\mathbf{Y}\|_{HS}\leq c_0\sqrt{|I|(k+2\ell)}/8)\cdot\mathbb{P}(\|H_{I^c}\mathbf{0}_\mathbf{Y}\|_{HS}\leq c_0\sqrt{|I^c|k}/8),
    \end{aligned}$$ where we use the independence of $H_I$ and $H_{I^c}$, combined with our assumption $|I|= n_e$ and we identify $U\in\mathcal{U}_{d_e,k}$ with an element in $\mathcal{U}_{2d,k}$ by appending $2d-d_e$ zero rows. The symbol $\mathbf{0}_\mathbf{Y}$ stands for the matrix $W_\mathbf{Y}$ in \eqref{WbfY}, where we take $W=\mathbf{0}$.
    
    We now take $\delta:=c_0^3/16$ and let $\mathcal{W}_{d_e,k}$ be the net given in Lemma \ref{netofmatrices}.
For the given matrix $H$, if $\|H_{I\times J}\|_{HS}\leq 2\sqrt{n_ed_e}$ and if there exists $U\in\mathcal{U}_{d_e,k}$ with $\|H_IU_\mathbf{Y}\|_{HS}\leq c_0\sqrt{|I|(k+2\ell)}/8$, then Lemma \ref{netofmatrices} implies the existence of $W\in\mathcal{W}_{d_e,k}\subset\mathcal{U}_{d_e,k}$ such that 
$$
\|H_IW_\mathbf{Y}\|_{HS}\leq \|H_I(W_\mathbf{Y}-U_\mathbf{Y})\|_{HS}+\|H_IU_\mathbf{Y}\|_{HS}\leq \delta(k/d_e)^{1/2}\|H_{I\times J}\|_{HS}+c_0\sqrt{|I|(k+2\ell)}/8
$$ and this is bounded from above by $c_0\sqrt{|I|(k+2\ell)}/4$. Thus we have
\begin{equation}\label{bigsums}\begin{aligned}
&\mathbb{P}_H(\exists U\in\mathcal{U}_{2d,k}:\|H_IU_\mathbf{Y}\|_{HS}\leq c_0\sqrt{|I|(k+2\ell)}/8)\\&\leq \mathbb{P}_H(\exists W\in\mathcal{W}:\|H_IW_\mathbf{Y}\|_{HS}\leq c_0\sqrt{|I|(k+2\ell)}/4)\\&\quad+\mathbb{P}_H(\|H_{I\times J}\|_{HS}\geq 2\sqrt{n_ed_e}\text{ for some }I\subset[n-d]:|I|=n_e)\\&\leq\sum_{W\in\mathcal{W}}\mathbb{P}_H(\|H_IW_\mathbf{Y}\|_2\leq c_0\sqrt{|I|(k+2\ell)}/4)+2\exp(-2^{-30}B^{-4}c_0^6n^2)
\end{aligned}\end{equation} where in the last step we used Fact \ref{largedeviationhilbertschmidt}.
    We bound the cardinality of the net $\mathcal{W}_{d_e,k}$ by
    $$
|\mathcal{W}_{d_e,k}|\leq (64/\delta)^{2d_ek}\leq \exp(64d_ek\log c_0^{-1})\leq\exp(c_0^3kn/6)
    $$ where we use the assumption $d_e\leq c_0^2n_e\leq c_0^4n$. Then \begin{equation}
        \begin{aligned}&\sum_{W\in\mathcal{W}}\mathbb{P}_H(\|H_IW_\mathbf{Y}\|_2\leq c_0\sqrt{|I|(k+2\ell)}/4)\\&\leq\exp(c_0^3kn/6)\max_{W\in\mathcal{W}}\mathbb{P}_H(\|H_IW_\mathbf{Y}\|_2\leq c_0\sqrt{|I|(k+2\ell)}/4).
    \end{aligned}\end{equation}
    Now for this $W\in\mathcal{W}$, we apply Lemma \ref{tensorization2} where we set $\beta:=\sqrt{c_0/3}$ and get
    \begin{equation}\label{tensorizationmid}
        \mathbb{P}_H(\|H_IW_\mathbf{Y}\|_2\leq c_0\sqrt{|I|(k+2\ell)}/4)\leq(32e^{2c_0(k+2\ell)/3}\mathcal{L}(W_\mathbf{Y}^T\tau,c_0^{1/2}\sqrt{k+2\ell}))^{|I|}.
    \end{equation}

Now we prepare for the application of Theorem \ref{twolittlewoodofford}. Recall that $\nu=2^{-15}$. By assumption, $\operatorname{LCD}_{L,\alpha}^\mathbf{t}(\frac{c_0}{32\sqrt{n}}\mathbf{X})=\operatorname{LCD}_{L,\alpha}^\mathbf{t}(\mathbf{Y})\geq 256B^2\sqrt{\ell}$. Also, for each $W\in\mathcal{W}$, we have that $\|W\|_{op}\leq 2$ and $\|W\|_{HS}\geq\sqrt{k}/2$. Then applying Theorem \ref{twolittlewoodofford}, we get that
$$
\mathcal{L}(W_\mathbf{Y}^T\tau,c_0^{1/2}\sqrt{k+2\ell})\leq (\frac{R}{\alpha})^{2\ell}(\prod_{i=1}^\ell t_i^2)\exp(-c_0k).
$$

Substituting this into \eqref{tensorizationmid}, we deduce that 
$$
\max_{W\in\mathcal{W}}\mathbb{P}_H(\|H_IW_\mathbf{Y}\|_2\leq c_0\sqrt{|I|(k+2\ell)}/4)\leq\frac{1}{2}(\prod_{i=1}^\ell \frac{2R\cdot t_i}{\alpha})^{2|I|}e^{-c_0k|I|/3}.
$$

For the probability involving  $H_{I^c}$, we apply Theorem \ref{twolittlewoodofford} with the choice $k=0$ and get 
$$
\mathbb{P}_H(\|H_{I^c}\mathbf{0}_\mathbf{Y}\|_2\leq c_0\sqrt{|I|(k+2\ell)}/4)\leq\frac{1}{2}(\prod_{i=1}^\ell \frac{2R\cdot t_i}{\alpha})^{2(n-d-|I|)}.
$$

Combining all the established bounds, we conclude that for this fixed interval $I$,
$$\begin{aligned}&
\mathbb{P}(\sigma_{d_e-k+1}(H_{I\times J})\leq c_0^2\sqrt{n}/16\text{ and }\|H_1X_i\|_2,\|H_2X_i\|_2\leq c_0n\text{ for each }i\in[\ell])\\&\leq (\prod_{i=1}^\ell\frac{R\cdot t_i}{\alpha})^{2n-2d}e^{-c_0^3nk/12},\end{aligned}
$$ where the restrictions $k\leq 2^{-32}B^{-4}c_0^3n$ and $\prod_{i=1}^\ell Rt_i\geq\exp(-2^{-32}B^{-4}c_0^6n)$ in the statement of Theorem \ref{theorem12.234567} are used to guarantee that the value of our final estimate is larger than $\exp(-2^{-30}B^{-4}c_0^6n^2)$
(for a possibly different value of $R>0$ changing from line to line) and thus complete the proof.
\end{proof}

Having proven the fixed interval estimates in this section, we can immediately complete the proof of Theorem \ref{compressibletheorem12.234567}:

\begin{proof}[\proofname\ of Theorem \ref{compressibletheorem12.234567}]

We explain how this theorem is deduced from the fixed interval estimates in this section. The main input is Theorem \ref{theorem12.234567}, which yields the conditioned inverse Littlewood-Offord estimates for each single interval $I$ after we pass to the two-block matrix model and impose the upper bounds on $\|H_1X_i\|$, $\|H_2X_i\|$.

More concretely, we write $H_1$ for the matrix appearing in the estimate \eqref{finalwehaveestimates} and write $H_2$ an independent copy of $H_1$. Then by independence, the probability stated on the left hand side of \eqref{compressibletheorem12.234567} is upper bounded by the square root of the corresponding probability in the two block formulation in \eqref{alignedconditionals}.
    Then we use Fact \ref{notionsoflcd} and the $\frac{1}{4}$-almost orthogonality assumption of $X_1,\cdots,X_\ell$ to convert the version of essential LCD used in Theorem \ref{theorem12.234567} into the version of essential LCD used in Theorem \ref{compressibletheorem12.234567}. This yields the estimate stated in Theorem \ref{compressibletheorem12.234567}.
\end{proof}

We also sketch the proof for Theorem \ref{compressibletheorem12.234567new}:

\begin{proof}[\proofname\ of Theorem \ref{compressibletheorem12.234567new}(Sketch)] 
  The argument follows the general scheme in the proof of Theorem \ref{compressibletheorem12.234567} but is technically simpler since we work with the full matrix $H$ without restriction to a subinterval. 
    The input is again the conditional inverse Littlewood-Offord estimates in this section, and we only need to transfer that estimate to the present formulation. 
    The only substantial bookkeeping is a change of parameters: the quantities $t_i$ in the present Theorem \ref{compressibletheorem12.234567new} correspond to the parameters $\sqrt{n}/t_i$ in the context of  Theorem \ref{theorem12.234567}. After this substitution, the LCD assumption and resulting small ball bounds have exactly the same form here. After this conversion, the proof is identical to that of Theorem \ref{compressibletheorem12.234567}, and thus we omit the routine repetition of these steps.
\end{proof}

\subsection{The Fourier replacement step}\label{prooffourierreplace}
In this section we prove Lemma \ref{lemma4.145}.

Let $A\in\operatorname{Sym}_n(\zeta)$, $l\in\mathbb{N}$ and $v_1,\cdots,v_l\in\mathbb{R}^n$. Define $\hat{A}=\operatorname{diag}(A,A,\cdots,A)$ as a $nl\times nl$ block diagonal matrix of $l$ diagonal blocks, each block is the same matrix $A$. Define $\hat{v}=(v_1,\cdots,v_l)\in\mathbb{R}^{nl}$ as the concatenation of the $l$ vectors $v_1,\cdots,v_l$. Then the matrix multiplication $\hat{A}\hat{v}$ is well-defined.

Recall that we define $M=\begin{bmatrix}
    \mathbf{0}_{d\times d}& H_1^T\\H_1& \mathbf{0}_{[n-d]\times[n-d]} 
\end{bmatrix}$ where $H_1$ has i.i.d. coordinates of distribution $\tilde{\zeta}Z_\nu$. Then we similarly define $\hat{M}=\operatorname{diag}(M,\cdots,M)$ as a $nl\times nl$ block diagonal matrix of $l$ diagonal blocks with identical blocks given by $M$.

We let $\phi_\zeta$ denote the characteristic function of $\zeta:$ $\phi_\zeta(t)=\mathbb{E}e^{2\pi i t\zeta}$ and similarly we let $\phi_{\tilde{\zeta}Z_\nu}$ denote the characteristic function of $\tilde{\zeta}Z_\nu:$ $\phi_{\tilde{\zeta}Z_\nu}(t)=\mathbb{E}e^{2\pi i t\tilde{\zeta}Z_\nu}$. By \cite{campos2024least}, Lemma V.1 we have that for any $t\in\mathbb{R}$ and $\nu\in(0,\frac{1}{4})$, 
\begin{equation}\label{boundedcharacter}
|\phi_\zeta(t)|\leq \phi_{\tilde{\zeta}Z_\nu}(t).
\end{equation}

For two fixed vectors $v,x\in\mathbb{R}^{nl}$ define the following two characteristic functions 
$$\begin{aligned} \psi_v(x) =&\mathbb{E}e^{\langle 2
\pi i\hat{A}v,x\rangle}=\prod_{k=1}^n\phi_\zeta(\sum_{s=1}^l v_{k+(s-1)n}x_{k+(s-1)n})\\&\cdot\prod_{j<k}\phi_\zeta(\sum_{s=1}^lv_{j+(s-1)n}x_{k+(s-1)n}+x_{j+(s-1)n}v_{k+(s-1)n})\end{aligned}$$
and 
$$\begin{aligned} \chi_v(x) =&\mathbb{E}e^{\langle 2
\pi i\hat{M}v,x\rangle}=\prod_{j=1}^d\prod_{k=d+1}^n\phi_{\tilde{\zeta}Z_\nu}(\sum_{s=1}^lv_{j+(s-1)n}x_{k+(s-1)n}+x_{j+(s-1)n}v_{k+(s-1)n})\end{aligned}.$$

The following simple fact can be checked by covering a $d$-dimensional ball of radius $r$ by $(1+2r/t)^d$ $d$-dimensional balls of radius $t$:
\begin{fact}\label{factreplaces}
    Let $X$ be a $d$-dimensional random variable and $r\geq t$. Then 
    $$
\mathcal{L}(X,t)\leq\mathcal{L}(X,r)\leq(1+2r/t)^d\mathcal{L}(X,t).
    $$
\end{fact}

\begin{fact}\label{charfact726}
    For $v\in\mathbb{R}^{nl}$ and $t\geq 2^{-c_0^5n}$, let $p\in\mathbb{N}_+$ satisfying $c_0^4nl\leq p\leq nl$ be such that for some given $L\geq 1$ we have 
    \begin{equation}\label{smallballprob}
\mathbb{P}(\|\hat{M}v\|_2\leq t\sqrt{nl}) \leq (Lt)^{p}.
\end{equation} Then we have the following estimate whenever $l\leq n$:
$$
\mathbb{E}\exp(-\pi\|\hat{M}v\|_2^2/t^2)\leq (9Lt)^p.
$$ 
\end{fact}

\begin{proof}
Using Fact \ref{factreplaces}, we see that for any $s\geq t$,
$$
\mathbb{P}(\|\hat{M}v\|_2\leq s\sqrt{nl})\leq (Lt)^p(3s/t)^{nl}.
$$

    We upper bound the expectation $\mathbb{E}\exp(-\pi\|\hat{M}v\|_2^2/t^2)$ by 
$$
\mathbb{P}(\|\hat{M}v\|_2\leq t\sqrt{nl})+\sqrt{nl}\int_t^\infty \exp(-\frac{s^2nl}{t^2})\mathbb{P}(\|\hat{M}v\|_2\leq s\sqrt{nl})ds.
$$
By our assumption \eqref{smallballprob}, we can upper bound the second term by
$$
\sqrt{nl}\int_t^\infty \exp(-\frac{s^2nl}{t^2})(Lt)^p (3s/t)^{nl} ds.
$$
Now we change the variables to $u=s/t$ and then the second integral becomes
$$
t^{-1}\sqrt{nl}(Lt)^p\int_1^\infty\exp(-u^2nl)u^{nl}du\leq t^{-1}\sqrt{nl}(Lt)^p 3^{nl}\int_1^\infty \exp(-u^2/2)du\leq (9Lt)^p 
$$ by our assumption that $l\leq n,p\geq c_0^4nl$ and $t\geq 2^{-c_0^5n}$.
\end{proof}

Now we complete the proof of Lemma \ref{lemma4.145}.
\begin{proof}[\proofname\ of Lemma \ref{lemma4.145}]
With the newly introduced notations $\hat{A}$ and $\hat{M}$, we shall prove the following version of Lemma \ref{lemma4.145}: we replace the assumption \eqref{smallballprob}
by the following: for any $v\in\mathbb{R}^{nl}$,
\begin{equation}
\mathbb{P}(\|\hat{M}v\|_2\leq t\sqrt{nl}) \leq (Lt)^{p}.
\end{equation}
Then we only need to prove the following conclusion:
$$\mathcal{L}(\hat{A}v,t\sqrt{nl})\leq (10\exp(2/c_0^4)Lt)^p.$$

    We begin with an application of Markov's inequality and get 
    $$
\mathbb{P}(\|\hat{A}v-w\|_2\leq t\sqrt{nl})\leq\exp(\pi nl/2)\mathbb{E}\exp(-\pi\|\hat{A}v-w\|_2^2/2t^2).
    $$
    By the Fourier inversion formula, we expand the expectation as 
    $$
\mathbb{E}_{\hat{A}}\exp(-\pi\|\hat{A}v-w\|_2^2/2t^2)=\int_{\mathbb{R}^{nl}} e^{-\pi\|\xi\|_2^2}  \cdot e^{-2\pi it^{-1}\langle w,\xi\rangle}
\psi_v(t^{-1}\xi)d\xi.
    $$
    By non-negativity of $\chi_v$ and \eqref{boundedcharacter}, the right hand side is upper bounded by 
    $$
\int_{\mathbb{R}^{nl}} e^{-\pi\|\xi\|_2^2} \chi_v(t^{-1}\xi)d\xi= \mathbb{E}_M\exp(-2\pi\|\hat{M}v\|_2^2/t^2)
    $$ by the Fourier inversion formula. By Fact \ref{charfact726}, the latter term is bounded by $(9Lt)^p$. Finally, we use $\exp(\pi nl/2)\leq\exp(2pl/c_0^4)$ to complete the proof.
\end{proof}

\subsection{Some other proofs}
\label{theotherproofs}
We outline here the proof of Lemma \ref{lemma1186}:

\begin{proof}[\proofname\ of Lemma \ref{lemma1186}]
    Recall that for $\ell\geq 0$, we define the intervals $I_\ell=[-2^\ell N,2^\ell N]\setminus[-2^{\ell-1}N,2^{\ell-1}N]$ and $I_0=[-N,N]$. For a given sequence $(\ell_{d+1},\cdots,\ell_n)\in \mathbb{Z}_{\geq 0}^{n-d}$ we define $B(\ell_{d+1},\cdots,\ell_n)=\prod_{j=d+1}^n I_{\ell_j}$. Then we define the family of boxes via
    $$
\mathcal{F}=\{B(\ell_{d+1},\cdots,\ell_n):\sum_{j:\ell_j> 0} 2^{2\ell_j}\leq 2n\}.
    $$

    To check that $\mathcal{F}$ is the desired family, let $v\in\Lambda_{t,\delta}$, then $X:=\delta^{-1}v\in B_{n-d}(0,t/\delta)\cap\mathbb{Z}^{n-d}$. For this vector $X$, let $\ell_i$ be such that $X_i\in I_{\ell_i}$ for each $i\in[d+1,n]$, so $X\in B(\ell_{d+1},\cdots,\ell_n)$. Then
    $$
\sum_{j:\ell_j>0}2^{2(\ell_j-1)}N^2\leq \sum_{j=d+1}^n X_j^2\leq \delta^{-2}\sum_{j=d+1}^n v_j^2
\leq \delta^{-2}t^{2}\leq N^2n,
    $$ verifying that $B(\ell_{d+1},\cdots,\ell_n)\in\mathcal{F}$.

    We can verify that $|\mathcal{F}|\leq 2^{10n}$ exactly in the same way as the proof of Lemma \ref{howlarageisthebox}.

Finally, for each box $B(\ell_{d+1},\cdots,\ell_n)\in\mathcal{F}$, we check as in Lemma \ref{howlarageisthebox} that
$$
|B(\ell_{d+1},\cdots,\ell_n)|\leq N^{n-d}2^{\sum_{j}\ell_j}\leq N^{n-d}(\frac{\sum_j 2^{2\ell_j}}{n-d})^{n-d}\leq (6N)^{n-d}
$$ where we use the AM-GM inequality for the second inequality and the assumption $d\leq n/2$ for the third inequality. This completes the proof.
\end{proof}

\section*{Funding}
The author is supported by a Simons Foundation Grant (601948, DJ).

\printbibliography

\end{document}